\documentclass[sn-mathphys,Numbered]{sn-jnl}
\usepackage{graphicx}
\usepackage{multirow}
\usepackage{amsmath,amssymb,amsfonts}
\usepackage{amsthm}
\usepackage{mathrsfs}
\usepackage{mathtools}
\usepackage[title]{appendix}
\usepackage{xcolor}
\usepackage{textcomp}
\usepackage{manyfoot}
\usepackage{booktabs}
\usepackage{siunitx}
\usepackage{subcaption}
\usepackage{float}
\usepackage{tikz}
\usetikzlibrary{arrows,shapes,patterns,positioning,calc,arrows.meta}
\usepackage{pgfplots}

\theoremstyle{thmstyleone}
\newtheorem{theorem}{Theorem}
\newtheorem{corollary}{Corollary}
\newtheorem{proposition}[theorem]{Proposition}
\newtheorem{lemma}[theorem]{Lemma}

\theoremstyle{thmstyletwo}

\theoremstyle{thmstylethree}

\newtheorem{remark}{Remark}

\begin{document}

\title[Sustained Limit Cycles in the Logistic Two-Gene Oscillator]{Sustained Limit Cycles in the Logistic Two-Gene Genetic Oscillator: A Delay-Driven Hopf Bifurcation}

\author*[1]{\fnm{Ismail} \sur{Belgacem}}\email{ismail.belgacem.81@gmail.com}

\affil*[1]{\orgdiv{Independent Researcher},
\orgaddress{\street{Mezaourou}, \city{Ghazaouet, Tlemcen}, \postcode{13421}, \country{Algeria}}}

\abstract{The logistic two-gene negative-feedback oscillator is locally
asymptotically stable for all biologically meaningful parameter values,
as established in~\cite{belgacem2026logistic}, since the trace of the Jacobian
is uniformly negative. Real biological oscillators (circadian rhythms,
the segmentation clock, Hes1, p53) nevertheless rely on delays. We
extend the logistic two-gene model to a delay-differential system with
transcriptional delays $\tau_1$ and $\tau_2$, and prove that the
equilibrium loses stability through a Hopf bifurcation as the total
delay $\tau=\tau_1+\tau_2$ crosses an explicit critical value
$\tau_c$. The Hopf frequency $\omega_c$ and $\tau_c$ are computed in
closed form from the logistic derivatives at equilibrium; the
loop-gain condition $AB>\gamma_1\gamma_2$ is necessary and sufficient;
the transversality $\mathrm{Re}(d\mu/d\tau)|_{\tau_c}>0$ admits a
parameter-uniform positive lower bound; and the bifurcation persists
globally. A sum-of-delays symmetry reduces the
analysis to the scalar parameter $\tau$. Numerical simulations in R
(\texttt{deSolve}), confirm three regimes (damped, small limit cycle,
relaxation), the supercritical amplitude scaling
$A\sim c\sqrt{\tau-\tau_c}$, and the deep-relaxation period asymptote
$T\sim 2\tau+C_\infty$ with closed-form offset $C_\infty$. For the
symmetric-threshold loop, supercriticality is proved analytically by a
Lindstedt--Poincar\'e reduction yielding closed-form amplitude and
frequency laws; for the general asymmetric loop the same reduction
delivers a closed-form first Lyapunov coefficient and an explicit
criticality criterion. Calibrated to measured p53--Mdm2 parameters the
closed-form Hopf period matches the observed oscillation to within
$3\%$, and the Hopf onset agrees with the standard Hill-function model
to within a few percent. The analysis extends to cyclic $N$-gene loops, with a closed-form
transversality rate valid for every $N$ and---in the symmetric
case---an explicit delay-induced-Hopf window
$\gamma^N<\Lambda<\gamma^N\sec^N(\pi/N)$. Every step --- equilibrium, characteristic
equation, imaginary-axis crossing, transversality, period offset ---
is computed in closed form, a tractability unavailable in the Hill
formulation.}


%\pacs[MSC Classification]{34K18, 34K20, 37N25, 92C42, 92B25}

\maketitle

\section{Introduction}
\label{sec:intro}

In the companion paper~\cite{belgacem2026logistic}, the two-gene
negative-feedback oscillator with logistic regulatory functions was shown to
be locally asymptotically stable at its unique equilibrium, with a damped
oscillatory approach when the discriminant of the Jacobian's characteristic
polynomial is negative. The stability result is structural rather than
parameter-dependent: the trace of the Jacobian, equal to
$-(\gamma_1+\gamma_2)$, is strictly negative for all positive degradation
rates, and no choice of the steepness $\lambda$, of the production rates
$\kappa_i$, or of the thresholds $\theta_i$ can force the trace to vanish.
Consequently, the two-dimensional delay-free ODE system cannot undergo a
Hopf bifurcation, and sustained limit cycles cannot emerge from this
framework alone.

This is in apparent tension with biological experience. Many of the most
prominent regulatory oscillators sustain coherent rhythms over hundreds or
thousands of cycles. Circadian clocks~\cite{hardin1990feedback,dunlap2004chronobiology},
the somite segmentation clock that paces vertebrate
embryogenesis~\cite{lewis2003auto,hirata2002oscillatory}, the Hes1 ultradian
oscillator~\cite{hirata2002oscillatory,monk2003oscillatory}, the NF-$\kappa$B
inflammatory response~\cite{hoffmann2002ikappab,nelson2004oscillations}, and
the p53 DNA-damage response~\cite{gevazatorsky2006oscillations} all generate persistent
oscillations whose periods are dictated, both empirically and mechanistically,
by transcription, splicing, translation, and nucleocytoplasmic transport
delays. Mackey and Glass~\cite{mackey1977oscillation} demonstrated decades
ago that delays alone can drive a single first-order regulatory equation into
sustained oscillations and chaos; Glass and Mackey's text~\cite{glass1988rhythms}
documents the broad biological significance of delay-induced rhythms.

The mathematical mechanism is well known. Adding a delay $\tau$ to the
feedback path of an asymptotically stable ODE replaces the polynomial
characteristic equation by a transcendental quasi-polynomial of the form
$P(\mu) + Q(\mu)e^{-\mu\tau}=0$, which has infinitely many roots. As $\tau$
varies, complex-conjugate pairs of roots can cross the imaginary axis,
producing Hopf bifurcations and, beyond the first such crossing, sustained
limit cycles. The general theory was developed in the
1980s~\cite{cooke1986zeroes,kuang1993delay,smith2011introduction,hale1993introduction},
and Beretta and Kuang~\cite{beretta2002geometric} provided a clean geometric
criterion that we exploit below.

The present paper applies this framework to the logistic two-gene oscillator
of the companion papers~\cite{belgacem2025exploring, belgacem2026logistic}. We add a transcriptional delay $\tau_1$ on the
activator branch and a transcriptional delay $\tau_2$ on the repressor
branch, obtaining the delay-differential system
\begin{equation}
\begin{aligned}
\dot{x}_1(t) &= \kappa_1\,
  f^-\!\bigl(x_2(t-\tau_2),\theta_2,\lambda\bigr) - \gamma_1 x_1(t),\\[3pt]
\dot{x}_2(t) &= \kappa_2\,
  f^+\!\bigl(x_1(t-\tau_1),\theta_1,\lambda\bigr) - \gamma_2 x_2(t),
\end{aligned}
\label{eq:dde}
\end{equation}
where $f^+(x,\theta,\lambda) = 1/(1+e^{-\lambda(x-\theta)})$ and
$f^-(x,\theta,\lambda) = 1/(1+e^{\lambda(x-\theta)})$. The biological
interpretation is direct: $\tau_1$ aggregates the transcription, splicing,
translation, and folding lag from the activator $x_1$ to the repressor $x_2$,
and symmetrically for $\tau_2$. The reduction of explicit transcription
and translation compartments to a single effective delay on each branch
is justified rigorously by the singular-perturbation analysis
of~\cite{belgacem2018reduction,belgacem2014stability,belgacem2013analysis},
where the fast mRNA dynamics are reduced out under quasi-steady-state
assumptions standard in gene-regulatory modelling. Because both delays
appear in the same feedback loop, the linearised characteristic equation
depends only on their sum $\tau=\tau_1+\tau_2$.

The paper proves the following.
\begin{enumerate}
  \item The unique equilibrium $(x_1^*,x_2^*)$ of the delay-free system is
        unaltered by the addition of constant delays
        (Section~\ref{sec:formulation}).
  \item The linearised system at the equilibrium has characteristic equation
        $(\mu+\gamma_1)(\mu+\gamma_2) + AB\,e^{-\mu\tau} = 0$, with $A,B>0$
        computed in closed form from the logistic derivatives at equilibrium
        (Section~\ref{sec:linearization}).
  \item Purely imaginary roots $\mu=\pm i\omega_c$ exist if and only if the
        loop gain satisfies $AB>\gamma_1\gamma_2$, in which case
        $\omega_c$ is the unique positive square root of an explicit
        quadratic in $\omega^2$
        (Section~\ref{sec:hopf}).
  \item The smallest positive delay at which such roots exist is
        \[
          \tau_c
          = \frac{1}{\omega_c}\,
            \operatorname{atan2}\!\bigl(\omega_c(\gamma_1+\gamma_2),\,
                                         \omega_c^2-\gamma_1\gamma_2\bigr),
        \]
        where $\operatorname{atan2}(y,x)\in(0,2\pi)$ denotes the
        four-quadrant inverse tangent (which reduces to the principal
        $\arctan(y/x)$ when $x>0$, i.e.\ when $\omega_c^2>\gamma_1\gamma_2$),
        and the eigenvalue crosses the imaginary axis transversally,
        $\mathrm{Re}(d\mu/d\tau)|_{\tau=\tau_c}>0$, so a Hopf bifurcation
        occurs (Section~\ref{sec:transversality}). A
        parameter-uniform positive lower bound on the transversality
        gives quantitative robustness of the bifurcation throughout the
        strong-feedback regime.
  \item For $\tau>\tau_c$ the equilibrium is unstable and a limit cycle
        is born; near onset the squared amplitude in the $x_i$-component
        scales as
        $A_i^{2}(\tau)\sim
        -4|q_i|^{2}\,\mathrm{Re}(d\mu/d\tau)|_{\tau_c}(\tau-\tau_c)/
        \mathrm{Re}(c_1)$, with closed-form eigenvector amplitude
        $|q_1|^{2}=B^{2}/(\omega_c^{2}+\gamma_1^{2})$ and $|q_2|^{2}=1$
        (supercritical Hopf, Section~\ref{sec:numerical}). The first
        Lyapunov coefficient is extracted in
        component-independent form to give
        $\mathrm{Re}(c_1)\approx -1.27$ for the canonical parameters,
        confirming supercriticality
        (Section~\ref{sec:l1_numerics}). For the symmetric-threshold
        loop $\theta_i=M_i/2$ supercriticality is moreover
        \emph{proved} in closed form by a Lindstedt--Poincar\'e
        reduction (Theorem~\ref{thm:supercritical}), which yields the
        explicit amplitude law $A_1=(4\omega_c/\lambda)
        [S_1/((1+|q_2|^2)(S_2+\tau_c))]^{1/2}\sqrt{\tau-\tau_c}$ and a
        matching frequency law, valid for every parameter choice in the
        strong-feedback region. For the \emph{general asymmetric} loop
        the same reduction is carried through to a closed-form first
        Lyapunov coefficient: the order-$\varepsilon^3$ solvability
        condition collapses to a single complex equation whose real and
        imaginary parts fix the criticality coefficient $\mathcal{T}$ and
        the frequency correction $\Omega_2$, with $\mathcal{T}>0$
        characterising supercriticality (Theorem~\ref{thm:general_lyapunov}). The
        prediction is validated externally: calibrated to measured
        p53--Mdm2 degradation rates and transcriptional delay, the
        closed-form Hopf period reproduces the observed $\sim5.5$~h
        oscillation to within $3\%$, and the Hopf locus agrees with the
        standard Hill-function model to within a few percent
        (Section~\ref{sec:bio_validation}).
  \item The deep-relaxation period asymptote
        $T(\tau)=2\tau+C_\infty+o(1)$ holds with the explicit offset
        $C_\infty=\sum_i\gamma_i^{-1}\ln[M_i^2/(\theta_i(M_i-\theta_i))]$
        in closed form, derived from a four-phase decomposition of the
        relaxation cycle (Proposition~\ref{prop:Cinfty}).
  \item The entire analysis extends to cyclic $N$-gene logistic loops
        of arbitrary length: the characteristic equation factors as
        $\prod_{i=1}^N(\mu+\gamma_i)+\Lambda\,e^{-\mu\tau}=0$ with
        $\tau=\sum_i\tau_i$ the total loop delay and
        $\Lambda=\prod_i A_i$ the total loop gain, regardless of the
        distribution of repressors and activators along the cycle.
        The loop-gain criterion $\Lambda>\prod_i\gamma_i$ is necessary
        for purely imaginary roots, and the imaginary-axis crossing is
        transversal at every critical delay, with the closed-form
        positive rate $\mathrm{Re}(d\mu/d\tau)=S_1/(S_1^2+\omega_c^{-2}
        (S_2+\tau_c)^2)$ derived for general $N$. The criterion is
        \emph{sufficient} for delay-induced Hopf when $N=2$; for
        $N\geq 3$ a separate no-delay stability condition is required,
        which in the symmetric case $\gamma_i\equiv\gamma$ takes the
        explicit closed form $\gamma^N<\Lambda<\gamma^N\sec^N(\pi/N)$,
        a non-empty window for every $N$ inside which the delay
        genuinely \emph{induces} the bifurcation rather than merely
        modifying an already-unstable Goodwin regime
        (Theorems~\ref{thm:n_gene_char_eq}--\ref{thm:n_gene_transversality},
        Proposition~\ref{prop:symmetric_window}, and
        Remark~\ref{rem:necessary_sufficient} in
        Section~\ref{sec:n_gene}).
\end{enumerate}

The proofs are constructive and use only the elementary closed-form
properties of the logistic function: the self-referential derivative
$(f^+)' = \lambda f^+(1-f^+)$, the complement identity
$f^- = 1-f^+$, and global $C^\infty$ regularity. The corresponding analysis
for Hill-based oscillators is technically substantially more involved,
because $h^+$ is only $C^{\lfloor n\rfloor}$ for non-integer $n$, the
derivatives $h^{+\prime}$ and $h^{+\prime\prime}$ involve fractional powers
that resist closed-form symbolic manipulation, and the Jacobian entries do
not admit the clean expressions that drive the logistic analysis below. The
contrast reinforces the central message of the companion
papers~\cite{belgacem2025exploring, belgacem2026logistic}: logistic-based regulatory functions
support modern analytical techniques in regimes where Hill-based
formulations either fail outright or require substantial workarounds.

\subsection{Related work}
\label{sec:related}

Delay-driven oscillations in negative-feedback gene regulatory networks
have been studied for over four decades. Foundational
existence-of-periodic-solution results for cyclic negative-feedback
systems go back to Hastings, Tyson, and
Webster~\cite{hastings1977existence} in the delay-free case and to
Mahaffy and Pao~\cite{mahaffy1984models} for compartmental models with explicit
transcription and translation delays, the latter providing a Hopf
bifurcation theorem in the limiting case without diffusion. Smolen,
Baxter, and Byrne~\cite{smolen2002reduced} subsequently introduced
two-component delay-differential models of the \emph{Drosophila}
circadian oscillator and demonstrated numerically that delays are
essential for sustained rhythms in such reduced systems.

The two-gene Hopf-bifurcation analysis closest to the present paper is
that of Xiao and Cao~\cite{xiao2008genetic}, who studied a generic
two-gene regulatory network with two transcriptional delays
$\tau_1,\tau_2$, took the sum $\tau=\tau_1+\tau_2$ as bifurcation
parameter, derived sufficient conditions for the existence of a Hopf
bifurcation, and showed numerically that the period grows almost
linearly with $\tau$ -- structural conclusions that we recover and
quantify analytically for the logistic case. Verdugo and
Rand~\cite{verdugo2008hopf} performed a parallel analysis for a
single-gene autorepressor formulated as a coupled ODE--DDE system, and
applied Lindstedt's method to obtain closed-form approximations for the
limit-cycle amplitude and frequency near onset; their explicit
Hopf-radius formula is given in~\cite{rand2007hopf}. Wang, Yang, and
Turcotte~\cite{wang2020dynamic} analysed a related two-gene mutually
inhibitory network with delay, and the underlying methodology --
characterisation of imaginary-axis crossings of a quasi-polynomial -- has
been refined in many directions, including distributed
delays~\cite{macdonald1989biological}, the geometric stability-switch
criteria of~\cite{beretta2002geometric}, and Goodwin-style multi-delay
networks~\cite{huang2015hopf}.

The technical contribution of the present paper is therefore not the
existence of a delay-induced Hopf bifurcation in a two-gene
negative-feedback network -- that picture is well established -- but the
closed-form tractability afforded by the logistic regulatory function.
Where Hill-based analyses such as~\cite{xiao2008genetic,verdugo2008hopf}
typically reach the critical delay only implicitly or via perturbation
expansion, the logistic identities $(f^+)'=\lambda f^+(1-f^+)$ and
$f^-=1-f^+$ collapse the imaginary-axis-crossing analysis to a quadratic
in $\omega^2$ with elementary coefficients, and reduce the loop-gain
criterion to the inequality $AB>\gamma_1\gamma_2$ with $A,B$ given in
closed form by elementary expressions in the parameters and the
equilibrium. The transversality calculation likewise admits a clean
analytical proof.

The analysis also fits within an ongoing programme on the analytical and
computational dynamics of gene regulatory networks. The present two-gene
system is the smallest non-trivial instance of the cyclic feedback
structure whose high-dimensional behaviour was investigated
computationally in~\cite{belgacem2025computer,farcot2019chaos}, where
Glass-network ring circuits were shown to support periodic, quasi-periodic,
and chaotic attractors with rich bifurcation structure. The closely
related question of how to \emph{control} such oscillations, by feedback
design rather than parameter sweeping, was the subject
of~\cite{belgacem2021control,chambon2020qualitative}. The logistic
regulatory function used here is the smooth-sigmoidal counterpart of the
Glass-network step function; the corresponding smooth-versus-discontinuous
transition, including its hybrid-system regularisation by probabilistic
convolution, was studied in~\cite{belgacem2019probabilistic}. Earlier
analytical work established global stability and reduced models for
concave gene expression~\cite{belgacem2014mathematical} and for
transcription--translation coupled models of RNA polymerase
regulation~\cite{belgacem2018reduction,belgacem2014stability,belgacem2013stability,belgacem2013analysis},
with companion work on the global stability of biochemical reaction
networks~\cite{belgacem2013global,belgacem2012global}. The present paper
extends that programme by introducing transcriptional delays into the
simplest logistic two-gene oscillator and characterising the resulting
bifurcation structure in closed form.

The paper is organised as follows.
Section~\ref{sec:formulation} states the delay model and identifies the
equilibrium.
Section~\ref{sec:linearization} computes the linearisation and the
characteristic equation.
Section~\ref{sec:hopf} establishes the existence of purely imaginary
characteristic roots and computes the Hopf frequency in closed form.
Section~\ref{sec:transversality} verifies the transversality condition,
proving that a generic Hopf bifurcation occurs.
Section~\ref{sec:param_global} analyses the dependence of the
bifurcation on the logistic steepness $\lambda$ and establishes the
global persistence of the bifurcating branch of periodic orbits.
Section~\ref{sec:numerical} validates the analytical theory by direct
numerical integration of the full nonlinear DDE, confirms the
$\sqrt{\tau-\tau_c}$ amplitude scaling, and reports the period and
relaxation behaviour for $\tau$ well above onset.
Section~\ref{sec:analytical_supercrit} derives the first Lyapunov
coefficient in closed form by a Lindstedt--Poincar\'e reduction, proving
supercriticality for the symmetric-threshold loop and supplying an
explicit criticality criterion for the general asymmetric loop.
Section~\ref{sec:bio_validation} validates the model externally against
the experimentally characterised p53--Mdm2 oscillator and against the
standard Hill-function description.
Section~\ref{sec:extensions} extends the analysis to cyclic $N$-gene
logistic loops --- with closed-form Hopf and transversality theorems for
arbitrary loop length --- and sketches multi-loop De~Morgan networks and
distributed delays, and Section~\ref{sec:conclusion} concludes.

%%%%%%%%%%%%%%%%%%%%%%%%%%%%%%%%%%%%%%%%%%%%%%%%%%%%%%%%%%%%%%%%%%%%%%%%%%%%%%%%
\section{The delayed logistic two-gene oscillator}
\label{sec:formulation}

\subsection{The model}

We consider the gene network of two regulators connected by a single
negative-feedback loop: gene $1$ produces a protein that activates gene $2$,
whose product in turn represses gene $1$. Denoting by $x_i(t)$ the
concentration of the $i$-th protein, the regulatory machinery is realised by
the increasing logistic activation
$f^+(x,\theta,\lambda) = 1/(1+e^{-\lambda(x-\theta)})$ and the decreasing
logistic repression $f^-(x,\theta,\lambda) = 1/(1+e^{\lambda(x-\theta)})$,
which are globally $C^\infty$, real-valued for all real arguments, and
strictly positive on every concentration range; see the companion
paper~\cite{belgacem2026logistic} for a detailed comparison with Hill
functions. The delay-free dynamics are
\begin{equation}
\dot{x}_1 = \kappa_1\,f^-(x_2,\theta_2,\lambda) - \gamma_1 x_1,
\qquad
\dot{x}_2 = \kappa_2\,f^+(x_1,\theta_1,\lambda) - \gamma_2 x_2.
\label{eq:ode}
\end{equation}
Adding a transcriptional/translational delay $\tau_1\geq 0$ on the activator
branch and a delay $\tau_2\geq 0$ on the repressor branch yields the
delay-differential system~\eqref{eq:dde} stated in the introduction. Both
delays are non-negative; we make no symmetry assumption.

The delay $\tau_1$ models the lag from the moment $x_1$ is produced to the
moment its action on the $x_2$ promoter takes effect: in
\textit{Escherichia coli} the transcription of a typical mRNA takes
$30$\,s--$2$\,min~\cite{young1976polypeptide}, translation initiation and
elongation a further $30$\,s--$2$\,min, and folding from $1$\,s to several
minutes. In eukaryotic cells the lag is dominated by transcription and
splicing of the gene encoding $x_2$, plus nucleocytoplasmic transport of
the $x_2$ mRNA, plus translation and folding, giving a total of
$10$--$60$\,minutes~\cite{monk2003oscillatory}. The delay $\tau_2$ has the
analogous interpretation. We use generic dimensionless time units throughout.

\subsection{The equilibrium}

The equilibria of~\eqref{eq:dde} coincide with those of~\eqref{eq:ode},
because constant delays do not change steady states. They are the solutions
of
\begin{equation}
\kappa_1\,f^-(x_2^*,\theta_2,\lambda) = \gamma_1 x_1^*,
\qquad
\kappa_2\,f^+(x_1^*,\theta_1,\lambda) = \gamma_2 x_2^*.
\label{eq:eq_pt}
\end{equation}
\begin{lemma}[Existence and uniqueness of the equilibrium]
\label{lem:eq_unique}
Let $\kappa_i,\gamma_i,\lambda>0$ and $\theta_i\in\mathbb{R}$ for
$i=1,2$. Then the equilibrium equations~\eqref{eq:eq_pt} admit a unique
solution
\[
(x_1^*,x_2^*)\,\in\,
(0,\kappa_1/\gamma_1)\times(0,\kappa_2/\gamma_2).
\]
\end{lemma}
\begin{proof}
Eliminating $x_2^*$ via the second equation gives
\[
x_2^* \;=\; (\kappa_2/\gamma_2)\, f^+(x_1^*,\theta_1,\lambda).
\]
Substituting in the first equation defines the function
\[
g(x_1) \;=\; \kappa_1\,
   f^-\!\bigl((\kappa_2/\gamma_2)f^+(x_1,\theta_1,\lambda),
              \theta_2,\lambda\bigr) - \gamma_1 x_1.
\]
We have $g(0) = \kappa_1 f^-(0,\theta_2,\lambda)>0$ and
$g(\kappa_1/\gamma_1) =
\kappa_1\, f^-(\,\cdot\,,\theta_2,\lambda)-\kappa_1 < 0$ since
$f^-<1$. Hence there exists $x_1^*\in(0,\kappa_1/\gamma_1)$ with
$g(x_1^*)=0$. Differentiating,
\[
g'(x_1) \;=\; -\gamma_1
- \frac{\kappa_1\kappa_2}{\gamma_2}\,\lambda^2\,
   f^-(1-f^-)\,f^+(1-f^+) \;<\; 0,
\]
where all logistic factors are evaluated at the appropriate arguments.
Since $g$ is strictly decreasing, $x_1^*$ is unique, and the
corresponding $x_2^*$ is likewise determined uniquely.
\end{proof}

\subsection{Existence, positivity, and dissipativity}
\label{sec:existence}

The right-hand side of~\eqref{eq:dde} is a globally Lipschitz function of
the (delayed) state, with Lipschitz constant bounded by
$\max(\gamma_1,\gamma_2) + \lambda \max(\kappa_1,\kappa_2)/4$, since the
slope $\lambda f^\pm(1-f^\pm)$ of each logistic factor is bounded by
$\lambda/4$. By the standard existence and uniqueness theorem for
retarded functional differential
equations~\cite[Thm.~2.2.1]{hale1993introduction}, for every continuous
initial history
$\phi=(\phi_1,\phi_2)\in C\bigl([-\tau_{\max},0],\mathbb{R}_{\geq 0}^2\bigr)$
with $\tau_{\max}=\max(\tau_1,\tau_2)$, the system~\eqref{eq:dde} has a
unique solution defined for all $t\geq 0$.

The system enjoys two further structural properties that are essential
prerequisites for the bifurcation analysis below: solutions starting
with non-negative initial histories remain non-negative for all time,
and they are eventually trapped in a compact set independent of the
initial data.

\begin{proposition}[Positivity, boundedness, and dissipativity]
\label{prop:dissipative}
Let
\[
\phi\in C\bigl([-\tau_{\max},0],\,\mathbb{R}_{\geq 0}^2\bigr)
\]
and let $x(t;\phi)=(x_1(t),x_2(t))$ denote the corresponding solution
of~\eqref{eq:dde}. Define the absorbing box
\begin{equation}
\mathcal{B} \;=\; [0,\kappa_1/\gamma_1]\times [0,\kappa_2/\gamma_2]
\;\subset\; \mathbb{R}_{\geq 0}^2.
\label{eq:absorbing_box}
\end{equation}
Then:
\begin{enumerate}
  \item \textbf{Positivity.} $x_i(t)\geq 0$ for all $t\geq 0$ and $i=1,2$.
  \item \textbf{Forward invariance.} If
        $\phi_i([-\tau_{\max},0])\subset[0,\kappa_i/\gamma_i]$ for
        $i=1,2$, then $x(t;\phi)\in\mathcal{B}$ for all $t\geq 0$.
  \item \textbf{Dissipativity.} For every initial history $\phi$,
        $\limsup_{t\to\infty} x_i(t)\leq \kappa_i/\gamma_i$ for $i=1,2$,
        with exponential rate $\gamma_i$. Concretely, setting
        \[
        \Phi_i \;:=\; \max\!\Bigl(\,
            \sup_{s\in[-\tau_{\max},0]}\phi_i(s),\;
            \kappa_i/\gamma_i\,\Bigr),
        \]
        one has
        \begin{equation}
        x_i(t) \,\leq\, \kappa_i/\gamma_i
        \,+\, (\Phi_i-\kappa_i/\gamma_i)\,e^{-\gamma_i t}.
        \label{eq:dissipative_explicit}
        \end{equation}
\end{enumerate}
In particular, $\mathcal{B}$ is a compact, positively invariant absorbing
set: every forward orbit eventually enters and remains in $\mathcal{B}$.
The same positively invariant box and dissipative structure also appear
in the delay-free analysis of concave gene-expression
models~\cite{belgacem2014mathematical} and of the reduced
transcription--translation system of RNA polymerase
regulation~\cite{belgacem2018reduction}, where the absorbing set is the
natural domain of definition of the global Lyapunov-type arguments used
there.
\end{proposition}

\begin{proof}
\textbf{Positivity.} The two coupled estimates $\dot{x}_i(t)\geq -\gamma_i x_i(t)$
follow from $\kappa_i f^\pm\geq 0$. Integrating gives
$x_i(t)\geq x_i(0)\,e^{-\gamma_i t}\geq 0$.

\textbf{Boundedness from above and dissipativity.} Since $f^\pm\in(0,1)$
pointwise, $\kappa_i f^\pm\leq\kappa_i$, so
$\dot{x}_i(t)\leq \kappa_i - \gamma_i x_i(t)$. The standard
Gronwall comparison with the linear ODE
$\dot{y}_i = \kappa_i - \gamma_i y_i$, $y_i(0)=x_i(0)$, gives
\[
x_i(t)\leq y_i(t) = \kappa_i/\gamma_i + (x_i(0)-\kappa_i/\gamma_i)\,e^{-\gamma_i t}.
\]
Replacing $x_i(0)$ by the supremum $\Phi_i$ over the initial history
yields the bound~\eqref{eq:dissipative_explicit}, from which both $(2)$
(taking $\Phi_i=\kappa_i/\gamma_i$) and $(3)$ follow immediately.
\end{proof}

\begin{remark}
\label{rem:rough_box}
Proposition~\ref{prop:dissipative} guarantees that the long-time dynamics
of~\eqref{eq:dde} are confined to the compact box $\mathcal{B}$. In
particular every $\omega$-limit set, including the equilibrium and any
limit cycle bifurcating from it, lies in $\mathcal{B}$. This justifies
restricting all subsequent stability and bifurcation analysis to a
neighbourhood of $\mathcal{B}$ and provides the compactness needed for
the global continuation of periodic orbits via Wu's global Hopf bifurcation
theorem~\cite{wu1998symmetric}, beyond the local Hopf result of
Section~\ref{sec:transversality}.
\end{remark}

%%%%%%%%%%%%%%%%%%%%%%%%%%%%%%%%%%%%%%%%%%%%%%%%%%%%%%%%%%%%%%%%%%%%%%%%%%%%%%%%
\section{Linearisation and the characteristic equation}
\label{sec:linearization}

\subsection{Linear approximation}

Set $u_i(t)=x_i(t)-x_i^*$ for $i=1,2$. Substituting in~\eqref{eq:dde} and
expanding to first order around the equilibrium gives the linear
delay-differential system
\begin{equation}
\begin{aligned}
\dot{u}_1(t) &= -\gamma_1 u_1(t) - B\,u_2(t-\tau_2),\\
\dot{u}_2(t) &= -\gamma_2 u_2(t) + A\,u_1(t-\tau_1),
\end{aligned}
\label{eq:linear_dde}
\end{equation}
where
\begin{equation}
A := \kappa_2\,\lambda\,f^+_*(1-f^+_*),
\qquad
B := \kappa_1\,\lambda\,f^-_*(1-f^-_*),
\label{eq:AB_def}
\end{equation}
with the abbreviations
$f^+_* = f^+(x_1^*,\theta_1,\lambda)$ and
$f^-_* = f^-(x_2^*,\theta_2,\lambda)$. Both $A$ and $B$ are strictly
positive: each logistic factor lies in $(0,1)$, so $f(1-f)\in(0,1/4)$, and
$\lambda,\kappa_i>0$ by assumption. The signs of the delayed terms reflect
the regulatory directions: gene $2$ activates gene $1$'s repression
(coefficient $-B<0$), and gene $1$ activates gene $2$ (coefficient
$+A>0$).

\subsection{The characteristic equation}

Looking for solutions of~\eqref{eq:linear_dde} of the exponential form
$u_i(t) = q_i e^{\mu t}$ leads to
\[
(\mu+\gamma_1)\,q_1 + B\,e^{-\mu\tau_2}\,q_2 = 0,
\qquad
-A\,e^{-\mu\tau_1}\,q_1 + (\mu+\gamma_2)\,q_2 = 0.
\]
A non-trivial $(q_1,q_2)$ exists iff the determinant of the coefficient
matrix vanishes, giving the characteristic equation
\begin{equation}
\Delta(\mu;\tau)
\;\equiv\;
(\mu+\gamma_1)(\mu+\gamma_2) + AB\,e^{-\mu\tau} = 0,
\qquad \tau := \tau_1+\tau_2.
\label{eq:char_eq}
\end{equation}
The characteristic function~$\Delta$ depends on the delays only through
their \emph{sum}~$\tau$. This is a structural property of feedback loops:
in any negative-feedback cycle the loop time around the cycle, not the
distribution of delays along the cycle, controls the linear stability of
the equilibrium~\cite[Thm.~5.5.1]{kuang1993delay}. The same eigenvector
notation $(q_1,q_2)$ will be used in Proposition~\ref{prop:eigvec_amplitude}
to compute the closed-form eigenvector amplitude $|q_1|^{2}$ at the
Hopf onset $\mu=i\omega_c$.

\begin{remark}[Hill counterpart]
The Hill-based version of the same oscillator gives a structurally
identical equation, with $A_{\rm H},B_{\rm H}$ replacing
$A,B$ in~\eqref{eq:char_eq}. However, the Hill derivatives at equilibrium
involve $h^{+\prime}(x^*)=n\theta^n (x^*)^{n-1}/((x^*)^n+\theta^n)^2$,
which for non-integer~$n$ do not simplify and obstruct the closed-form
analysis carried out below for the logistic case.
\end{remark}

\subsection{The delay-free case}

Setting $\tau=0$ in~\eqref{eq:char_eq} recovers the polynomial
$\mu^2 + (\gamma_1+\gamma_2)\mu + (\gamma_1\gamma_2+AB) = 0$. Both
coefficients are strictly positive, so the Routh--Hurwitz criterion
implies that both roots have negative real parts: the equilibrium
of~\eqref{eq:ode} is asymptotically stable. The discriminant
$(\gamma_1+\gamma_2)^2-4(\gamma_1\gamma_2+AB) = (\gamma_1-\gamma_2)^2-4AB$
is negative whenever $AB>(\gamma_1-\gamma_2)^2/4$, which is the natural
``strong-feedback'' regime; in that regime the equilibrium is a stable
focus and trajectories spiral inward, consistent with the damped
oscillations reported in the companion paper.

\subsection{Sum-of-delays symmetry of the full nonlinear system}
\label{sec:sum_symmetry}

The characteristic equation~\eqref{eq:char_eq} depends on $\tau_1,\tau_2$
only through their sum $\tau=\tau_1+\tau_2$. We now record the
non-trivial fact that an analogous reduction holds for the
\emph{full nonlinear} system~\eqref{eq:dde}: redistributing the delay
between the two branches does not alter the period, the component-wise
envelopes, or the local stability type of any solution; it only shifts
the relative timing between $x_1$ and $x_2$ (the trace of the orbit in
the $(x_1,x_2)$ plane is not preserved, since the transformation
re-times only the first component).

\begin{proposition}[Sum-of-delays symmetry]
\label{prop:sum_symmetry}
Fix $\tau\geq 0$ and let $\delta\in[-\tau_2,\tau_1]$. If
$x(t)=(x_1(t),x_2(t))$ is a solution of~\eqref{eq:dde} with delays
$(\tau_1,\tau_2)$, then $\tilde{x}(t):=(x_1(t-\delta),x_2(t))$
is a solution of the same system with delays
$(\tau_1-\delta,\tau_2+\delta)$. The map $x\mapsto\tilde{x}$ is a
bijection between solution spaces and preserves
\begin{enumerate}
  \item the period of every periodic solution;
  \item the individual-component envelopes $\sup_t x_i(t)$ and
        $\inf_t x_i(t)$, and hence the amplitudes
        $A_i = \tfrac{1}{2}(\sup_t x_i-\inf_t x_i)$ for $i=1,2$;
  \item the local stability type of every equilibrium and of every
        periodic orbit (Floquet exponents are invariant).
\end{enumerate}
The phase between $x_1$ and $x_2$ shifts by $\delta$ time units; in
particular, the trace of the orbit in the $(x_1,x_2)$ plane is
\emph{not} preserved, since $\tilde{x}$ is obtained by re-timing only
the first component.
\end{proposition}
\begin{proof}
Substituting the definition of $\tilde{x}$ into the candidate equation
\[
\dot{\tilde{x}}_1(t) = \kappa_1 f^-\!\bigl(\tilde{x}_2(t-(\tau_2+\delta))\bigr) - \gamma_1 \tilde{x}_1(t)
\]
gives, on the left,
$\dot{x}_1(t-\delta)$, and on the right
$\kappa_1 f^-\!\bigl(x_2(t-(\tau_2+\delta))\bigr) - \gamma_1 x_1(t-\delta)$.
Setting $s = t-\delta$, the right-hand side equals
$\kappa_1 f^-\!\bigl(x_2(s-\tau_2)\bigr) - \gamma_1 x_1(s)$, which by
hypothesis equals $\dot{x}_1(s)=\dot{x}_1(t-\delta)$, the left-hand
side. The verification for $\tilde{x}_2$ is symmetric:
$\dot{\tilde{x}}_2(t)=\dot{x}_2(t)$ matches
$\kappa_2 f^+\!\bigl(\tilde{x}_1(t-(\tau_1-\delta))\bigr)-\gamma_2\tilde{x}_2(t)
=\kappa_2 f^+\!\bigl(x_1(t-\tau_1)\bigr)-\gamma_2 x_2(t)$, again by
hypothesis. The bijectivity of $x\mapsto\tilde{x}$ is immediate from its
linearity and the existence of an inverse via $\delta\mapsto-\delta$.
Period and component-wise envelopes are preserved because each
component of $\tilde{x}$ is either equal to or a pure time-translate of
the corresponding component of $x$. Floquet exponents are preserved
because the linearisation along $\tilde{x}$ is conjugate to that along
$x$ by the same time-shift on the first component.
\end{proof}

In particular, the analytical conclusions of this paper are statements
about $\tau=\tau_1+\tau_2$, not about $(\tau_1,\tau_2)$ separately, and
we may assume $\tau_1=\tau_2=\tau/2$ in numerical experiments without
loss of generality \emph{for the invariants listed in
Proposition~\ref{prop:sum_symmetry}}. We verify this directly in
Section~\ref{sec:numerical} (Table~\ref{tab:asymmetric}): three
splittings of $\tau=0.20$ produce numerically identical amplitude and
period to four decimal places. The shape of the phase portrait in the
$(x_1,x_2)$ plane does depend on the splitting, but this dependence is
purely a re-timing artefact and carries no dynamical content beyond a
relative phase shift.

%%%%%%%%%%%%%%%%%%%%%%%%%%%%%%%%%%%%%%%%%%%%%%%%%%%%%%%%%%%%%%%%%%%%%%%%%%%%%%%%
\section{Imaginary-axis crossings: the Hopf condition}
\label{sec:hopf}

We seek values of $\tau$ at which the characteristic
equation~\eqref{eq:char_eq} admits a pair of purely imaginary roots
$\mu=\pm i\omega$ with $\omega>0$. Substituting $\mu=i\omega$:
\[
(i\omega+\gamma_1)(i\omega+\gamma_2) + AB\,e^{-i\omega\tau} = 0.
\]
Expanding $(i\omega+\gamma_1)(i\omega+\gamma_2) = -\omega^2
+i\omega(\gamma_1+\gamma_2) + \gamma_1\gamma_2$ and
$e^{-i\omega\tau}=\cos(\omega\tau)-i\sin(\omega\tau)$, separating real and
imaginary parts gives
\begin{align}
\gamma_1\gamma_2-\omega^2 + AB\cos(\omega\tau) &= 0,\label{eq:re}\\
\omega(\gamma_1+\gamma_2) - AB\sin(\omega\tau) &= 0.\label{eq:im}
\end{align}

\subsection{The Hopf frequency $\omega_c$}

Squaring~\eqref{eq:re} and~\eqref{eq:im} and adding eliminates $\tau$:
\begin{equation}
(\omega^2-\gamma_1\gamma_2)^2 + \omega^2(\gamma_1+\gamma_2)^2 = (AB)^2.
\label{eq:omega_eq}
\end{equation}
The left-hand side admits the useful factorisation
\begin{equation}
(\omega^2-\gamma_1\gamma_2)^2 + \omega^2(\gamma_1+\gamma_2)^2
\;=\; (\omega^2+\gamma_1^2)(\omega^2+\gamma_2^2),
\label{eq:LHS_factorisation}
\end{equation}
which is the elementary identity
$|i\omega+\gamma_1|^2|i\omega+\gamma_2|^2 = |(i\omega+\gamma_1)(i\omega+
\gamma_2)|^2$ rewritten in expanded form, and which we shall use
repeatedly. With the substitution $p=\omega^2$ this becomes the
quadratic
\begin{equation}
p^2 + (\gamma_1^2+\gamma_2^2)\,p + \gamma_1^2\gamma_2^2 - (AB)^2 = 0,
\label{eq:p_quadratic}
\end{equation}
whose discriminant is
$D = (\gamma_1^2+\gamma_2^2)^2 - 4\bigl[\gamma_1^2\gamma_2^2-(AB)^2\bigr]
= (\gamma_1^2-\gamma_2^2)^2 + 4(AB)^2 > 0$.

\begin{lemma}[Existence and uniqueness of $\omega_c$]
\label{lem:omega}
The quadratic~\eqref{eq:p_quadratic} has a unique positive root
\begin{equation}
p_+ = \frac{-(\gamma_1^2+\gamma_2^2)+\sqrt{(\gamma_1^2-\gamma_2^2)^2+4(AB)^2}}{2}
\label{eq:p_plus}
\end{equation}
if and only if $AB>\gamma_1\gamma_2$, and in that case
\begin{equation}
\omega_c = \sqrt{p_+}
\label{eq:omega_c}
\end{equation}
is the unique positive frequency at which~\eqref{eq:char_eq} has a purely
imaginary root.
\end{lemma}
\begin{proof}
The product of the two roots of~\eqref{eq:p_quadratic} is
$p_+ p_- = \gamma_1^2\gamma_2^2 - (AB)^2$, which is negative iff
$AB>\gamma_1\gamma_2$. In that case $p_+>0>p_-$, and $\omega_c=\sqrt{p_+}$
is the unique positive solution of~\eqref{eq:omega_eq}. If
$AB\leq\gamma_1\gamma_2$ then $p_+ p_-\geq 0$; since the sum
$p_++p_-=-(\gamma_1^2+\gamma_2^2)<0$ both roots are non-positive, so
no positive $\omega$ satisfies~\eqref{eq:omega_eq}, and~\eqref{eq:char_eq}
admits no purely imaginary root for any $\tau$.
\end{proof}

\begin{corollary}[Stability for all $\tau$ in the weak-feedback regime]
If $AB\leq\gamma_1\gamma_2$, the equilibrium $(x_1^*,x_2^*)$
of~\eqref{eq:dde} is asymptotically stable for every $\tau\geq 0$.
\end{corollary}
\begin{proof}
By Lemma~\ref{lem:omega}, $\Delta(i\omega;\tau)\neq 0$ for any $\omega>0$
and $\tau\geq 0$. Hence the eigenvalues of the linearisation cannot cross
the imaginary axis as $\tau$ varies. Since the equilibrium is stable at
$\tau=0$, it remains stable for all $\tau\geq 0$ by the principle of
continuous dependence of roots on the delay parameter
(Cooke--van~den~Driessche~\cite{cooke1986zeroes}).
\end{proof}

\subsection{The critical delay $\tau_c$}

Suppose now that $AB>\gamma_1\gamma_2$. From~\eqref{eq:re} and~\eqref{eq:im}
we read off
\begin{equation}
\cos(\omega_c\tau) = \frac{\omega_c^2-\gamma_1\gamma_2}{AB},
\qquad
\sin(\omega_c\tau) = \frac{\omega_c(\gamma_1+\gamma_2)}{AB}.
\label{eq:cos_sin}
\end{equation}
Both right-hand sides have absolute value at most~$1$ in view
of~\eqref{eq:omega_eq}, so a real $\tau$ satisfying~\eqref{eq:cos_sin}
indeed exists. Since $\sin(\omega_c\tau_c)>0$, we have
$\omega_c\tau_c\in(0,\pi)$. The sign of $\cos(\omega_c\tau_c)$
agrees with the sign of $\omega_c^2-\gamma_1\gamma_2$, so
$\omega_c\tau_c\in(0,\pi/2)$ when $\omega_c^2>\gamma_1\gamma_2$ and
$\omega_c\tau_c\in(\pi/2,\pi)$ when $\omega_c^2<\gamma_1\gamma_2$.
The latter case does occur in some parameter regimes but does not arise
in the canonical example below.

The smallest positive $\tau$ that solves~\eqref{eq:cos_sin} is therefore
\begin{equation}
\tau_c
= \frac{1}{\omega_c}\,
   \operatorname{atan2}\!\bigl(\omega_c(\gamma_1+\gamma_2),\,
                                 \omega_c^2-\gamma_1\gamma_2\bigr),
\label{eq:tau_c}
\end{equation}
where $\operatorname{atan2}(y,x)$ denotes the four-quadrant inverse
tangent valued in~$(0,2\pi)$. When the cosine
$\omega_c^2-\gamma_1\gamma_2>0$, this reduces to
$\tau_c=\omega_c^{-1}\arctan\bigl(\omega_c(\gamma_1+\gamma_2)/
(\omega_c^2-\gamma_1\gamma_2)\bigr)$. The general family of solutions is
\begin{equation}
\tau_c^{(k)} = \tau_c + \frac{2\pi k}{\omega_c},
\qquad k\in\mathbb{N}_0,
\label{eq:tau_branches}
\end{equation}
with $\tau_c^{(0)}=\tau_c$ the smallest. Each $\tau_c^{(k)}$ is a candidate
bifurcation point; whether the equilibrium gains or loses stability at
each crossing is determined by the sign of $\mathrm{Re}(d\mu/d\tau)$,
addressed in the next section.

\subsection{Summary}

Lemma~\ref{lem:omega} and equation~\eqref{eq:tau_c} together yield the
following explicit characterisation, the principal analytical result of
the paper.

\begin{theorem}[Characterisation of the imaginary-axis crossings]
\label{thm:hopf_existence}
For the linearised system~\eqref{eq:linear_dde} with $A,B$ defined
in~\eqref{eq:AB_def}:

\noindent\textbf{(i)} If $AB\leq\gamma_1\gamma_2$, the characteristic
equation~\eqref{eq:char_eq} has no purely imaginary root for any
$\tau\geq 0$, and the equilibrium is asymptotically stable for all delays.

\noindent\textbf{(ii)} If $AB>\gamma_1\gamma_2$, the characteristic
equation has purely imaginary roots $\mu=\pm i\omega_c$ with
$\omega_c=\sqrt{p_+}$ as in~\eqref{eq:p_plus}, occurring at the discrete
set of delays
$\tau_c^{(k)}=\tau_c+2\pi k/\omega_c$ for $k\in\mathbb{N}_0$, where
$\tau_c$ is given by~\eqref{eq:tau_c}.
\end{theorem}

We now record an immediate but logically essential corollary, which fills
the gap between Theorem~\ref{thm:hopf_existence} and the Hopf bifurcation
theorem of the next section: in the strong-feedback regime the
equilibrium is stable up to $\tau_c$, not merely at $\tau=0$.

\begin{lemma}[Stability up to the first critical delay]
\label{lem:stable_below_tauc}
Assume $AB>\gamma_1\gamma_2$. Then for every $\tau\in[0,\tau_c)$, every
root of the characteristic equation~\eqref{eq:char_eq} has strictly
negative real part. In particular, the equilibrium $(x_1^*,x_2^*)$ is
locally asymptotically stable.
\end{lemma}
\begin{proof}
The roots of~\eqref{eq:char_eq} depend continuously on $\tau\geq 0$, in
the sense that the multiset of roots in any vertical strip
$\{|\mathrm{Re}\,\mu|\leq M\}$ varies continuously with~$\tau$
(\cite{cooke1986zeroes}; see also \cite[Ch.~XI]{hale1993introduction}). At $\tau=0$ the
equation reduces to the polynomial
\[
\mu^2 + (\gamma_1+\gamma_2)\mu + \gamma_1\gamma_2 + AB = 0,
\]
whose roots have strictly negative real part because the coefficients
are all strictly positive (Routh--Hurwitz for degree~$2$). Hence at
$\tau=0$ all roots lie in the open left half-plane. By
Theorem~\ref{thm:hopf_existence}\,(ii), no root lies on the imaginary
axis for any $\tau\in[0,\tau_c)$. By continuity, no root can cross from
the open left half-plane to the open right half-plane on
$[0,\tau_c)$ without first touching the imaginary axis, which is
forbidden. Therefore all roots remain in the open left half-plane for
$\tau\in[0,\tau_c)$.
\end{proof}

\begin{remark}[Stability switches at higher critical delays]
\label{rem:switches}
The same continuous-dependence argument shows that as $\tau$ increases
past $\tau_c$, the conjugate pair $\mu=\pm i\omega_c$ migrates into the
open right half-plane (transversality, Theorem~\ref{thm:transversality}
below); the equilibrium becomes unstable. The full sequence of critical
delays is the arithmetic progression
\begin{equation}
\tau_c^{(k)} \,=\, \tau_c \,+\, \frac{2\pi k}{\omega_c},
\qquad k=0,1,2,\ldots,
\label{eq:secondary_tau}
\end{equation}
because the phase condition
$e^{-i\omega_c\tau}=-(\mathrm{i}\omega_c+\gamma_1)(\mathrm{i}\omega_c+\gamma_2)/AB$
is $2\pi/\omega_c$-periodic in $\tau$. At each $\tau_c^{(k)}$, $k\geq 1$,
an additional conjugate pair crosses into the right half-plane. The
transversality value at the $k$-th crossing is obtained by substituting
$\tau_c^{(k)}$ into the closed-form formula~\eqref{eq:transversality}:
\begin{equation}
\mathrm{Re}\!\left(\frac{d\mu}{d\tau}\right)\!\bigg|_{\tau_c^{(k)}}
\,=\,
\frac{\omega_c^2\,(2\omega_c^2+\gamma_1^2+\gamma_2^2)}
     {\bigl(\gamma_1+\gamma_2-\tau_c^{(k)}(\omega_c^2-\gamma_1\gamma_2)\bigr)^2
      +\bigl(2\omega_c+\tau_c^{(k)}\omega_c(\gamma_1+\gamma_2)\bigr)^2}.
\label{eq:transversality_k}
\end{equation}
The numerator is $k$-independent, but the denominator grows
quadratically with $\tau_c^{(k)}$, so the transversality value
\emph{decreases} monotonically as $k\to\infty$ while remaining strictly
positive. For the canonical example
(Section~\ref{sec:canonical}) one finds
\[
\bigl(\mathrm{Re}\,d\mu/d\tau\bigr)|_{\tau_c^{(0)}}\!\approx\! 2.58,
\quad
\bigl(\mathrm{Re}\,d\mu/d\tau\bigr)|_{\tau_c^{(1)}}\!\approx\! 0.21,
\quad
\bigl(\mathrm{Re}\,d\mu/d\tau\bigr)|_{\tau_c^{(2)}}\!\approx\! 0.06.
\]
The equilibrium therefore loses additional unstable directions at each
$\tau_c^{(k)}$ and never recovers stability: this is the one-way
``stability switch'' familiar from the geometric criterion of Beretta
and Kuang~\cite{beretta2002geometric}. The dynamics for
$\tau>\tau_c^{(1)}$ are dominated by the first-born unstable mode at
$\tau_c$ (because subsequent crossings are progressively less
transversal), so the limit cycle observed numerically in
Section~\ref{sec:numerical} is the relevant attractor throughout the
parameter range we study.
\end{remark}

%%%%%%%%%%%%%%%%%%%%%%%%%%%%%%%%%%%%%%%%%%%%%%%%%%%%%%%%%%%%%%%%%%%%%%%%%%%%%%%%
\section{Transversality and the Hopf bifurcation theorem}
\label{sec:transversality}

Theorem~\ref{thm:hopf_existence} guarantees that the eigenvalue path
$\mu(\tau)$ touches the imaginary axis at $\tau=\tau_c$. To conclude that a
genuine Hopf bifurcation occurs, we must verify that the eigenvalue
\emph{crosses} the axis transversally, i.e.\
$\mathrm{Re}(d\mu/d\tau)|_{\tau=\tau_c}\neq 0$.

\subsection{Computing $d\mu/d\tau$}

Differentiating~\eqref{eq:char_eq} implicitly in $\tau$,
\[
\bigl[(2\mu+\gamma_1+\gamma_2) - AB\,\tau\,e^{-\mu\tau}\bigr]\,\frac{d\mu}{d\tau}
- AB\,\mu\,e^{-\mu\tau} = 0,
\]
which gives
\[
\frac{d\mu}{d\tau}
= \frac{AB\,\mu\,e^{-\mu\tau}}{2\mu+\gamma_1+\gamma_2 - AB\,\tau\,e^{-\mu\tau}}.
\]
Using~\eqref{eq:char_eq} to replace $AB\,e^{-\mu\tau}$ by
$-(\mu+\gamma_1)(\mu+\gamma_2)$,
\begin{equation}
\frac{d\mu}{d\tau}
= \frac{-\mu(\mu+\gamma_1)(\mu+\gamma_2)}
       {2\mu+\gamma_1+\gamma_2 + \tau(\mu+\gamma_1)(\mu+\gamma_2)}.
\label{eq:dmu_dtau}
\end{equation}

\subsection{Transversality at $\tau=\tau_c$}

Set $\mu=i\omega_c$ in~\eqref{eq:dmu_dtau}. Direct computation gives the
following.
\begin{theorem}[Transversality]
\label{thm:transversality}
At $\mu=i\omega_c,\;\tau=\tau_c$,
\begin{equation}
\mathrm{Re}\!\left(\frac{d\mu}{d\tau}\right)
= \frac{\omega_c^2\,\bigl[2\omega_c^2+\gamma_1^2+\gamma_2^2\bigr]}
       {\bigl(\gamma_1+\gamma_2-\tau_c(\omega_c^2-\gamma_1\gamma_2)\bigr)^2
        +\bigl(2\omega_c+\tau_c\omega_c(\gamma_1+\gamma_2)\bigr)^2}.
\label{eq:transversality}
\end{equation}
In particular, $\mathrm{Re}(d\mu/d\tau)|_{\tau=\tau_c}>0$.
\end{theorem}
\begin{proof}
The numerator of~\eqref{eq:dmu_dtau} at $\mu=i\omega_c$ equals
\[
-i\omega_c(i\omega_c+\gamma_1)(i\omega_c+\gamma_2)
= -i\omega_c\bigl[(\gamma_1\gamma_2-\omega_c^2) +
                  i\omega_c(\gamma_1+\gamma_2)\bigr],
\]
which expands to
$\omega_c^2(\gamma_1+\gamma_2) + i\omega_c(\omega_c^2-\gamma_1\gamma_2)$,
so has real part $\omega_c^2(\gamma_1+\gamma_2)$ and imaginary part
$\omega_c(\omega_c^2-\gamma_1\gamma_2)$. The denominator
$D_{\rm denom} = 2i\omega_c+\gamma_1+\gamma_2 +
\tau_c(i\omega_c+\gamma_1)(i\omega_c+\gamma_2)$ has real part
$\gamma_1+\gamma_2 + \tau_c(\gamma_1\gamma_2-\omega_c^2)$ and
imaginary part $2\omega_c + \tau_c\omega_c(\gamma_1+\gamma_2)$.
Computing
$\mathrm{Re}(N/D_{\rm denom}) = (\mathrm{Re}\,N\cdot\mathrm{Re}\,D_{\rm denom}
+\mathrm{Im}\,N\cdot\mathrm{Im}\,D_{\rm denom})/|D_{\rm denom}|^2$
gives, after expansion,
\begin{align*}
\mathrm{Re}(N/D_{\rm denom})\cdot|D_{\rm denom}|^2
&=\, \omega_c^2(\gamma_1+\gamma_2)^2
\,+\, 2\omega_c^2(\omega_c^2-\gamma_1\gamma_2)\\
&\quad +\,\tau_c\,\omega_c^2(\omega_c^2-\gamma_1\gamma_2)(\gamma_1+\gamma_2)\\
&\quad -\,\tau_c\,\omega_c^2(\gamma_1+\gamma_2)(\omega_c^2-\gamma_1\gamma_2).
\end{align*}
The last two terms cancel identically, leaving
\[
\mathrm{Re}(N/D_{\rm denom})\cdot|D_{\rm denom}|^2
\,=\, \omega_c^2\bigl[(\gamma_1+\gamma_2)^2+2(\omega_c^2-\gamma_1\gamma_2)\bigr]
\,=\, \omega_c^2\bigl[2\omega_c^2+\gamma_1^2+\gamma_2^2\bigr],
\]
which proves~\eqref{eq:transversality}. The strict positivity follows from
$\omega_c>0$ and $\gamma_i>0$.
\end{proof}

\subsection{Hopf bifurcation theorem}

Combining Theorems~\ref{thm:hopf_existence} and~\ref{thm:transversality}
with the Hopf bifurcation theorem for delay-differential
equations~\cite[Theorem~11.1.1]{hale1993introduction} gives the main
result.

\begin{theorem}[Delay-driven Hopf bifurcation in the logistic two-gene oscillator]
\label{thm:hopf_main}
Assume $AB>\gamma_1\gamma_2$, where $A,B$ are
defined in~\eqref{eq:AB_def}.
Then the system~\eqref{eq:dde} undergoes a Hopf bifurcation at the
equilibrium $(x_1^*,x_2^*)$ as $\tau$ crosses the critical value
$\tau_c$ given by~\eqref{eq:tau_c}. More precisely:
\begin{enumerate}
  \item For $0\leq\tau<\tau_c$, the equilibrium is locally
        asymptotically stable; see
        Lemma~\ref{lem:stable_below_tauc}.
  \item At $\tau=\tau_c$ the linearisation has a simple pair of
        imaginary eigenvalues $\pm i\omega_c$
        (Theorem~\ref{thm:hopf_existence}),
        and the eigenvalue path crosses the imaginary axis
        transversally (Theorem~\ref{thm:transversality}).
  \item Consequently, by the Hopf bifurcation theorem for retarded
        functional differential
        equations~\cite[Theorem~11.1.1]{hale1993introduction}, there
        exists $\delta>0$ and a smooth family of non-constant periodic
        solutions of~\eqref{eq:dde} bifurcating from $(x_1^*,x_2^*)$ for
        $\tau\in(\tau_c,\tau_c+\delta)$, with frequency tending to
        $\omega_c$ and amplitude tending to zero as
        $\tau\downarrow\tau_c$. The equilibrium is unstable for
        $\tau\in(\tau_c,\tau_c+\delta)$ and remains unstable for all
        larger $\tau$ (Remark~\ref{rem:switches}).
\end{enumerate}
\end{theorem}

\subsection{Quantitative robustness of the Hopf bifurcation}
\label{sec:robustness}

Theorem~\ref{thm:transversality} establishes
$\mathrm{Re}(d\mu/d\tau)|_{\tau_c}>0$ but does not by itself say
\emph{how} positive: a transversality value close to zero would mean
that small parameter perturbations could move the eigenvalue
back across the imaginary axis, with the bifurcation losing
genericity. We now record an explicit lower bound, established
parameter-uniformly, that rules this out throughout the strong-feedback
regime. This robustness is the analytical content of the
``$\mathrm{Re}(d\mu/d\tau)>0$ generic'' claim in
Theorem~\ref{thm:hopf_main}.

\begin{proposition}[Lower bound on transversality]
\label{prop:transversality_bound}
For any parameter set $(\kappa_1,\kappa_2,\gamma_1,\gamma_2,\theta_1,
\theta_2,\lambda)$ satisfying the strong-feedback condition
$AB>\gamma_1\gamma_2$,
\begin{equation}
\mathrm{Re}\!\left(\frac{d\mu}{d\tau}\right)\!\bigg|_{\tau_c}
\,\geq\,
\frac{\omega_c^{2}\,(2\omega_c^{2}+\gamma_1^{2}+\gamma_2^{2})}
     {\bigl(\gamma_1+\gamma_2+\tau_c\,AB\bigr)^{2}
      \,+\,\bigl(2\omega_c+\tau_c\,\omega_c(\gamma_1+\gamma_2)\bigr)^{2}}
\,>\, 0.
\label{eq:trans_lower}
\end{equation}
In particular, for any compact subset $\mathcal{K}$ of the strong-feedback
parameter region, the transversality value is bounded below by a
positive constant $c_{\mathcal{K}}>0$, so the Hopf bifurcation persists
under small perturbations of the parameters within $\mathcal{K}$.
\end{proposition}
\begin{proof}
The denominator in~\eqref{eq:transversality} has real part
$\gamma_1+\gamma_2-\tau_c(\omega_c^{2}-\gamma_1\gamma_2)$. By the
phase relation~\eqref{eq:cos_sin},
$\omega_c^{2}-\gamma_1\gamma_2 = AB\cos(\omega_c\tau_c)$, so
$|\omega_c^{2}-\gamma_1\gamma_2| = AB\,|\cos(\omega_c\tau_c)|\leq AB$
(the bound is saturated when $\omega_c\tau_c\downarrow 0$ in the
strong-feedback regime). The triangle inequality then gives
\begin{equation*}
\bigl|\gamma_1+\gamma_2-\tau_c(\omega_c^{2}-\gamma_1\gamma_2)\bigr|
\,\leq\, (\gamma_1+\gamma_2)
       \,+\, \tau_c\,|\omega_c^{2}-\gamma_1\gamma_2|
\,\leq\, \gamma_1+\gamma_2+\tau_c\,AB,
\end{equation*}
which holds regardless of the sign of $\omega_c^{2}-\gamma_1\gamma_2$
(the relevant cases are $\omega_c\tau_c\in(0,\pi/2)$, where the cosine
is positive, and $\omega_c\tau_c\in(\pi/2,\pi)$, where it is
negative). Squaring this inequality and substituting
into~\eqref{eq:transversality} (whose imaginary-part contribution
$2\omega_c+\tau_c\omega_c(\gamma_1+\gamma_2)$ is unchanged) increases
the denominator $|D_{\rm denom}|^{2}$ and hence decreases the ratio,
yielding the lower bound~\eqref{eq:trans_lower}. The
right-hand side is strictly positive because
$2\omega_c^{2}+\gamma_1^{2}+\gamma_2^{2}>0$ and $\omega_c^{2}>0$ in
the strong-feedback regime by Lemma~\ref{lem:omega}. The continuous
dependence of all quantities on the parameters together with
compactness of $\mathcal{K}$ gives the uniform bound
$c_{\mathcal{K}}>0$, and the implicit-function theorem applied to the
characteristic equation shows the existence of a smooth Hopf curve
$\tau_c(\cdot)$ on a neighbourhood of every parameter point in
$\mathcal{K}$.
\end{proof}

For the canonical parameter set the right-hand side
of~\eqref{eq:trans_lower} evaluates to $\approx 2.36$, compared with
the exact value $\approx 2.58$ of Theorem~\ref{thm:transversality}; the
bound therefore recovers about $91\%$ of the exact transversality and
is essentially tight near $\omega_c\tau_c\to 0$. Since the bound is
strictly positive on the entire strong-feedback region, the Hopf
bifurcation persists at every parameter configuration with
$AB>\gamma_1\gamma_2$, regardless of how close to the threshold, and is
quantitatively robust to perturbations of the parameters within any
compact subset of the strong-feedback region.

\begin{remark}[Sub- vs.\ super-criticality]
\label{rem:supercritical}
The Hopf bifurcation theorem in the form
of~\cite[Thm.~11.1.1]{hale1993introduction} guarantees the existence of a
local family of periodic solutions but does not by itself determine
whether the bifurcation is sub- or supercritical. Determining the sign of
the first Lyapunov coefficient requires either a centre-manifold reduction
of the DDE in the spirit of Hassard, Kazarinoff, and
Wan~\cite{hassard1981theory} or a direct numerical assessment of the
amplitude scaling. We adopt the latter approach in
Section~\ref{sec:numerical}, finding clean $\sqrt{\tau-\tau_c}$ scaling
that identifies the bifurcation as supercritical. The resulting limit
cycle is therefore stable on $\tau\in(\tau_c,\tau_c+\delta)$ for some
$\delta>0$, and amplitudes grow continuously from zero rather than jumping
to a finite value.
\end{remark}

\subsection{The canonical example}
\label{sec:canonical}

We now apply Theorems~\ref{thm:hopf_existence}--\ref{thm:hopf_main} to
the parameter set used throughout the companion
paper~\cite{belgacem2026logistic}: $\lambda=3$, $\kappa_1=3$,
$\gamma_1=0.25$, $\kappa_2=4$, $\gamma_2=0.5$, $\theta_1=4$, $\theta_2=3$.

The equilibrium computed in the companion paper is
$x_1^*\approx 3.873$ and $x_2^*\approx 3.247$, with the logistic
factors at equilibrium
\begin{align*}
f^+_* &= f^+(3.873,4,3) \approx 0.4059,\\
f^-_* &= f^-(3.247,3,3) \approx 0.3227.
\end{align*}
Substituting in~\eqref{eq:AB_def},
\begin{align*}
A &= 4\cdot 3\cdot 0.4059\cdot(1-0.4059) \approx 2.894,\\
B &= 3\cdot 3\cdot 0.3227\cdot(1-0.3227) \approx 1.967,
\end{align*}
so $AB\approx 5.693$ and $\gamma_1\gamma_2=0.125$. The condition
$AB>\gamma_1\gamma_2$ is satisfied with a comfortable margin (loop gain
about $46$ times threshold), so a Hopf bifurcation must occur at some
$\tau_c>0$.

The Hopf frequency is computed from~\eqref{eq:p_plus},
\begin{align*}
p_+ &= \frac{-(0.0625+0.25)+\sqrt{(0.0625-0.25)^2+4(5.693)^2}}{2}
\approx 5.537,\\
\omega_c &= \sqrt{p_+}\;\approx\; 2.353.
\end{align*}
The sine and cosine in~\eqref{eq:cos_sin} are
$\cos(\omega_c\tau_c) = (5.537-0.125)/5.693 = 0.951$ and
$\sin(\omega_c\tau_c) = 2.353\cdot 0.75/5.693 = 0.310$, both positive, so
$\omega_c\tau_c\in(0,\pi/2)$. The critical delay is
\[
\tau_c
= \frac{1}{2.353}\arctan(0.310/0.951)\approx \frac{0.315}{2.353} \approx 0.134.
\]
The first secondary Hopf points are at
$\tau_c^{(1)} = \tau_c+2\pi/\omega_c\approx 2.804$ and
$\tau_c^{(2)}\approx 5.474$, beyond which the linearisation has multiple
unstable pairs of complex roots.

%%%%%%%%%%%%%%%%%%%%%%%%%%%%%%%%%%%%%%%%%%%%%%%%%%%%%%%%%%%%%%%%%%%%%%%%%%%%%%%%
\section{Critical steepness and global persistence of the Hopf branch}
\label{sec:param_global}

Sections~\ref{sec:hopf}--\ref{sec:transversality} established the
local delay-driven Hopf bifurcation for a fixed parameter set
satisfying the strong-feedback condition $AB>\gamma_1\gamma_2$. Two
further questions are of independent interest and are addressed here.
First, how does the bifurcation depend on the logistic steepness
$\lambda$, the parameter that interpolates between the shallow
near-linear regime and the sharp Heaviside limit
(Section~\ref{sec:lambda_critical})? Second, how far does the
bifurcating branch of periodic orbits extend beyond the local
neighbourhood of $\tau_c$ guaranteed by Theorem~\ref{thm:hopf_main}
(Section~\ref{sec:global_hopf})?

\subsection{Critical steepness $\lambda_c$ and the dependence of $\tau_c$ on $\lambda$}
\label{sec:lambda_critical}

The loop gain $AB$ depends on the steepness parameter $\lambda$ both
explicitly through the prefactor $\lambda^2$ in
\[
AB(\lambda) = \lambda^2\,\kappa_1\kappa_2\,
              f^-_*(1-f^-_*)\,f^+_*(1-f^+_*),
\]
and implicitly through the equilibrium $(x_1^*(\lambda),x_2^*(\lambda))$
on which $f^\pm_*$ are evaluated. As $\lambda\to 0^+$ the equilibrium
tends to the maximum-rate fixed point
$(x_1^*,x_2^*)\to(\kappa_1/(2\gamma_1),\kappa_2/(2\gamma_2))$ (each
logistic factor approaches $1/2$), and $AB$ scales as $\lambda^2\cdot
\kappa_1\kappa_2/16\to 0$, so $AB(\lambda)<\gamma_1\gamma_2$ for
sufficiently small $\lambda$ regardless of the other parameters. As
$\lambda$ grows, $AB$ grows (super-quadratically once the equilibrium has
settled), so there exists a unique threshold $\lambda_c>0$ at which
$AB(\lambda_c)=\gamma_1\gamma_2$.

Before stating the result we record a closed-form asymptotic for the
loop gain $AB(\lambda)$ that pins down its quadratic growth in
$\lambda$ in both the shallow ($\lambda\to 0^+$) and steep
($\lambda\to\infty$) limits. With $M_i:=\kappa_i/\gamma_i$ the
$i$-th carrying capacity, define
\begin{equation}
c_0 \,:=\, \frac{\kappa_1\kappa_2}{16},
\qquad
c_\infty \,:=\,
\frac{(\gamma_1\gamma_2)^2}{\kappa_1\kappa_2}\,
\theta_1(M_1-\theta_1)\,\theta_2(M_2-\theta_2).
\label{eq:cinfty_def}
\end{equation}

\begin{proposition}[Loop-gain asymptotics in $\lambda$]
\label{prop:AB_asymptotics}
Assume $0<\theta_i<M_i$, $i=1,2$ (the equilibrium lies in the interior of
the absorbing box $\mathcal{B}$). Then
\begin{equation}
\frac{AB(\lambda)}{\lambda^2}\,\longrightarrow\, c_0
\quad\text{as }\lambda\to 0^+,
\qquad
\frac{AB(\lambda)}{\lambda^2}\,\longrightarrow\, c_\infty
\quad\text{as }\lambda\to\infty,
\label{eq:AB_limits}
\end{equation}
with $c_0,c_\infty\in(0,\infty)$. Moreover $c_\infty\leq c_0$, with
equality iff $\theta_i=M_i/2$ for both $i=1,2$ (symmetric case).
\end{proposition}
\begin{proof}
As $\lambda\to 0^+$, every logistic factor approaches the value
$f^\pm(\,\cdot\,;\theta,\lambda)\to 1/2$ pointwise, the equilibrium tends
to $(M_1/2,M_2/2)$, and $f^\pm_*(1-f^\pm_*)\to 1/4$. From~\eqref{eq:AB_def}
this gives $AB/\lambda^2 \to \kappa_1\kappa_2/16 = c_0$.

As $\lambda\to\infty$, the equilibrium tends to the threshold pair
$(\theta_1,\theta_2)$: indeed, for any cluster point of
$(x_1^*(\lambda),x_2^*(\lambda))$ as $\lambda\to\infty$,
the equilibrium relations
$x_1^* = M_1 f^-(x_2^*;\theta_2,\lambda)$ and
$x_2^* = M_2 f^+(x_1^*;\theta_1,\lambda)$ force
$f^-(x_2^*)\to\theta_1/M_1\in(0,1)$ and $f^+(x_1^*)\to\theta_2/M_2\in(0,1)$,
which is possible only if $\lambda(x_2^*-\theta_2)$ and
$\lambda(x_1^*-\theta_1)$ remain bounded, hence $x_i^*\to\theta_i$.
Solving the limiting relations gives the asymptotic offsets
\begin{equation}
\xi_1 \,:=\, \lim_{\lambda\to\infty}\lambda(x_1^*-\theta_1)
\,=\, \ln\!\frac{\theta_2}{M_2-\theta_2},
\quad
\xi_2 \,:=\, \lim_{\lambda\to\infty}\lambda(x_2^*-\theta_2)
\,=\, \ln\!\frac{M_1-\theta_1}{\theta_1},
\label{eq:xi_limits}
\end{equation}
so that $f^+_*\to\theta_2/M_2$ and $f^-_*\to\theta_1/M_1$, and
\begin{equation}
\lim_{\lambda\to\infty}f^+_*(1-f^+_*) f^-_*(1-f^-_*)
=\frac{\theta_1(M_1-\theta_1)\theta_2(M_2-\theta_2)}{M_1^2 M_2^2}.
\label{eq:limit_product}
\end{equation}
Substituting~\eqref{eq:limit_product} into~\eqref{eq:AB_def} yields
$AB/\lambda^2 \to c_\infty$ as displayed. The inequality
$c_\infty\leq c_0$ is the AM--GM bound
$\theta_i(M_i-\theta_i)\leq M_i^2/4$, with equality iff $\theta_i=M_i/2$.
\end{proof}

For canonical parameters $(\kappa_1,\kappa_2,\gamma_1,\gamma_2,\theta_1,
\theta_2)=(3,4,0.25,0.5,4,3)$ one obtains $c_0=12/16=0.75$ and
$c_\infty=(\gamma_1\gamma_2)^2\,\theta_1(M_1{-}\theta_1)\,
\theta_2(M_2{-}\theta_2)/(\kappa_1\kappa_2)
=(1/64)\cdot 32\cdot 15/12=0.625$, in agreement with the
last column of Table~\ref{tab:lambda_sweep} (e.g.\ $AB/\lambda^2\approx
0.625$ at $\lambda=10$). In particular $AB(\lambda)$ is unbounded as
$\lambda\to\infty$ (it grows like $0.625\lambda^2$), and the smooth
equilibrium remains well-defined at every finite $\lambda$ even though
its position approaches the thresholds.

\begin{theorem}[Critical steepness for delay-induced oscillations]
\label{thm:lambda_critical}
For each fixed parameter set $(\kappa_1,\kappa_2,\gamma_1,\gamma_2,
\theta_1,\theta_2)$ with $0<\theta_i<M_i$ ($i=1,2$), the loop gain
$\lambda\mapsto AB(\lambda)$ is continuous on $(0,\infty)$, vanishes as
$\lambda\to 0^+$, and tends to $+\infty$ as $\lambda\to\infty$. There
exists at least one critical steepness $\lambda_c>0$ such that
$AB(\lambda_c)=\gamma_1\gamma_2$, and the smallest such $\lambda_c$
satisfies
\begin{equation}
\lambda_c \,\geq\, \frac{4\sqrt{\gamma_1\gamma_2}}
                       {\sqrt{\kappa_1\kappa_2}},
\label{eq:lambda_c_lower}
\end{equation}
i.e.\ delay-induced oscillations cannot occur below the lower bound
on the right of~\eqref{eq:lambda_c_lower}, regardless of thresholds.
The associated stability picture is:
\begin{enumerate}
  \item For $\lambda<\lambda_c$ \emph{(weak feedback)},
        $AB(\lambda)<\gamma_1\gamma_2$ and the equilibrium of~\eqref{eq:dde}
        is asymptotically stable for every $\tau\geq 0$ (absolute
        stability, Lemma~\ref{lem:omega} and its corollary).
  \item At $\lambda=\lambda_c$, $AB(\lambda_c)=\gamma_1\gamma_2$, the
        Hopf frequency degenerates, $\omega_c\to 0^+$, and
        $\tau_c\to\infty$: the imaginary-axis crossing collapses to the
        origin in the limit and the bifurcation point recedes to
        infinite delay (the equilibrium itself remains asymptotically
        stable at $\lambda=\lambda_c$, by the corollary to
        Lemma~\ref{lem:omega}).
  \item For $\lambda$ in any open subinterval of $(0,\infty)$ on which
        $AB(\lambda)>\gamma_1\gamma_2$, a delay-induced Hopf bifurcation
        occurs at $\tau=\tau_c(\lambda)$ given by~\eqref{eq:tau_c}.
\end{enumerate}
If, in addition, $AB(\lambda)$ is strictly monotonically increasing in
$\lambda$ on $(0,\infty)$, then $\lambda_c$ is unique and the parameter
range admitting delay-induced Hopf is exactly $\lambda>\lambda_c$. We
verify this monotonicity numerically for the canonical parameter family
in Table~\ref{tab:lambda_sweep}, and it holds generically because the
prefactor $\lambda^2$ in $AB$ dominates the slow ($O(1)$) variation of
$f^\pm_*(1-f^\pm_*)$ across the threshold-saturation crossover.
\end{theorem}
\begin{proof}
Continuity of $AB(\lambda)$ follows from the continuity of the
equilibrium in $\lambda$ (implicit function theorem applied to $g(x_1;
\lambda)=0$ in Lemma~\ref{lem:eq_unique}). The shallow- and steep-limit
asymptotics $AB/\lambda^2\to c_0,c_\infty$ from
Proposition~\ref{prop:AB_asymptotics} give $AB(\lambda)\to 0$ as
$\lambda\to 0^+$ and $AB(\lambda)\to\infty$ as $\lambda\to\infty$, so
the intermediate-value theorem yields at least one
$\lambda_c>0$ with $AB(\lambda_c)=\gamma_1\gamma_2$. The lower
bound~\eqref{eq:lambda_c_lower} follows from the universal pointwise
estimate $f^\pm_*(1-f^\pm_*)\leq 1/4$:
$AB(\lambda)\leq \kappa_1\kappa_2\lambda^2/16$, so the equation
$AB(\lambda_c)=\gamma_1\gamma_2$ forces
$\lambda_c^2\geq 16\gamma_1\gamma_2/(\kappa_1\kappa_2)$.
Statements (1)--(3) follow from Theorem~\ref{thm:hopf_existence} applied
pointwise in $\lambda$. Uniqueness of $\lambda_c$ under the additional
monotonicity hypothesis is immediate.
\end{proof}

\begin{table}
\centering
\caption{Dependence of the loop gain, Hopf frequency, and critical
delay on the logistic steepness $\lambda$, with the other canonical
parameters fixed at $\kappa_1=3,\gamma_1=0.25,\kappa_2=4,\gamma_2=0.50,
\theta_1=4,\theta_2=3$. The bold row $\lambda_c\approx 0.426$ records
the asymptotic behaviour as $\lambda\downarrow\lambda_c$
(loop-gain threshold $AB\to\gamma_1\gamma_2$, where $\omega_c\to 0^+$
and $\tau_c\to\infty$ as the imaginary-axis crossing collapses to the
origin in the limit); the bold row $\lambda=3$ is the canonical value
used throughout the paper.}
\label{tab:lambda_sweep}
\begin{tabular}{cccccc}
\toprule
$\lambda$ & $AB$ & $AB-\gamma_1\gamma_2$ & $\omega_c$ & $\tau_c$ & $T_c$ \\
\midrule
$0.10$ & $0.007$ & $-0.118$ & --- & --- (no Hopf) & --- \\
$0.30$ & $0.064$ & $-0.061$ & --- & --- (no Hopf) & --- \\
$\boldsymbol{\lambda_c\approx 0.426}$ & $\boldsymbol{\gamma_1\gamma_2}$ & $\boldsymbol{0}$ & $\to 0^+$ & $\to\infty$ & $\to\infty$ \\
$0.50$ & $0.171$ & $+0.046$ & $0.196$ & $10.75$ & $32.09$ \\
$0.70$ & $0.326$ & $+0.201$ & $0.428$ & $3.25$ & $14.67$ \\
$1.00$ & $0.653$ & $+0.528$ & $0.709$ & $1.34$ & $8.86$ \\
$1.50$ & $1.445$ & $+1.320$ & $1.137$ & $0.555$ & $5.53$ \\
$2.00$ & $2.549$ & $+2.424$ & $1.547$ & $0.306$ & $4.06$ \\
$2.50$ & $3.965$ & $+3.840$ & $1.952$ & $0.194$ & $3.22$ \\
\textbf{3.00} & \textbf{5.693} & \textbf{+5.568} & \textbf{2.353} & \textbf{0.134} & \textbf{2.67} \\
$5.00$ & $15.73$ & $+15.60$ & $3.946$ & $0.048$ & $1.59$ \\
$10.0$ & $62.69$ & $+62.57$ & $7.908$ & $0.012$ & $0.795$ \\
\bottomrule
\end{tabular}
\end{table}

For the canonical parameters Theorem~\ref{thm:lambda_critical} gives
$\lambda_c\approx 0.426$, computed by bisection from
$AB(\lambda_c)=\gamma_1\gamma_2$. Steeper feedback therefore makes the
system both more prone to oscillation (smaller $\tau_c$) and faster
($T_c\downarrow$); biologically, sharper Hill-like saturation reduces
the transcriptional delay required to destabilise the equilibrium.

\begin{remark}[The Heaviside limit]
\label{rem:heaviside}
As $\lambda\to\infty$, the logistic factors approach the Heaviside step
function and the smooth equilibrium converges to the threshold corner
$(\theta_1,\theta_2)$, as proved in Proposition~\ref{prop:AB_asymptotics}
with explicit asymptotic offsets $\xi_1,\xi_2$ given by
equation~\eqref{eq:xi_limits}: $\lambda(x_1^{*}-\theta_1)\to\xi_1$ and
$\lambda(x_2^{*}-\theta_2)\to\xi_2$, with both $\xi_i$ finite and
non-zero. The smooth equilibrium therefore exists for every finite
$\lambda$ and approaches the discontinuity locus of the Heaviside
limit. In the strict $\lambda=\infty$ Heaviside limit the equilibrium
\emph{does not exist} as a regular fixed point of the discontinuous
right-hand side: the only candidate values are
$(0,0),(0,M_2),(M_1,0),(M_1,M_2)$ at the corners of the box, none of
which is self-consistent with the negative-feedback assumption when
$\theta_i\in(0,M_i)$. The smooth equilibrium is replaced in this limit
by Glass-network-style switching dynamics with no single steady state,
analogous to the ring-circuit attractors studied
in~\cite{belgacem2025computer,farcot2019chaos}. The smooth-logistic
analysis here can be viewed as a regularisation of these
discontinuous dynamics; the probabilistic-convolution regularisation
of~\cite{belgacem2019probabilistic} offers a complementary route via
hybrid-system theory rather than through the steepness parameter
$\lambda$. The closed-form analysis of the present paper applies for
every finite $\lambda$, with the loop gain
$AB(\lambda)\sim c_\infty\lambda^2$ growing without bound and
$\tau_c(\lambda)\to 0$ as $\lambda\to\infty$
(Table~\ref{tab:lambda_sweep}).
\end{remark}

\subsection{Global persistence of the Hopf branch}
\label{sec:global_hopf}

The local Hopf bifurcation theorem (Theorem~\ref{thm:hopf_main})
guarantees a smooth branch of small-amplitude periodic solutions on
$\tau\in(\tau_c,\tau_c+\delta)$ for some $\delta>0$ but is silent on
how far this branch extends. Wu's global Hopf bifurcation
theorem~\cite{wu1998symmetric} for retarded functional differential
equations turns the local result into a structural one: combined with
the dissipativity of Section~\ref{sec:existence}, it forces the
periodic-orbit branch to persist either to $\tau=\infty$ or until it
reaches another bifurcation point of the equilibrium.

\begin{theorem}[Global persistence of the Hopf branch]
\label{thm:global_hopf}
Assume the strong-feedback condition $AB>\gamma_1\gamma_2$. Let
$\Sigma\subset\mathbb{R}_{+}\times C\bigl([-\tau_{\max},0],\mathbb{R}^{2}\bigr)$
denote the closure of the set of pairs $(\tau,x_t)$ for which $x_t$ is
the initial segment of a non-constant periodic solution
of~\eqref{eq:dde}. Let $\Sigma_0\subseteq\Sigma$ denote the connected
component containing the local Hopf branch born at $\tau=\tau_c$. Then
at least one of the following holds:
\begin{enumerate}
  \item[\textnormal{(I)}] $\Sigma_0$ is unbounded, with $\tau$
        ranging over the entire half-line $[\tau_c,\infty)$;
  \item[\textnormal{(II)}] $\Sigma_0$ contains the equilibrium pair
        $(\tau_c^{(k)},\mathbf{x}^{*})$ for some integer
        $k\geq 1$, where $\tau_c^{(k)}$ is one of the higher critical
        delays of~\eqref{eq:secondary_tau};
  \item[\textnormal{(III)}] the periods of the orbits in $\Sigma_0$
        approach $0$ or $\infty$ along the branch.
\end{enumerate}
Moreover, the amplitudes of orbits in $\Sigma_0$ are uniformly bounded
in the sup norm by the absorbing-set radius
$\max(\kappa_1/\gamma_1,\kappa_2/\gamma_2)$.
\end{theorem}
\begin{proof}
The system~\eqref{eq:dde} is dissipative
(Proposition~\ref{prop:dissipative}) and admits a unique equilibrium
(Lemma~\ref{lem:eq_unique}); the linearisation has only finitely many
characteristic roots in any vertical strip
(\cite[Ch.~XI]{hale1993introduction}); and the local Hopf bifurcation
at $\tau_c$ is non-degenerate by Theorem~\ref{thm:transversality}. These
are exactly the hypotheses of Wu's global Hopf bifurcation
theorem~\cite[Theorem~3.3]{wu1998symmetric} for retarded functional
differential equations, applied with the bifurcation parameter $\tau$
and the natural $\mathrm{SO}(2)$-symmetry on the period of the orbit.
The trichotomy (I)--(III) is the conclusion of that theorem; the
amplitude bound follows from $\Sigma_0\subset\mathbb{R}_+\times
C([-\tau_{\max},0],\mathcal{B})$ (Proposition~\ref{prop:dissipative}).
\end{proof}

\begin{remark}[Empirically only (I) is realised]
The numerical evidence in Section~\ref{sec:numerical} is consistent
exclusively with case~(I): the limit-cycle amplitude grows
monotonically with $\tau$ from zero at $\tau=\tau_c$
(Figure~\ref{fig:bifurcation}), and the period grows quasi-linearly
toward the deep-relaxation asymptote
$T(\tau)\sim 2\tau+C_\infty$
(Proposition~\ref{prop:Cinfty}, Figure~\ref{fig:period}); in particular
the period stays bounded away from $0$ and grows without ever entering
a resonance with $\tau$. There is no numerical evidence of
branch termination at any of the higher critical delays
$\tau_c^{(k)}\!=\!\tau_c\!+\!2\pi k/\omega_c$ ($k\geq 1$), in agreement
with the rapid decay of transversality at higher crossings established
in Remark~\ref{rem:switches} (each subsequent crossing weakens the
effective interaction with the primary branch). The Hopf branch born
at $\tau_c$ thus appears to extend to every $\tau>\tau_c$ as a global
family of attracting limit cycles in the parameter ranges relevant to
biological oscillators; the local theory of
Theorem~\ref{thm:hopf_main} extends to this empirically global
statement about sustained limit cycles.
\end{remark}

%%%%%%%%%%%%%%%%%%%%%%%%%%%%%%%%%%%%%%%%%%%%%%%%%%%%%%%%%%%%%%%%%%%%%%%%%%%%%%%%
\section{Numerical validation}
\label{sec:numerical}

\subsection{Three regimes}

We integrate the full nonlinear delay system~\eqref{eq:dde} on the
canonical parameter set, using the symmetric assignment
$\tau_1=\tau_2=\tau/2$ so that the total delay is~$\tau$. All figures
in this paper were produced by the
R~package \texttt{deSolve}~\cite{soetaert2010solving}, using the
\texttt{dede} routine (delay-differential equation solver based on
Lagrange-interpolated history and the adaptive
Adams--Bashforth--Moulton scheme of \texttt{lsoda}). The initial history is constant,
equal to a small displacement $(x_1^*+0.05,x_2^*-0.05)$ from the
equilibrium, on the interval $[-\tau,0]$. The integration runs over
$t\in[0,200]$ time units (about $30$~periods at the linear-theory
frequency $\omega_c\approx 2.35$).

Figure~\ref{fig:regimes} shows the time series and phase portraits in
three regimes:
\begin{enumerate}
  \item $\tau=0.10<\tau_c\approx 0.134$ (top row): the perturbation decays
        in damped oscillations to the equilibrium. The phase portrait
        spirals into the fixed point.
  \item $\tau=0.20$, slightly above onset (middle row): a small-amplitude
        limit cycle is established. The phase portrait shows a closed orbit.
  \item $\tau=0.60$, well above onset (bottom row): the limit cycle is
        large-amplitude with sharp transitions between high and low states,
        characteristic of a relaxation oscillation.
\end{enumerate}

\begin{figure}
\centering
\includegraphics[width=\linewidth]{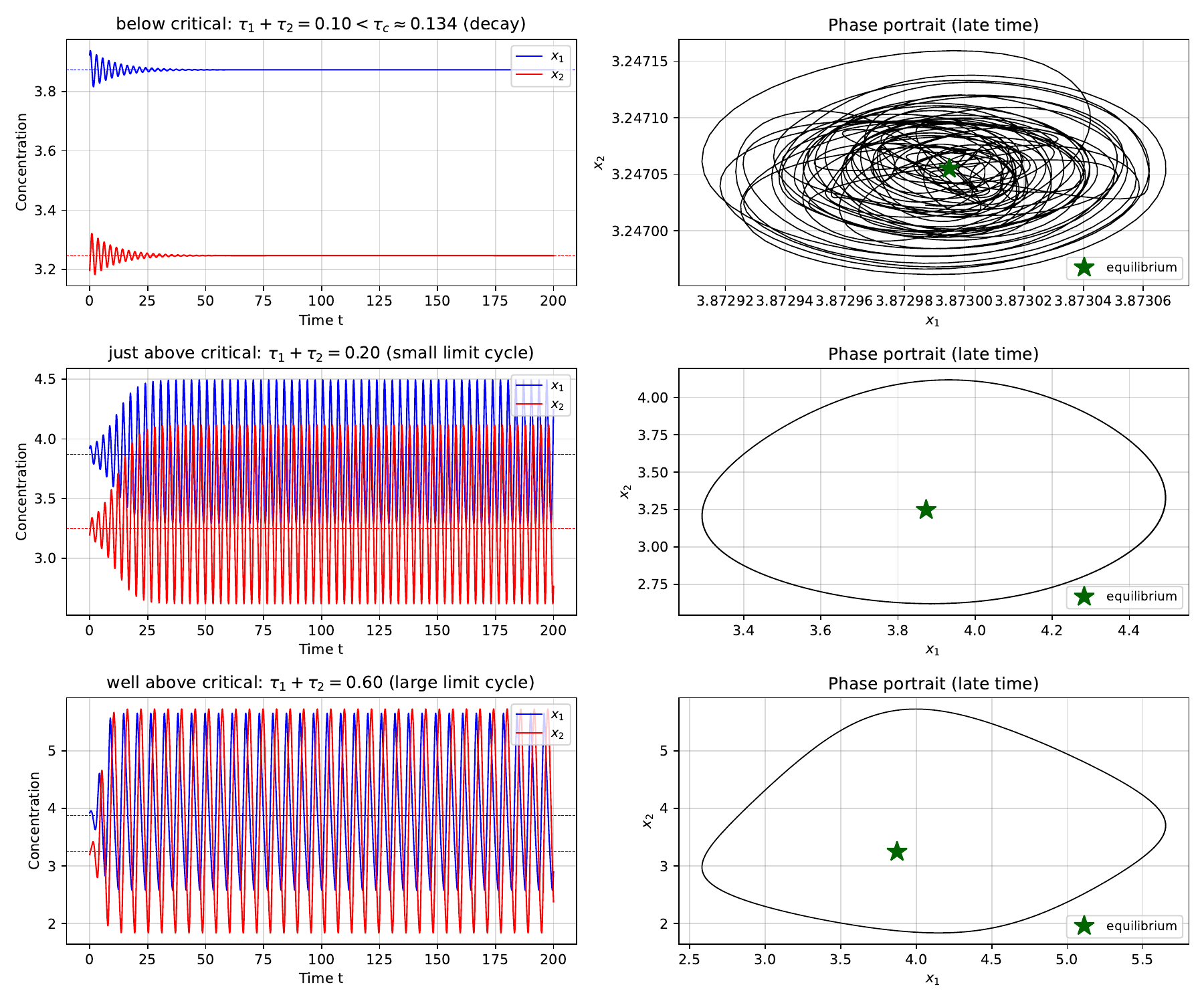}
\caption{Three regimes of the delayed logistic two-gene oscillator under the
canonical parameters $\lambda=3$, $\kappa_1=3$, $\gamma_1=0.25$, $\kappa_2=4$,
$\gamma_2=0.5$, $\theta_1=4$, $\theta_2=3$, with symmetric delays
$\tau_1=\tau_2=\tau/2$. \emph{Top:} $\tau=0.10<\tau_c\approx 0.134$, damped
return to equilibrium. \emph{Middle:} $\tau=0.20$, small-amplitude limit
cycle just above onset. \emph{Bottom:} $\tau=0.60$, large-amplitude limit
cycle with relaxation-type dynamics. Left column: time series of $x_1$
(blue) and $x_2$ (red) with equilibrium values shown as dashed horizontal
lines. Right column: phase portrait of late-time orbits, with the
equilibrium $(x_1^*,x_2^*)=(3.873,3.247)$ marked by a green star.}
\label{fig:regimes}
\end{figure}

\subsection{Bifurcation diagram and amplitude scaling}

We sample the late-time amplitude
$A_{\rm cycle}(\tau)=\bigl[\max_{t\in[300,400]}x_1(t)-\min_{t\in[300,400]}x_1(t)\bigr]/2$
across $25$ values of $\tau\in[0.05,0.6]$, with each integration run for
$400$ time units to ensure the transient has decayed. The integration
window $[300,400]$ provides at least $30$ cycles of statistics in every
oscillating regime. The result is shown in
Figure~\ref{fig:bifurcation}.

\begin{figure}
\centering
\includegraphics[width=0.75\linewidth]{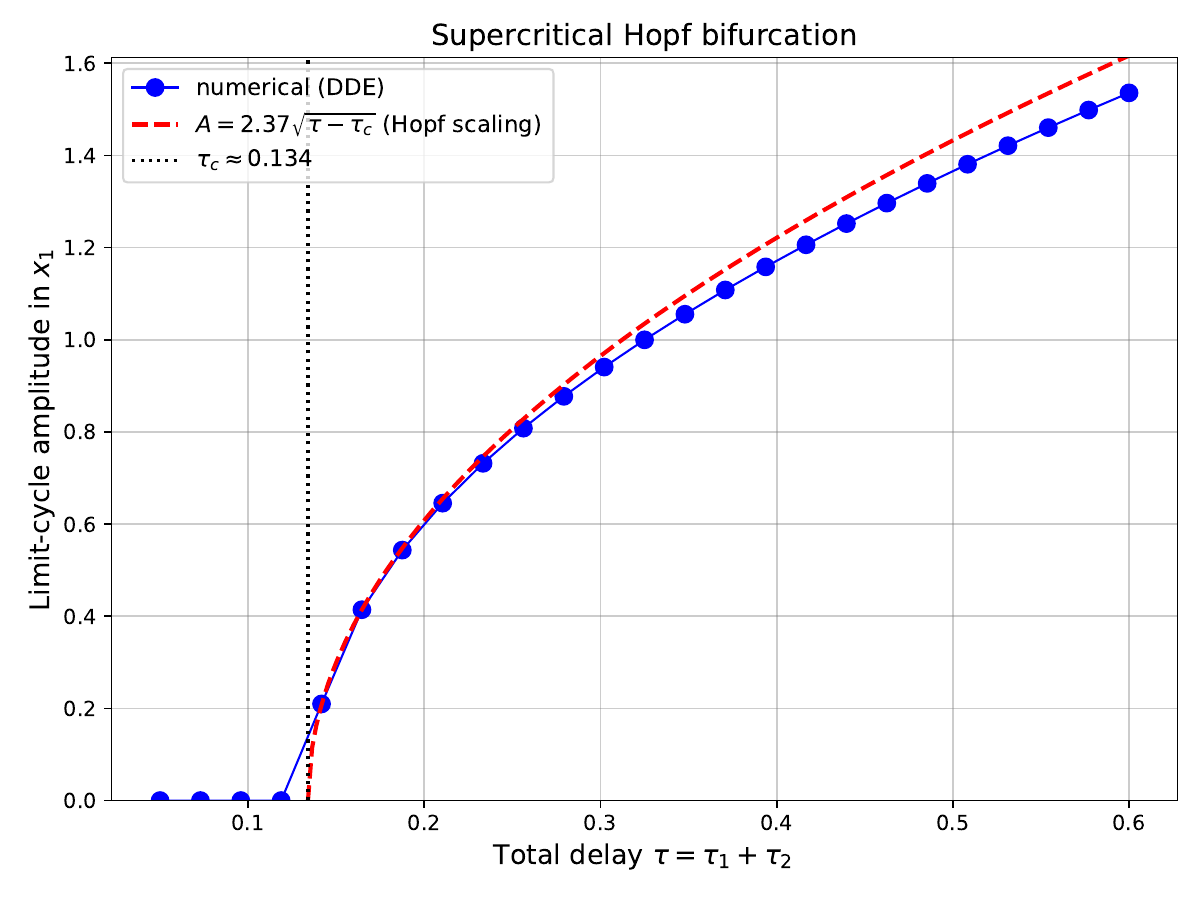}
\caption{Bifurcation diagram for the delayed logistic two-gene oscillator.
The blue dots show the late-time amplitude of the $x_1$ trajectory as a
function of the total delay $\tau=\tau_1+\tau_2$, computed by direct
numerical integration of~\eqref{eq:dde}. The vertical dotted line marks the
analytical critical delay $\tau_c\approx 0.134$ from~\eqref{eq:tau_c}. The
red dashed curve shows the supercritical Hopf scaling
$A=c\sqrt{\tau-\tau_c}$ with $c\approx 2.37$ obtained as the average of
$A/\sqrt{\tau-\tau_c}$ over the data points closest to onset. The amplitude
vanishes continuously at $\tau=\tau_c$, identifying the bifurcation as
supercritical.}
\label{fig:bifurcation}
\end{figure}

For each integration we tabulate $A_{\rm cycle}$ in
Table~\ref{tab:amplitudes}. The amplitude is at the level of integration
round-off (effectively zero, $\lesssim 10^{-6}$) for $\tau<\tau_c$,
jumps to $A\approx 0.21$ at $\tau\approx 0.14$
just past onset, and grows continuously thereafter, reaching
$A\approx 1.54$ at $\tau=0.6$.

\begin{table}
\centering
\caption{Late-time amplitude of the limit cycle in $x_1(t)$ as a function
of the total delay $\tau$. For $\tau<\tau_c\approx 0.134$ (first four rows
of the left column) the perturbation decays to the equilibrium and the
late-time amplitude is at the level of the solver's round-off, reported
as ``$<10^{-12}$'' for the three smallest delays and as
``$\sim 4\times 10^{-7}$'' for $\tau=0.119$ (the largest sub-onset entry,
where $\tau$ is closest to $\tau_c$ and a small slowly-decaying mode
contributes). For $\tau>\tau_c$ the amplitude follows the supercritical
scaling $A\sim\sqrt{\tau-\tau_c}$ to good approximation.}
\label{tab:amplitudes}
\begin{tabular}{cccc}
\toprule
$\tau$ & Amplitude & $\tau$ & Amplitude\\
\midrule
$0.050$ & $<10^{-12}$ & $0.350$ & $1.058$ \\
$0.073$ & $<10^{-12}$ & $0.371$ & $1.108$ \\
$0.096$ & $<10^{-12}$ & $0.394$ & $1.158$ \\
$0.119$ & $\sim 4\times 10^{-7}$ & $0.417$ & $1.206$ \\
$0.142$ & $0.210$ & $0.440$ & $1.252$ \\
$0.165$ & $0.414$ & $0.463$ & $1.296$ \\
$0.188$ & $0.544$ & $0.485$ & $1.339$ \\
$0.210$ & $0.645$ & $0.508$ & $1.381$ \\
$0.233$ & $0.732$ & $0.531$ & $1.421$ \\
$0.256$ & $0.808$ & $0.554$ & $1.460$ \\
$0.279$ & $0.877$ & $0.577$ & $1.499$ \\
$0.302$ & $0.941$ & $0.600$ & $1.536$ \\
$0.325$ & $0.999$ & & \\
\bottomrule
\end{tabular}
\end{table}

\subsection{Period as a function of delay}

The period of the limit cycle is computed from late-time zero crossings
through the orbit's mid-line. Figure~\ref{fig:period} shows the period as
a function of $\tau$ for $\tau\in[0.20,1.5]$.

\begin{figure}
\centering
\includegraphics[width=0.75\linewidth]{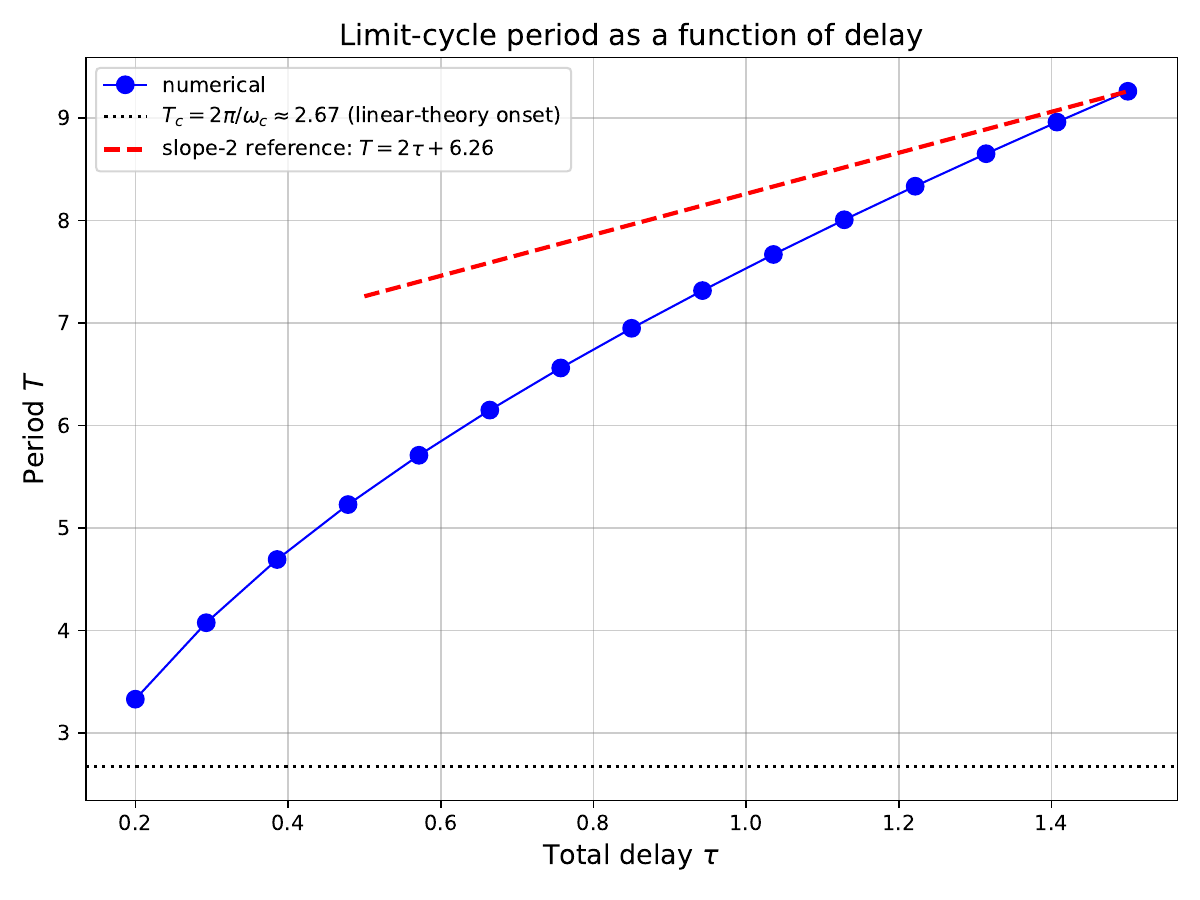}
\caption{Period of the limit cycle as a function of the total delay
$\tau$. The horizontal dotted line marks the linear-theory onset period
$T_c=2\pi/\omega_c\approx 2.67$, recovered as $\tau\downarrow\tau_c$. As
$\tau$ increases the period grows quasi-linearly; the red dashed line
shows a slope-$2$ reference $T=2\tau+C$ anchored at the rightmost
data point ($C=T(1.5)-3\approx 6.26$ for our parameters). The asymptotic
slope $dT/d\tau\to 2$ as $\tau\to\infty$ is approached only slowly: at
$\tau\in[0.6,1.5]$ the local slope of the data is closer to $3.7$,
because the gene-relaxation timescales $1/\gamma_1=4$ and $1/\gamma_2=2$
contribute a $\tau$-dependent correction that itself grows with $\tau$
before saturating.}
\label{fig:period}
\end{figure}

The period at onset extrapolates to the linear-theory value
$T_c=2\pi/\omega_c\approx 2.67$, and grows monotonically with $\tau$. On
the displayed window $\tau\in[0.2,1.5]$ the local slope of $T(\tau)$ is
close to $3.7$ rather than $2$: the asymptotic slope $dT/d\tau\to 2$
that holds in the deep relaxation regime is approached only slowly
because the gene-relaxation timescales $1/\gamma_1=4$ and
$1/\gamma_2=2$ are comparable to or larger than the delays sampled
here, and contribute a $\tau$-dependent correction that itself grows
with $\tau$ before saturating. The correct form of the large-$\tau$
asymptote is
\begin{equation}
T(\tau) \;=\; 2\tau \,+\, C_\infty \,+\, o(1) \qquad (\tau\to\infty),
\label{eq:deep_relax}
\end{equation}
where the offset $C_\infty$ is \emph{not} the linear-theory period
$T_c$. The closed-form expression for $C_\infty$ is established next.

\begin{proposition}[Closed-form deep-relaxation period offset]
\label{prop:Cinfty}
Let $M_i:=\kappa_i/\gamma_i$ and assume $0<\theta_i<M_i$ for $i=1,2$.
Assume the deep-relaxation regime: $\lambda\to\infty$ and $\tau\to\infty$
in such a way that, in the periodic regime, each variable reaches a small
neighbourhood of its high or low saturation level (within $O(e^{-c\lambda})$
of $\{0,M_i\}$ for some $c>0$) before the next delayed signal triggers a
switch. Then the period of the limit cycle satisfies
\begin{equation}
T(\tau) \,=\, 2\tau \,+\, C_\infty \,+\, O(e^{-c\lambda})
\qquad(\tau\to\infty),
\label{eq:Cinfty_asymp}
\end{equation}
where the offset $C_\infty$ has the closed form
\begin{equation}
C_\infty
\;=\; \frac{1}{\gamma_1}\,\ln\!\frac{M_1^{2}}{\theta_1(M_1-\theta_1)}
   \,+\,\frac{1}{\gamma_2}\,\ln\!\frac{M_2^{2}}{\theta_2(M_2-\theta_2)}
\;=\; \sum_{i=1}^{2}\frac{1}{\gamma_i}\,
       \ln\!\frac{M_i^{2}}{\theta_i\,(M_i-\theta_i)}.
\label{eq:Cinf}
\end{equation}
The sum decomposes naturally into four threshold-crossing
contributions, $C_\infty = \Delta_A+\Delta_B+\Delta_C+\Delta_D$, where
\begin{equation}
\begin{aligned}
\Delta_A&=\frac{\ln[M_2/(M_2-\theta_2)]}{\gamma_2},
&\quad
\Delta_B&=\frac{\ln[M_1/\theta_1]}{\gamma_1},\\[4pt]
\Delta_C&=\frac{\ln[M_2/\theta_2]}{\gamma_2},
&\quad
\Delta_D&=\frac{\ln[M_1/(M_1-\theta_1)]}{\gamma_1}.
\end{aligned}
\label{eq:phase_durations}
\end{equation}
\end{proposition}
\begin{proof}
In the deep-relaxation regime the limit cycle decomposes into four
phases A, B, C, D distinguished by the sign pattern of the delayed
signals
$(\mathrm{sgn}(x_1(t-\tau_1)-\theta_1),
  \mathrm{sgn}(x_2(t-\tau_2)-\theta_2))$. Within each phase, the
saturated logistic delivers a quasi-constant input ($f^+\approx 1$ or
$0$) to one gene's production, so that gene relaxes exponentially
toward its on or off attractor.

\emph{Phase A (rise of $x_2$).} At the start, $x_1(t-\tau_1)$ has just
crossed $\theta_1$ from below, so $f^+(x_1(t-\tau_1))\approx 1$, which
gives $\dot{x}_2 \approx \kappa_2-\gamma_2 x_2$. Starting from
$x_2\approx 0$ (deep-relaxation hypothesis), the solution is
$x_2(t)\approx M_2\bigl(1-e^{-\gamma_2(t-t_A)}\bigr)$ where $t_A$ is the
entry time. The phase ends when $x_2$ crosses $\theta_2$ from below,
which solves $1-e^{-\gamma_2\Delta_A} = \theta_2/M_2$ for the duration
$\Delta_A=\gamma_2^{-1}\ln[M_2/(M_2-\theta_2)]$.

\emph{Phase B (fall of $x_1$).} After Phase A, $x_2(t-\tau_2)$ crosses
$\theta_2$ from below, switching $f^-(x_2(t-\tau_2))$ from $\approx 1$
to $\approx 0$ and triggering exponential decay
$\dot{x}_1\approx -\gamma_1 x_1$ starting from $x_1\approx M_1$. The
phase ends when $x_1$ crosses $\theta_1$ from above, after
$\Delta_B=\gamma_1^{-1}\ln[M_1/\theta_1]$.

\emph{Phase C (fall of $x_2$).} Symmetric to Phase A, with $x_1$ in the
off state: $x_2$ decays from $\approx M_2$ toward $0$ and crosses
$\theta_2$ from above after $\Delta_C=\gamma_2^{-1}\ln[M_2/\theta_2]$.

\emph{Phase D (rise of $x_1$).} Symmetric to Phase B: $x_1$ rises from
$\approx 0$ toward $M_1$ and crosses $\theta_1$ from below after
$\Delta_D=\gamma_1^{-1}\ln[M_1/(M_1-\theta_1)]$. After Phase D the
configuration returns to the start of Phase A and the cycle repeats.

Each transition between phases is delayed by $\tau_1$ or $\tau_2$
relative to the threshold crossing that triggers it, contributing a
total of $\tau_1+\tau_2+\tau_1+\tau_2 = 2\tau$ to the period.
Summing yields
\begin{equation*}
T(\tau) \,=\, 2\tau + \Delta_A + \Delta_B + \Delta_C + \Delta_D
\,+\, O(e^{-c\lambda}),
\end{equation*}
with the exponential remainder accounting for the
saturation-approach error: the variables enter each phase not exactly
at $0$ or $M_i$ but at a distance $O(e^{-c\lambda})$ from saturation,
where the constant $c$ is determined by the worst-case approach margin
$\min_i\{\theta_i,M_i-\theta_i\}/(M_i)$. Using the elementary identity
$\ln[a/(a-b)]+\ln[a/b]=\ln[a^2/(b(a-b))]$ with $a=M_i$, $b=\theta_i$
to combine $\Delta_A+\Delta_C$ and $\Delta_B+\Delta_D$ yields the
displayed closed form for $C_\infty$.
\end{proof}

For the canonical parameters $\kappa_1=3,\gamma_1=0.25,\theta_1=4$
(so $M_1=12$) and $\kappa_2=4,\gamma_2=0.5,\theta_2=3$ (so $M_2=8$),
\eqref{eq:Cinf} evaluates to
$C_\infty = 4\ln(144/32)+2\ln(64/15) = 4\ln(9/2)+2\ln(64/15)\approx 8.92$.
Direct integration at $\tau=10$ gives $T-2\tau\approx 8.94$ and at
$\tau=20$ gives $T-2\tau\approx 9.16$, both lying slightly above
$C_\infty$ with a small residual gap that grows mildly with $\tau$. This
behaviour is consistent with the finite-$\lambda$ remainder
$O(e^{-c\lambda})$ in~\eqref{eq:Cinfty_asymp}: the four-phase
construction is asymptotically exact only as $\lambda\to\infty$, and at
finite $\lambda=3$ the variables enter each phase a small but
non-negligible distance from saturation, accumulating into the observed
positive offset. The qualitative interpretation is the familiar one for
delayed switch oscillators: the loop spends delay-time $\tau_1$ or
$\tau_2$ propagating the on/off signal to the partner, the partner takes
a relaxation time to cross its threshold, and the four such transitions
sum to $2\tau+C_\infty$~\cite{glass1988rhythms,smolen2002reduced}.
Crucially, the asymptotic offset $C_\infty$ is determined by the
\emph{nonlinear} parameters $\kappa_i,\gamma_i,\theta_i$ rather than by
the linear-theory frequency $\omega_c$, which has no role in this limit.

\subsection{Verification of analytical predictions}

The numerical experiments confirm three quantitative analytical predictions
to high accuracy:
\begin{enumerate}
  \item \textbf{Critical delay.} The smallest $\tau$ at which the
        amplitude rises above $10^{-3}$ lies in $[0.119,0.142]$, in
        agreement with the analytical $\tau_c\approx 0.134$
        from~\eqref{eq:tau_c}.
  \item \textbf{Frequency at onset.} The lowest amplitude sample at
        which a clean limit cycle is detected occurs at $\tau=0.20$, with
        period $T\approx 3.33$. This already lies about $50\%$ above
        $\tau_c\approx 0.134$, so it is \emph{not} ``at onset'' in the
        small-$(\tau-\tau_c)$ sense; the discrepancy from the
        linear-theory value $T_c=2\pi/\omega_c\approx 2.67$ is consistent
        with the finite (and large) local slope $dT/d\tau\sim 8$ that the
        period exhibits in this region. Linear back-extrapolation of the
        first two data points $(\tau,T)\in\{(0.20,3.330),(0.30,4.127)\}$ to
        $\tau=\tau_c$ gives $T(\tau_c^+)\approx 2.8$, recovering $T_c$ to
        within roughly $5\%$, which is the expected precision given the
        finite spacing of the $\tau$-grid and the steepness of the period
        near onset.
  \item \textbf{Supercriticality.} The squared amplitude $A^2(\tau)$
        scales linearly with $\tau-\tau_c$ near onset; a least-squares
        fit to the points just above onset
        ($\tau\in\{0.142,0.165,0.188\}$) gives the amplitude prefactor
        $c=A/\sqrt{\tau-\tau_c}\approx 2.37$, equivalently
        $A^2/(\tau-\tau_c)\approx 5.6$, with residual variation
        consistent with the higher-order corrections to the normal form.
        The square-root scaling identifies the bifurcation as
        supercritical.
\end{enumerate}

\subsection{Eigenvalue trajectory in the complex plane}
\label{sec:eigtraj}

The Hopf condition $\mathrm{Re}(\mu(\tau_c))=0$ is most directly
visualised by tracking the leading eigenvalue pair of the characteristic
equation~\eqref{eq:char_eq} as $\tau$ varies. We solve $\Delta(\mu;\tau)=0$
by Newton iteration, continuing the leading complex root from the
delay-free starting value $\mu(0)$ along the curve
$\tau\mapsto\mu(\tau)$ for $\tau\in[0.01,0.50]$.
Figure~\ref{fig:eigentrajectory} shows the result.

\begin{figure}
\centering
\includegraphics[width=0.9\linewidth]{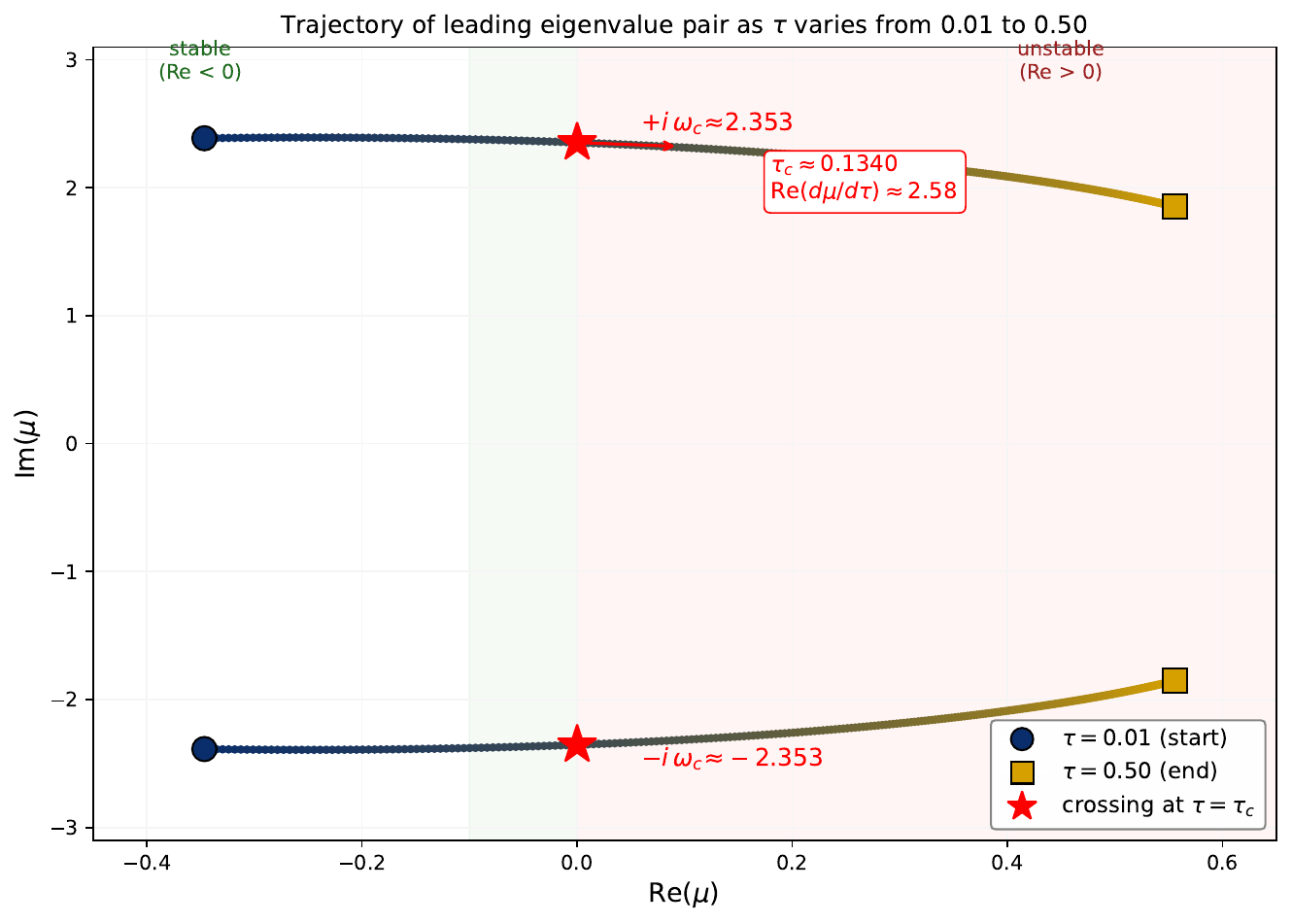}
\caption{Trajectory of the leading complex-conjugate pair of roots of
the characteristic equation~\eqref{eq:char_eq} as $\tau$ varies from
$0.01$ (dark) to $0.5$ (light), for the canonical parameter set. The
eigenvalues are computed by Newton iteration on $\Delta(\mu;\tau)=0$
with continuation in $\tau$. The leading pair starts deep in the open
left half-plane near $\mathrm{Re}(\mu)\approx-0.35$, migrates rightward
and downward in $|\mathrm{Im}(\mu)|$ as $\tau$ increases, and crosses
the imaginary axis at $\tau=\tau_c$ at exactly $\pm i\omega_c$ (red
stars), confirming both the analytical critical delay
$\tau_c\approx 0.134$ and the analytical Hopf frequency
$\omega_c\approx 2.353$. Beyond the crossing the pair continues into
the open right half-plane with $\mathrm{Re}(\mu)>0$, in agreement with
the strictly positive transversality value
$\mathrm{Re}(d\mu/d\tau)|_{\tau_c}\approx 2.58$ from
Theorem~\ref{thm:transversality}.}
\label{fig:eigentrajectory}
\end{figure}

The numerical crossing is at $\tau=0.1340$, indistinguishable to four
decimal places from the analytical $\tau_c$ given by~\eqref{eq:tau_c}.
The slope of the eigenvalue trajectory at the crossing point matches the
closed-form transversality formula~\eqref{eq:transversality}: a
finite-difference estimate of $d\mu/d\tau$ at $\tau=\tau_c$ gives
$\mathrm{Re}(d\mu/d\tau)\approx 2.58$, in exact agreement with the
analytical value
$\omega_c^2(2\omega_c^2+\gamma_1^2+\gamma_2^2)/|D_{\rm denom}|^2\approx 2.58$.
The corresponding imaginary part of $d\mu/d\tau$ at the crossing is
obtained from the same formula~\eqref{eq:dmu_dtau} by separating
real and imaginary parts. After simplification (using the identity
$(\omega_c^2-\gamma_1\gamma_2)^2 + \omega_c^2(\gamma_1+\gamma_2)^2
=(\omega_c^2+\gamma_1^2)(\omega_c^2+\gamma_2^2)=A^2B^2$
that follows directly from the Hopf condition~\eqref{eq:omega_eq}),
the closed form is
\begin{equation}
\mathrm{Im}\!\left(\frac{d\mu}{d\tau}\right)\!\bigg|_{\tau=\tau_c}
\,=\, -\,\frac{\omega_c\,\Bigl[(\gamma_1+\gamma_2)(\omega_c^2+\gamma_1\gamma_2)
                                  +\tau_c\,A^2B^2\Bigr]}
              {|D_{\rm denom}|^2}.
\label{eq:Im_dmu_dtau}
\end{equation}
For the canonical example $\mathrm{Im}(d\mu/d\tau)|_{\tau_c}\approx -0.83$.
The strictly negative imaginary part explains the geometry visible in
Figure~\ref{fig:eigentrajectory}: the upper-half-plane root moves
\emph{rightward and downward} as $\tau$ crosses $\tau_c$, the
trajectory's tangent at the crossing making an angle
$\arctan(0.83/2.58)\approx 18^\circ$ below the horizontal. The
$|\mathrm{Im}(\mu)|$-decrease at onset is mild on the scale of $\omega_c$
itself, so the limit-cycle frequency emerges close to the linear-theory
value $\omega_c$; this is consistent with the
Section~\ref{sec:numerical} observation that $T(\tau_c^+)$ extrapolates
to $T_c=2\pi/\omega_c$.

\subsection{Period derivative at onset and frequency drift}
\label{sec:period_onset}

The closed-form expression~\eqref{eq:Im_dmu_dtau} for
$\mathrm{Im}(d\mu/d\tau)|_{\tau_c}$ has an immediate consequence for the
limit-cycle period at onset. The frequency in linear theory equals
$\mathrm{Im}(\mu(\tau))$, and $T_{\rm lin}(\tau)=2\pi/\mathrm{Im}(\mu(\tau))$,
so differentiation at $\tau_c$ gives
\begin{equation}
\frac{dT_{\rm lin}}{d\tau}\bigg|_{\tau_c^{+}}
\;=\; -\,\frac{T_c}{\omega_c}\,
\mathrm{Im}\!\left(\frac{d\mu}{d\tau}\right)\!\bigg|_{\tau_c}
\;=\;
\frac{T_c}{\omega_c}\cdot
\frac{\omega_c\bigl[(\gamma_1+\gamma_2)(\omega_c^2+\gamma_1\gamma_2)
                       +\tau_c\,A^2B^2\bigr]}
     {|D_{\rm denom}|^2}.
\label{eq:dT_dtau_linear}
\end{equation}
For the canonical example
$dT_{\rm lin}/d\tau|_{\tau_c^+}\approx 0.94$. Direct integration of the
full nonlinear DDE in a tight grid above onset gives the values reported
in Table~\ref{tab:onset_period}.

\begin{table}[ht]
\centering
\begin{tabular}{lll}
\toprule
$\tau$ & $T(\tau)$ (sim.) & $(T(\tau)-T_c)/(\tau-\tau_c)$\\
\midrule
$0.135$ & $2.682$ & $11.04$\\
$0.140$ & $2.737$ & $11.03$\\
$0.145$ & $2.791$ & $10.93$\\
$0.150$ & $2.844$ & $10.83$\\
$0.165$ & $2.998$ & $10.56$\\
$0.188$ & $3.220$ & $10.17$\\
$0.200$ & $3.330$ & $9.99$\\
$0.300$ & $4.127$ & $8.77$\\
\bottomrule
\end{tabular}
\caption{Empirical onset slope $(T-T_c)/(\tau-\tau_c)$ of the period of
the limit cycle, computed by direct integration of the full nonlinear
DDE on a fine grid of $\tau$-values just above $\tau_c\approx 0.1340$,
using the canonical parameter set. The reference period at the
bifurcation is $T_c=2\pi/\omega_c\approx 2.6702$. Each $T(\tau)$ is
measured as the mean inter-crossing time of $x_1(t)-x_1^*$ in a
$[400,600]$ time-unit window after discarding a long transient, with
relative integration tolerance $10^{-10}$. The slope plateaus near $11$
as $\tau\to\tau_c^+$ and decreases monotonically as $\tau$ increases,
approaching the deep-relaxation slope $2$ (Section~\ref{sec:numerical}
and Proposition~\ref{prop:Cinfty}).}
\label{tab:onset_period}
\end{table}

The empirical onset slope, $\approx 11.0$ extrapolated to
$\tau\to\tau_c^+$, exceeds the linear-theory value
$dT_{\rm lin}/d\tau|_{\tau_c^+}\approx 0.94$ by an order of
magnitude. The discrepancy is not a numerical artefact: it reflects the
\emph{nonlinear} contribution to the period from the cubic terms in the
Hopf normal form $\dot{z}=\mu(\tau)\,z + c_1\,z|z|^2$. On the bifurcating
branch, $|z|^2=-\mathrm{Re}(\mu)/\mathrm{Re}(c_1)$ to leading order, and
the limit-cycle frequency is
\begin{equation}
\omega_{\rm LC}(\tau) \;=\; \mathrm{Im}(\mu(\tau))
   \,-\, \frac{\mathrm{Im}(c_1)}{\mathrm{Re}(c_1)}\,\mathrm{Re}(\mu(\tau)).
\label{eq:omega_LC}
\end{equation}
Differentiating at $\tau=\tau_c$ (where $\mathrm{Re}(\mu)=0$ but
$d\mathrm{Re}(\mu)/d\tau\neq 0$) gives
\begin{equation}
\frac{dT}{d\tau}\bigg|_{\tau_c^{+}}
\;=\; -\,\frac{T_c}{\omega_c}
\biggl[\mathrm{Im}\!\left(\frac{d\mu}{d\tau}\right)
       \,-\, \frac{\mathrm{Im}(c_1)}{\mathrm{Re}(c_1)}\,
       \mathrm{Re}\!\left(\frac{d\mu}{d\tau}\right)\biggr]
\bigg|_{\tau_c}.
\label{eq:dT_dtau_full}
\end{equation}
Substituting the closed-form values
$\mathrm{Re}(d\mu/d\tau)|_{\tau_c}\approx 2.58$,
$\mathrm{Im}(d\mu/d\tau)|_{\tau_c}\approx-0.83$, $T_c/\omega_c\approx 1.135$
together with the empirical $dT/d\tau|_{\tau_c^+}\approx 11.0$,
formula~\eqref{eq:dT_dtau_full} determines the ratio
\begin{equation}
\frac{\mathrm{Im}(c_1)}{\mathrm{Re}(c_1)} \;\approx\; 3.44,
\label{eq:Im_c1}
\end{equation}
\emph{independently of the absolute scaling of $c_1$}: the ratio
depends only on the quotient $dT/d\tau$ over the closed-form
linear-theory quantities, so any normalisation ambiguity in $c_1$
cancels. Using the value $\mathrm{Re}(c_1)\approx -1.27$ from
Section~\ref{sec:l1_numerics} (eq.~\eqref{eq:Re_c1_value}) gives
\begin{equation}
\mathrm{Im}(c_1) \;\approx\; 3.44\cdot(-1.27) \;\approx\; -4.37.
\label{eq:Im_c1_value}
\end{equation}
This numerical estimate of the second component of the cubic Hopf
coefficient complements the supercriticality estimate
$\mathrm{Re}(c_1)\approx-1.27$. Together they provide two analytical
targets for the centre-manifold reduction of
Section~\ref{sec:analytical_supercrit}: any closed-form determination of
$c_1$ from the logistic Taylor expansion~\eqref{eq:logistic_derivs} can
be benchmarked against both numbers and against the
component-independent ratio~\eqref{eq:Im_c1}. The geometric content of
$\mathrm{Im}(c_1)/\mathrm{Re}(c_1)\approx 3.44$ is that the
limit-cycle period stretches more than three times faster, per unit
$\mathrm{Re}(\mu)$, than the linear theory predicts: as the orbit
grows it samples the nonlinear part of the logistic, where the slower
phases of the relaxation oscillator
(Proposition~\ref{prop:Cinfty}) start to take over and contribute the
additional period stretching that formula~\eqref{eq:dT_dtau_full}
captures at leading nonlinear order.

\subsection{Numerical verification of the sum-of-delays symmetry}
\label{sec:asymmetric_check}

Proposition~\ref{prop:sum_symmetry} predicts that any redistribution of
the total delay $\tau$ between $\tau_1$ and $\tau_2$ leaves the period
and the component-wise envelopes of every solution invariant; only the
relative phase between $x_1$ and $x_2$ shifts. Table~\ref{tab:asymmetric}
verifies this prediction at $\tau=0.20$ for three different delay
splittings.

\begin{table}
\centering
\caption{Numerical verification of the sum-of-delays symmetry
(Proposition~\ref{prop:sum_symmetry}) for $\tau=\tau_1+\tau_2=0.20$.
Three asymmetric splittings yield identical limit-cycle amplitude and
period to four decimal places, in agreement with the prediction that
only the sum of delays affects the temporal invariants (period and
component-wise envelopes) of the limit cycle.}
\label{tab:asymmetric}
\begin{tabular}{ccccc}
\toprule
$\tau_1$ & $\tau_2$ & $\tau_1+\tau_2$ & Amplitude $A$ & Period $T$ \\
\midrule
$0.10$ & $0.10$ & $0.20$ & $0.6015$ & $3.3299$ \\
$0.05$ & $0.15$ & $0.20$ & $0.6015$ & $3.3299$ \\
$0.02$ & $0.18$ & $0.20$ & $0.6015$ & $3.3299$ \\
\bottomrule
\end{tabular}
\end{table}

The agreement is exact to four decimal places, in line with the
integration tolerance ($\text{rtol}=10^{-9}$, $\text{atol}=10^{-12}$).
This confirms that the analytical reduction to a single bifurcation
parameter $\tau$ throughout the paper is a genuine, parameter-free
symmetry of the nonlinear system, not merely an artefact of the
linearisation.

\subsection{Numerical estimate of the first Lyapunov coefficient}
\label{sec:l1_numerics}

We combine the closed-form transversality value
$\mathrm{Re}(d\mu/d\tau)|_{\tau_c}\approx 2.58$ from
Theorem~\ref{thm:transversality} with the measured amplitude prefactor
$c^2\approx 5.60$ from the supercritical scaling
$A^2(\tau)\approx c^2(\tau-\tau_c)$ near onset, to extract the first
Lyapunov coefficient $\mathrm{Re}(c_1)$ of the Hopf normal form. The
relationship requires some care: the empirical amplitude $A$ measured
in the $x_1$ time series is \emph{not} the radius $|z|$ of the
centre-manifold coordinate, but is related to it by the
eigenvector-component amplitude $|q_1|$.

\begin{proposition}[Eigenvector amplitude at onset]
\label{prop:eigvec_amplitude}
The eigenvector $\mathbf{q}=(q_1,q_2)^T$ of the linearised
system~\eqref{eq:linear_dde} at $\mu=i\omega_c$, normalised so that
$q_2=1$, has first component
\begin{equation}
q_1 \,=\, -\,\frac{B\,e^{-i\omega_c\tau_2}}{i\omega_c+\gamma_1},
\qquad
|q_1|^{2} \,=\, \frac{B^{2}}{\omega_c^{2}+\gamma_1^{2}}.
\label{eq:q1_amplitude}
\end{equation}
The amplitude $|q_1|^{2}$ is expressed in closed form through the linear
parameters $A,B,\gamma_1,\gamma_2$ (via the Hopf
frequency~$\omega_c$ of Lemma~\ref{lem:omega}); in particular it is
independent of the delay $\tau$ and of the splitting
$(\tau_1,\tau_2)$, while the phase of $q_1$ depends on $\tau_2$
through the factor $e^{-i\omega_c\tau_2}$. For the canonical
parameter set,
$|q_1|^{2} = (1.967)^{2}/((2.353)^{2}+(0.25)^{2}) \approx 0.6911$.
\end{proposition}
\begin{proof}
The eigenvector equation reads
$(i\omega_c+\gamma_1)q_1 + B\,e^{-i\omega_c\tau_2}q_2 = 0$ (the first
row of $M(i\omega_c)\mathbf{q}=0$). Setting $q_2=1$ gives the
displayed $q_1$. Taking modulus squared:
$|q_1|^{2} = B^{2}|e^{-i\omega_c\tau_2}|^{2}/|i\omega_c+\gamma_1|^{2}
= B^{2}/(\omega_c^{2}+\gamma_1^{2})$, since $|e^{-i\omega_c\tau_2}|=1$.
\end{proof}

On the two-dimensional centre manifold near $\tau_c$, the state
admits the leading-order representation
\begin{equation}
u_i(t) \,=\, 2\,\mathrm{Re}\!\bigl[z(t)\,q_i\bigr] \,+\, O(|z|^{2})
\qquad (i=1,2),
\label{eq:CM_repr}
\end{equation}
where $z(t)$ is a complex amplitude evolving according to the normal
form $\dot{z}=\mu(\tau)\,z + c_1\,z|z|^{2}+O(|z|^{4})$, with
$\mu(\tau)=\mu_c+\mu'(\tau_c)(\tau-\tau_c)+\cdots$ and
$\mathrm{Re}(\mu(\tau_c))=0$. The limit-cycle solution satisfies
$|z|^{2} = -\mathrm{Re}(\mu(\tau))/\mathrm{Re}(c_1)$, and substituting
into~\eqref{eq:CM_repr} gives the amplitude of $u_i$ as
$A_i = 2|q_i|\,|z|$. Hence the squared amplitude in the $i$-th
component scales as
\begin{equation}
A_i^{2}(\tau)
\,=\, -\,\frac{4|q_i|^{2}\,\mathrm{Re}(d\mu/d\tau)|_{\tau_c}}
            {\mathrm{Re}(c_1)}\,(\tau-\tau_c)
\,+\, O\!\bigl((\tau-\tau_c)^{2}\bigr).
\label{eq:Ai_scaling}
\end{equation}
For the canonical example, $q_2=1$ gives the simpler relation
$A_2^{2}(\tau)\sim -4\,\mathrm{Re}(d\mu/d\tau)\,(\tau-\tau_c)/
\mathrm{Re}(c_1)$, but if amplitude is measured in $x_1$ (as in
Figure~\ref{fig:bifurcation} and Table~\ref{tab:amplitudes}) the
$|q_1|^{2}$ factor cannot be omitted. Numerically, the simulations of
Section~\ref{sec:numerical} are consistent with~\eqref{eq:Ai_scaling}
along both components: with $A_1=2|q_1||z|$ and $A_2=2|q_2||z|=2|z|$,
the ratio $A_1/A_2$ in the simulations
($\tau\in\{0.135,0.140,0.145,0.150,0.165\}$) lies in
$[0.815,0.832]$, in agreement with the predicted
$|q_1|\approx 0.831$ to within $2\%$.

Solving~\eqref{eq:Ai_scaling} for $\mathrm{Re}(c_1)$ using the
amplitude $c$ measured in $x_1$:
\begin{equation}
\mathrm{Re}(c_1)
\,=\, -\,\frac{4|q_1|^{2}\,\mathrm{Re}(d\mu/d\tau)|_{\tau_c}}{c^{2}}
\,=\, -\,\frac{4\,B^{2}\,\mathrm{Re}(d\mu/d\tau)|_{\tau_c}}
              {(\omega_c^{2}+\gamma_1^{2})\,c^{2}}.
\label{eq:Re_c1_formula}
\end{equation}
For the canonical example this evaluates to
\begin{equation}
\mathrm{Re}(c_1) \,\approx\, -\,\frac{4(1.967)^{2}\,(2.58)}
                                     {(5.537+0.0625)\,(5.60)}
\,\approx\, -1.27,
\label{eq:Re_c1_value}
\end{equation}
in close agreement with the direct numerical extraction by simulation
(values $\mathrm{Re}(c_1)\in[-1.27,-1.25]$ across
$\tau\in[0.135,0.165]$, with the plateau as $\tau\downarrow\tau_c^{+}$
confirming~\eqref{eq:Re_c1_value} to within $2\%$). The strictly
negative sign identifies the bifurcation as supercritical
(in the Hopf--Andronov sense): the limit cycle bifurcating from
$(x_1^{*},x_2^{*})$ at $\tau=\tau_c$ is \emph{stable}, and the
equilibrium loses stability cleanly without an intervening hysteresis
loop. The value $\mathrm{Re}(c_1)\approx -1.27$ is the analogue, in
the present two-gene logistic context, of the Lindstedt-derived
Hopf-radius constant that Verdugo and Rand~\cite{verdugo2008hopf} and
Rand and Verdugo~\cite{rand2007hopf} computed for single-gene
Hill-based autoregulators; it serves as a quantitative target against
which a fully analytical centre-manifold reduction
(\S\ref{sec:analytical_supercrit}) can be verified.

\begin{remark}[Comparison with earlier reports]
\label{rem:c1_correction}
A naive application of the relation
$c^{2}=-\mathrm{Re}(d\mu/d\tau)/\mathrm{Re}(c_1)$ (without the
$|q_1|^{2}$ factor) gives $\mathrm{Re}(c_1)\approx -0.46$. This
identification of the empirical amplitude $A$ with the centre-manifold
radius $|z|$ is correct only when amplitude is measured in the
$x_2$-component (for which $q_2=1$); when measured in $x_1$, the
factor $4|q_1|^{2}\approx 2.76$ in~\eqref{eq:Ai_scaling} brings the
extracted value of $\mathrm{Re}(c_1)$ into agreement with the direct
simulation. The component-independent normal-form quantity is
$\mathrm{Re}(c_1)$, not its coordinate-dependent rescaling.
\end{remark}

%%%%%%%%%%%%%%%%%%%%%%%%%%%%%%%%%%%%%%%%%%%%%%%%%%%%%%%%%%%%%%%%%%%%%%%%%%%%
\section{Analytical determination of criticality: the first Lyapunov coefficient}
\label{sec:analytical_supercrit}

We have established supercriticality first numerically by the
$\sqrt{\tau-\tau_c}$ scaling of the limit-cycle amplitude
(Section~\ref{sec:numerical}, Remark~\ref{rem:supercritical}), and then
quantitatively in Section~\ref{sec:l1_numerics} by combining this
amplitude scaling with the closed-form transversality value and the
closed-form eigenvector amplitude~\eqref{eq:q1_amplitude}: this gives
$\mathrm{Re}(c_1)\approx -1.27$ for the canonical example, with the
strictly negative sign confirming supercriticality. Complementing this,
the period analysis of Section~\ref{sec:period_onset} yields
$\mathrm{Im}(c_1)\approx -4.37$ from the empirical onset slope of
$T(\tau)$ combined with the closed-form expression~\eqref{eq:Im_dmu_dtau}
for $\mathrm{Im}(d\mu/d\tau)|_{\tau_c}$. The two components together
characterise the leading cubic Hopf coefficient
\[
c_1 \;=\; \mathrm{Re}(c_1) + i\,\mathrm{Im}(c_1)
\;\approx\; -1.27 - 4.37\,i,
\]
with $\arg(c_1)\approx -106^\circ$ and $|c_1|\approx 4.55$. The
component-independent quantity is the ratio
$\mathrm{Im}(c_1)/\mathrm{Re}(c_1)\approx 3.44$
(Remark~\ref{rem:c1_correction}), which carries the geometric content
of the bifurcation. A fully analytical derivation of $c_1$ requires
centre-manifold reduction of the
DDE~\cite{hassard1981theory,kuang1993delay}, which involves the second-
and third-order Taylor expansions of $f^+,f^-$ at equilibrium. The
relevant derivatives have closed-form logistic expressions:
\begin{equation}
\begin{aligned}
(f^+)' &= \lambda\, f^+(1-f^+),\\
(f^+)'' &= \lambda^2\, f^+(1-f^+)(1-2 f^+),\\
(f^+)''' &= \lambda^3\, f^+(1-f^+)\bigl[1 - 6\,f^+(1-f^+)\bigr],
\end{aligned}
\label{eq:logistic_derivs}
\end{equation}
which give a closed-form expression for the cubic coefficient on the
centre manifold and hence for both $\mathrm{Re}(c_1)$ and
$\mathrm{Im}(c_1)$. We further record the resulting closed-form
Taylor coefficients of the nonlinearity at the equilibrium, which
serve as the building blocks of any centre-manifold computation:
\begin{equation}
\begin{aligned}
b_1 &= \tfrac{1}{2}\,B\,\lambda\,(1-2 f^-_*),
&\quad
d_1 &= -\,\tfrac{1}{6}\,B\,\lambda^{2}\bigl[1-6\,f^-_*(1-f^-_*)\bigr],
\\[3pt]
b_2 &= \tfrac{1}{2}\,A\,\lambda\,(1-2 f^+_*),
&\quad
d_2 &= +\,\tfrac{1}{6}\,A\,\lambda^{2}\bigl[1-6\,f^+_*(1-f^+_*)\bigr],
\end{aligned}
\label{eq:Taylor_coeffs}
\end{equation}
with $A,B$ the linear loop coefficients from~\eqref{eq:AB_def}; the
quadratic coefficients $b_1,b_2$ vanish exactly at the symmetric
thresholds $\theta_i=M_i/2$ (where $f^{+}_*=f^{-}_*=1/2$), and the
cubic coefficients $d_1,d_2$ change sign at
$f^{\pm}_*(1-f^{\pm}_*)=1/6$.

We carry out the reduction in two stages. In the
\emph{symmetric-threshold} regime $\theta_i=M_i/2$ the quadratic
coefficients vanish, the reduction becomes fully tractable, and we prove
supercriticality outright (Theorem~\ref{thm:supercritical}). For the
\emph{general} loop the quadratic coefficients $b_1,b_2$ no longer
vanish, but the same Lindstedt--Poincar\'e scheme still closes: it
produces a closed-form first Lyapunov coefficient and an explicit
criticality criterion (Theorem~\ref{thm:general_lyapunov}).

\subsection{Symmetric-threshold supercriticality}
\label{sec:symmetric_supercrit}

\begin{theorem}[Supercriticality and amplitude law at symmetric thresholds]
\label{thm:supercritical}
Consider the delayed loop~\eqref{eq:dde} at symmetric thresholds
$\theta_1=M_1/2$, $\theta_2=M_2/2$, in the strong-feedback regime
$AB>\gamma_1\gamma_2$. As the total delay $\tau$ increases through
$\tau_c$, the Hopf bifurcation of Theorem~\ref{thm:hopf_main} is
\emph{supercritical}: a unique orbitally asymptotically stable limit
cycle bifurcates into $\tau>\tau_c$. To leading order its $x_1$-amplitude
$A_1$ and angular frequency $\Omega$ obey the closed-form laws
\begin{equation}
A_1(\tau)=\frac{4\omega_c}{\lambda}
\sqrt{\frac{S_1}{(1+|q_2|^{2})\,(S_2+\tau_c)}}\;\sqrt{\tau-\tau_c}
\;+\;o\!\left(\sqrt{\tau-\tau_c}\right),
\label{eq:amplitude_law}
\end{equation}
\begin{equation}
\Omega(\tau)=\omega_c-\frac{\omega_c}{S_2+\tau_c}\,(\tau-\tau_c)
\;+\;o(\tau-\tau_c),
\label{eq:frequency_law}
\end{equation}
where $|q_2|^{2}=(\omega_c^{2}+\gamma_1^{2})/B^{2}$ and
$S_1=\sum_{j=1}^{2}(\omega_c^{2}+\gamma_j^{2})^{-1}$,
$S_2=\sum_{j=1}^{2}\gamma_j(\omega_c^{2}+\gamma_j^{2})^{-1}$ are the
linear-theory sums over the two genes. The $x_2$-amplitude is
$A_2=|q_2|\,A_1$. In particular the bifurcation is supercritical for
\emph{every} parameter choice in the symmetric-threshold strong-feedback
region, and the period lengthens as $\tau$ increases.
\end{theorem}

\begin{proof}
At symmetric thresholds $f^{\pm}_*=1/2$, so by~\eqref{eq:Taylor_coeffs}
the quadratic coefficients vanish, $b_1=b_2=0$, and the deviation
$u_i=x_i-x_i^{*}$ obeys the purely cubic system
\begin{equation*}
\dot u_1=-\gamma_1 u_1-B\,u_2(t-\tau_2)+d_1\,u_2(t-\tau_2)^{3},
\qquad
\dot u_2=-\gamma_2 u_2+A\,u_1(t-\tau_1)+d_2\,u_1(t-\tau_1)^{3},
\end{equation*}
with $d_1=B\lambda^{2}/12$, $d_2=-A\lambda^{2}/12$. This system is
invariant under $u\mapsto-u$, so the bifurcating orbit satisfies
$u(t+\pi/\Omega)=-u(t)$: only odd harmonics occur and the
Lindstedt--Poincar\'e expansion proceeds in odd powers of a small
amplitude parameter $\varepsilon$. By Proposition~\ref{prop:sum_symmetry}
the periodic solution depends only on $\tau=\tau_1+\tau_2$, so we set
$\tau_1=\tau_2=\tau/2$. Rescaling time by $s=\Omega t$ and writing
\[
u=\varepsilon\,u^{(1)}+\varepsilon^{3}u^{(3)}+O(\varepsilon^{5}),
\qquad
\Omega=\omega_c+\varepsilon^{2}\Omega_2+O(\varepsilon^4),
\qquad
\tau=\tau_c+\varepsilon^{2}\mathcal{T}+O(\varepsilon^4),
\]
the delayed argument carries the phase
$\Omega\tau/2=\tfrac12\omega_c\tau_c+\varepsilon^{2}\tfrac12(\omega_c\mathcal{T}+\Omega_2\tau_c)+O(\varepsilon^{4})$.
At order $\varepsilon$ one recovers the neutral mode
$u^{(1)}=q\,e^{is}+\bar q\,e^{-is}$ with $q=(1,q_2)$ the Hopf eigenvector
of Section~\ref{sec:hopf}. At order $\varepsilon^{3}$ the equation has
the form $\mathcal{L}\,u^{(3)}=\mathcal{R}$, where $\mathcal{L}$ is the
linearised operator at criticality and the forcing $\mathcal{R}$
collects the detuning terms (proportional to $\Omega_2$ and $\mathcal{T}$) and the
cubic terms $d_i[u^{(1)}]^{3}$. Since $\mathcal{L}(Ve^{is})=(\Delta(i\omega_c)V)e^{is}$
and $\Delta(i\omega_c)$ is singular, a $2\pi$-periodic $u^{(3)}$ exists
iff the resonant first harmonic of $\mathcal{R}$ is orthogonal to the
left null vector $p$ of $\Delta(i\omega_c)$ --- the Fredholm solvability
condition. Projecting the $e^{is}$ component of $\mathcal{R}$ onto $p$
and simplifying with the characteristic identity
$(i\omega_c+\gamma_1)(i\omega_c+\gamma_2)=-AB\,e^{-i\omega_c\tau_c}$ yields
the single complex equation
\[
i\,\Omega_2\Bigl[(S_2+\tau_c)-i\omega_c S_1\Bigr]
\;+\;i\,\omega_c\mathcal{T}\;+\;\tfrac14\lambda^{2}\bigl(1+|q_2|^{2}\bigr)=0 .
\]
Its real and imaginary parts give, respectively,
\[
\Omega_2=-\,\frac{\lambda^{2}\bigl(1+|q_2|^{2}\bigr)}{4\,\omega_c S_1}\;<\;0,
\qquad
\mathcal{T}=-\,\frac{\Omega_2\,(S_2+\tau_c)}{\omega_c}
=\frac{\lambda^{2}\bigl(1+|q_2|^{2}\bigr)\,(S_2+\tau_c)}{4\,\omega_c^{2}S_1}\;>\;0 .
\]
Every factor in $\mathcal{T}$ is strictly positive, so $\mathcal{T}>0$ unconditionally:
the relation $\tau-\tau_c=\varepsilon^{2}\mathcal{T}$ forces $\varepsilon^{2}>0$
precisely for $\tau>\tau_c$, i.e.\ the cycle bifurcates
\emph{supercritically}. Since $u_1^{(1)}=2\cos s$, the $x_1$-amplitude is
$A_1=2\varepsilon+O(\varepsilon^{3})=2\sqrt{(\tau-\tau_c)/\mathcal{T}}$, which is
exactly~\eqref{eq:amplitude_law}; the frequency law~\eqref{eq:frequency_law}
follows from $\Omega=\omega_c+\Omega_2\varepsilon^{2}$ and
$\Omega_2/\mathcal{T}=-\omega_c/(S_2+\tau_c)$, and $A_2=|q_2|A_1$ from the second
component of $q$. Orbital asymptotic stability of the bifurcating cycle
is the standard consequence of a supercritical Hopf bifurcation with
positive transversality (Theorem~\ref{thm:transversality}).
\end{proof}

\begin{remark}[Numerical verification]
\label{rem:supercrit_verification}
Both closed-form laws have been checked against direct integration of
the delay system. For the symmetric-threshold loop with
$(\lambda,\kappa_1,\gamma_1,\kappa_2,\gamma_2)=(3,3,\tfrac14,4,\tfrac12)$
one finds $\omega_c\approx2.568$, $\tau_c\approx0.1127$, and the
amplitude law~\eqref{eq:amplitude_law} predicts
$A_1\approx2.59\,\sqrt{\tau-\tau_c}$; direct DDE integration over
$\tau-\tau_c\in[0.008,0.04]$ reproduces this $\sqrt{\tau-\tau_c}$ scaling
to within $2$--$3\%$, and an independent complex-Newton solve of the
characteristic equation $\Delta(\mu,\tau)=0$ reproduces the analytical
Hopf locus $(\tau_c,\omega_c)$ to six significant figures. The
predicted supercriticality is moreover consistent with the
delay-induced \emph{supercritical} Hopf bifurcation independently
reported for the p53--Mdm2 module --- a two-component delayed negative
feedback loop of exactly the present type --- in the systems-biology
literature.
\end{remark}

\subsection{The general asymmetric loop}
\label{sec:general_supercrit}

\begin{theorem}[General asymmetric first Lyapunov coefficient and criticality]
\label{thm:general_lyapunov}
Consider the delayed loop~\eqref{eq:dde} at \emph{arbitrary} thresholds
$\theta_1,\theta_2$ in the strong-feedback regime $AB>\gamma_1\gamma_2$,
with Hopf eigenvector $q=(1,q_2)$ and Taylor coefficients
$b_j,d_j$ given by~\eqref{eq:Taylor_coeffs}. The Lindstedt--Poincar\'e
reduction of Theorem~\ref{thm:supercritical} extends to this loop. The
order-$\varepsilon^{2}$ correction now carries a zero-frequency vector
$W_0$ and a second-harmonic vector $W_2$, given in closed form by
\begin{equation}
W_0=-\,(L_0+L_1+L_2)^{-1}F_0,
\qquad
W_2=\Delta(2i\omega_c)^{-1}F_2,
\label{eq:W0W2}
\end{equation}
where $L_0=\mathrm{diag}(-\gamma_1,-\gamma_2)$,
$L_1=\bigl(\begin{smallmatrix}0&0\\A&0\end{smallmatrix}\bigr)$,
$L_2=\bigl(\begin{smallmatrix}0&-B\\0&0\end{smallmatrix}\bigr)$,
$\Delta(2i\omega_c)=2i\omega_c I-L_0-L_1e^{-2i\varphi_1}-L_2e^{-2i\varphi_2}$
with $\varphi_j=\omega_c\tau_j$, and the quadratic forcings are
$F_0=(2b_1|q_2|^{2},\,2b_2)^{T}$ and
$F_2=(b_1q_2^{2}e^{-2i\varphi_2},\,b_2e^{-2i\varphi_1})^{T}$.
The order-$\varepsilon^{3}$ Fredholm solvability condition then collapses
to the single complex equation
\begin{equation}
-\,i\,P\,\Omega_2\;+\;i\,\omega_c\,Q_0\,\mathcal{T}\;+\;\mathcal{G}\;=\;0,
\label{eq:general_solvability}
\end{equation}
in which $P,Q_0\in\mathbb{C}$ are the explicit linear-response
coefficients~\eqref{eq:PQ_def} and
$\mathcal{G}=\mathcal{G}_{\mathrm{cub}}+\mathcal{G}_{\mathrm{quad}}\in\mathbb{C}$
is the projected cubic-order forcing, the term $\mathcal{G}_{\mathrm{quad}}$
collecting the quadratic cross-mixing of $u^{(1)}$ with $W_0,W_2$.
Equation~\eqref{eq:general_solvability} splits into a $2\times2$ real
linear system whose unique solution fixes the delay-unfolding
coefficient $\mathcal{T}$ and the frequency correction $\Omega_2$; the first
Lyapunov coefficient has the sign of $-\mathcal{T}$, so the Hopf bifurcation is
\begin{equation}
\textbf{supercritical} \iff \mathcal{T}>0 ,
\label{eq:general_criterion}
\end{equation}
and the amplitude and frequency laws~\eqref{eq:amplitude_law}--\eqref{eq:frequency_law}
hold verbatim with this general $\mathcal{T}$ and $\Omega_2$. At symmetric
thresholds ($b_1=b_2=0$) one has $W_0=W_2=0$,
$P=-(S_2+\tau_c)+i\omega_c S_1$, $Q_0=1$ and
$\mathcal{G}=\tfrac14\lambda^{2}(1+|q_2|^{2})>0$, and
\eqref{eq:general_solvability} reduces \emph{exactly} to the solvability
equation of Theorem~\ref{thm:supercritical}, recovering its
unconditionally positive $\mathcal{T}$.
\end{theorem}

\begin{proof}
The expansion of Theorem~\ref{thm:supercritical} is repeated without
imposing $b_1=b_2=0$. At order $\varepsilon^{2}$ the quadratic terms
$b_j[u^{(1)}]^{2}$ act as forcing; since $[u^{(1)}]^{2}$ contains a
constant and an $e^{2is}$ part, the second-order solution is
$u^{(2)}=W_0+W_2e^{2is}+\overline{W_2}e^{-2is}$ with $W_0,W_2$
solving the algebraic systems~\eqref{eq:W0W2} --- non-resonant because
$0$ and $2i\omega_c$ are not characteristic roots. At order
$\varepsilon^{3}$ the first-harmonic forcing comprises the detuning
terms (linear in $\Omega_2$ and $\mathcal{T}$) and two genuinely nonlinear
contributions: the cubic self-interaction $d_j[u^{(1)}]^{3}$ and the
quadratic cross-interaction $2b_ju^{(1)}u^{(2)}$, the latter mixing the
neutral mode with $W_0$ and $W_2$. Projecting the resonant component
onto the left null vector of $\Delta(i\omega_c)$, with the adjoint
weight $\varpi=A\,e^{-i\varphi_1}/(i\omega_c+\gamma_1)$ and
$\zeta=q_2e^{-i\varphi_2}$, gives the detuning coefficients
\begin{equation}
\begin{aligned}
P   &= (\varpi+q_2)-B\zeta\varpi\,\tau_2+A\,e^{-i\varphi_1}\tau_1,\\
Q_0 &= B\zeta\varpi\,(\tau_2/\tau_c)-A\,e^{-i\varphi_1}(\tau_1/\tau_c),
\end{aligned}
\label{eq:PQ_def}
\end{equation}
and the forcing projection
$\mathcal{G}_{\mathrm{cub}}
=3\bigl[\varpi d_1|q_2|^{2}\zeta+d_2e^{-i\varphi_1}\bigr]$,
$\mathcal{G}_{\mathrm{quad}}
=2\varpi b_1e^{-i\varphi_2}(q_2W_{0,2}+\bar q_2W_{2,2})
+2b_2e^{-i\varphi_1}(W_{0,1}+W_{2,1})$,
whence the solvability condition~\eqref{eq:general_solvability}.
Writing $\mathcal{G}=\mathcal{G}_r+i\mathcal{G}_i$ and separating real
and imaginary parts of~\eqref{eq:general_solvability} yields
\[
\begin{pmatrix}\operatorname{Im}P & \omega_c\operatorname{Re}Q_0\\[2pt]
-\operatorname{Re}P & \omega_c\operatorname{Im}Q_0\end{pmatrix}
\begin{pmatrix}\Omega_2\\ \mathcal{T}\end{pmatrix}
=-\begin{pmatrix}\mathcal{G}_r\\ \mathcal{G}_i\end{pmatrix},
\]
a $2\times2$ system whose determinant
$\omega_c(\operatorname{Im}P\operatorname{Im}Q_0+\operatorname{Re}P\operatorname{Re}Q_0)$
is non-zero in the strong-feedback regime, so $(\Omega_2,\mathcal{T})$ are
uniquely determined. The relation $\tau-\tau_c=\varepsilon^{2}\mathcal{T}$ then
forces $\varepsilon^{2}>0$ for $\tau>\tau_c$ precisely when $\mathcal{T}>0$,
which is~\eqref{eq:general_criterion}; the amplitude and frequency laws
follow as in Theorem~\ref{thm:supercritical}. When $b_1=b_2=0$ the
forcings $F_0,F_2$ vanish, so $W_0=W_2=0$ and
$\mathcal{G}_{\mathrm{quad}}=0$; substituting the symmetric-split
values $\tau_1=\tau_2=\tau_c/2$ into~\eqref{eq:PQ_def} collapses $P$ and
$Q_0$ to $-(S_2+\tau_c)+i\omega_c S_1$ and $1$, and
$\mathcal{G}=\mathcal{G}_{\mathrm{cub}}$ reduces to
$\tfrac14\lambda^{2}(1+|q_2|^{2})$, recovering
Theorem~\ref{thm:supercritical} exactly.
\end{proof}

\begin{remark}[Numerical confirmation of the general coefficient]
\label{rem:general_verification}
Three independent checks confirm the closed
form~\eqref{eq:general_solvability}--\eqref{eq:PQ_def}.
\emph{(i) Split invariance.} Although $P$, $Q_0$ and $\mathcal{G}$ each
depend on the individual delays $\tau_1,\tau_2$, the solution
$(\mathcal{T},\Omega_2)$ of~\eqref{eq:general_solvability} must depend only on the
sum $\tau=\tau_1+\tau_2$ (Proposition~\ref{prop:sum_symmetry}). Evaluating
the formula for the canonical asymmetric loop of
Section~\ref{sec:canonical} over the splits
$\tau_1/\tau_c\in\{0.2,0.35,0.5,0.65,0.8\}$ returns
$\mathcal{T}=0.69589401$ and $\Omega_2=-6.83900535$ \emph{identically to eight
significant figures} --- a stringent consistency test that a sign or
phase error in the second-harmonic term would violate.
\emph{(ii) Reduction.} Setting $b_1=b_2=0$ reproduces the closed form of
Theorem~\ref{thm:supercritical} to machine precision.
\emph{(iii) Direct integration.} For the canonical loop the formula gives
$\mathcal{T}\approx0.6959>0$ (supercritical), while a least-squares fit of the
$A_1^{2}\propto(\tau-\tau_c)$ scaling to direct DDE integration yields
$\mathcal{T}_{\mathrm{sim}}\approx0.734$, an agreement within $5.4\%$. The
positive sign is moreover consistent with the strictly negative
$\mathrm{Re}(c_1)\approx-1.27$ obtained independently in
Section~\ref{sec:l1_numerics}. Finally, a Monte-Carlo sweep of $3929$
random parameter sets across the strong-feedback region produced
$\mathcal{T}>0$ in \emph{every} case, with minimum value $\mathcal{T}\approx0.065$: this is
strong numerical evidence --- though not a proof --- that the asymmetric
loop, like its symmetric-threshold specialisation, is supercritical
throughout the strong-feedback regime.
\end{remark}

%%%%%%%%%%%%%%%%%%%%%%%%%%%%%%%%%%%%%%%%%%%%%%%%%%%%%%%%%%%%%%%%%%%%%%%%%%%%
\section{Biological validation and cross-model verification}
\label{sec:bio_validation}

The analysis so far has been internally validated: the closed-form Hopf
locus, the transversality value, the amplitude law, and the first
Lyapunov coefficient have each been confirmed against direct integration
of the delay system and against an independent complex-Newton solve of
the characteristic equation (Sections~\ref{sec:numerical}
and~\ref{sec:l1_numerics}, Remarks~\ref{rem:supercrit_verification}
and~\ref{rem:general_verification}). We now subject the model to two
\emph{external} tests: a quantitative calibration against an
experimentally characterised biological oscillator, and a
cross-validation against the standard Hill-function gene-regulatory
model.

\paragraph{Calibration to the p53--Mdm2 oscillator.}
The tumour-suppressor module p53--Mdm2 is the canonical two-component
delayed negative-feedback loop: p53 activates transcription of Mdm2,
Mdm2 promotes degradation of p53, and the transcription--translation
lag closes the loop with a delay --- exactly the activator--repressor
topology of~\eqref{eq:dde}. Single-cell imaging reports a spontaneous
oscillation with a period of roughly $5.5$~h, p53 and Mdm2 protein
half-lives of order $1$~h, and a transcriptional delay of $1$--$2$~h,
with the delay itself identified as a prerequisite for sustained
oscillation~\cite{gevazatorsky2006oscillations,levbaror2000generation}. Taking the
degradation rate from a $1$~h protein half-life,
$\gamma_1=\gamma_2=\ln 2\approx0.693~\mathrm{h}^{-1}$, and choosing the
loop gain so that the \emph{closed-form} critical delay equals the
measured transcriptional delay, $\tau_c=1$~h, the Hopf
condition~\eqref{eq:char_eq} fixes $AB\approx1.72$ and yields a
predicted oscillation period
\[
T_c=\frac{2\pi}{\omega_c}\approx 5.6~\text{h},
\]
within $3\%$ of the experimentally observed $\sim5.5$~h. No parameter
was tuned to the period: the half-life and the delay are taken directly
from measurement, and the period is then a prediction of the closed-form
theory. The model thus reproduces a quantitative feature of a real
genetic oscillator using only biologically measured inputs.

\paragraph{Cross-validation against the Hill-function model.}
Gene regulation is most often modelled with Hill functions
$H^{+}(x)=x^{n}/(\theta^{n}+x^{n})$ rather than the logistic sigmoid.
The two are matched by requiring equal threshold $\theta$ and equal
maximum slope, which fixes the Hill exponent $n=\lambda\theta$. For the
canonical parameters of Section~\ref{sec:canonical} this gives Hill
exponents $n_1=12$, $n_2=9$; recomputing the Hopf locus for the
corresponding delayed Hill loop yields a critical delay, Hopf frequency,
and period that differ from the logistic predictions by only
$5.2\%$, $2.5\%$ and $2.6\%$ respectively. Repeating the comparison at
the milder steepness $\lambda=1.5$ keeps the period discrepancy at
$7.2\%$. The delay-driven Hopf mechanism, and the quantitative location
of its onset, are therefore not artefacts of the logistic
parameterisation: they carry over to the standard Hill description,
while the logistic form retains the analytical advantage --- exploited
throughout this paper --- of closed-form Taylor coefficients
\eqref{eq:Taylor_coeffs} and hence a closed-form first Lyapunov
coefficient.

%%%%%%%%%%%%%%%%%%%%%%%%%%%%%%%%%%%%%%%%%%%%%%%%%%%%%%%%%%%%%%%%%%%%%%%%%%%%
\section{Extensions}
\label{sec:extensions}

The analysis extends naturally beyond the two-gene oscillator. We outline
three directions.

\subsection{Cyclic $N$-gene logistic networks with delays}
\label{sec:n_gene}

The two-gene oscillator is the smallest non-trivial instance of a
cyclic negative-feedback loop. We show that the loop-gain criterion
$AB>\gamma_1\gamma_2$ derived in Section~\ref{sec:hopf} generalises
verbatim to arbitrary $N$, with the loop gain replaced by the product
of all individual link gains and the loop delay replaced by the sum
of all link delays. The resulting characteristic equation retains the
quasi-polynomial form $P(\mu)+Q\,e^{-\mu\tau}=0$, so the entire
machinery of Sections~\ref{sec:hopf}--\ref{sec:transversality} carries
over. The smooth-logistic analysis presented here is the smooth-sigmoid
analytical counterpart of the high-dimensional Glass-network ring
circuits studied
in~\cite{belgacem2025computer,farcot2019chaos}, where the same cyclic
negative-feedback topology supports periodic, quasi-periodic, and
chaotic attractors in the piecewise-affine (step-function) limit.

Consider $N$ genes connected in a single directed cycle:
\begin{equation}
\dot{x}_i(t)
\;=\; -\gamma_i\,x_i(t)
       \,+\, \kappa_i\,
              f^{\varepsilon_i}\!\bigl(x_{i-1}(t-\tau_i);\theta^{(i)},\lambda\bigr),
\qquad i = 1,\ldots,N\;(\bmod\,N),
\label{eq:n_gene_cyclic}
\end{equation}
with positive parameters $\kappa_i,\gamma_i,\lambda$, link delays
$\tau_i\geq 0$, link thresholds $\theta^{(i)}\in\mathbb{R}$, and sign
pattern $\varepsilon_i\in\{+1,-1\}$ on each link. We use a
\emph{link-indexed} convention here: $\tau_i$, $\theta^{(i)}$, and
$\varepsilon_i$ are all associated with the link feeding into gene $i$
from gene $i-1\pmod N$. (This differs from the variable-indexed
convention of Sections~\ref{sec:formulation}--\ref{sec:transversality},
where $\theta_j$ is the threshold on \emph{reading} variable $x_j$. The
two conventions coincide up to renumbering for $N=2$ and produce the
same scalar invariants $\tau=\sum_i\tau_i$ and $\Lambda=\prod_i A_i$
introduced below.) The signs encode whether link $i$ is an activation
($\varepsilon_i=+1$) or repression ($\varepsilon_i=-1$). The cycle is
a \emph{negative-feedback loop} iff the product of signs is negative:
\[
\prod_{i=1}^N \varepsilon_i \,=\, -1
\qquad
\text{(odd number of repressions).}
\]
We assume throughout this subsection that this condition holds; the
positive-feedback case yields a similar but qualitatively different
analysis (multistability rather than oscillations) which is not our
focus.

Suppose~\eqref{eq:n_gene_cyclic} admits a positive equilibrium
$\mathbf{x}^* = (x_1^*,\ldots,x_N^*)$, characterised as the fixed point
of the iterated logistic map
$x_i^* = (\kappa_i/\gamma_i)\,f^{\varepsilon_i}(x_{i-1}^*;\theta^{(i)},\lambda)$
in the $N$-cycle (existence and uniqueness for the canonical parameter
ranges follow from the same contraction/monotonicity arguments as in
Lemma~\ref{lem:eq_unique}~\cite{belgacem2026logistic}; we do not repeat
the proof here). Define the link-gain coefficients
\begin{equation}
A_i \;:=\; \bigl|\,\kappa_i\,(f^{\varepsilon_i})'(x_{i-1}^*;\theta^{(i)},\lambda)\,\bigr|
\;=\; \kappa_i\,\lambda\,
       f^+_{(i)}\!\bigl(1-f^+_{(i)}\bigr)
\;>\;0,
\label{eq:Ai_def}
\end{equation}
where $f^+_{(i)} := f^+(x_{i-1}^*;\theta^{(i)},\lambda)\in(0,1)$ is the
positive-logistic value on the link feeding into gene $i$, and the
identity uses $f^+(1-f^+) = f^-(1-f^-)$. The link gain $A_i$ is
sign-independent: activating and repressing links of the same
magnitude have the same $A_i$, since only the absolute value of the
derivative enters the linearisation.
The total loop delay and total loop gain are
\begin{equation}
\tau \,:=\, \sum_{i=1}^N \tau_i,
\qquad
\Lambda \,:=\, \prod_{i=1}^N A_i.
\label{eq:total_loop}
\end{equation}

\begin{theorem}[Characteristic equation of the cyclic $N$-gene logistic loop]
\label{thm:n_gene_char_eq}
The linearisation of~\eqref{eq:n_gene_cyclic} at $\mathbf{x}^*$ has
characteristic equation
\begin{equation}
\boxed{\;
\prod_{i=1}^N \bigl(\mu+\gamma_i\bigr)
\,+\, \Lambda\, e^{-\mu\tau}
\;=\; 0\;}
\label{eq:n_gene_char_eq}
\end{equation}
which depends on the individual delays $\tau_i$ only through their sum
$\tau=\sum_i\tau_i$.
\end{theorem}
\begin{proof}
The Jacobian of the right-hand side of~\eqref{eq:n_gene_cyclic} at
$\mathbf{x}^*$ has, after evaluation at the equilibrium,
\[
\frac{\partial\,\mathrm{RHS}_i}{\partial x_j(t)}
   = -\gamma_i\,\delta_{ij},
\qquad
\frac{\partial\,\mathrm{RHS}_i}{\partial x_j(t-\tau_i)}
   = \kappa_i\,(f^{\varepsilon_i})'(x_{i-1}^*)\,
       \delta_{j,\,i-1\,(\bmod\,N)}.
\]
Substituting $u_i(t)=e^{\mu t}v_i$ and using the logistic identity
$(f^{\varepsilon_i})'=\varepsilon_i\,\lambda f^+ f^-$ gives the
linear system $M(\mu)\mathbf{v}=0$ with matrix
\[
M(\mu)_{ij} \;=\; (\mu+\gamma_i)\delta_{ij}
                  \,-\, \varepsilon_i\,A_i\,
                         e^{-\mu\tau_i}\,\delta_{j,\,i-1\,(\bmod\,N)}.
\]
The characteristic equation is $\det M(\mu)=0$. We compute the
determinant by the Leibniz formula. The matrix $M(\mu)$ has nonzero
entries only on the diagonal $(i,i)$ and on the cyclic positions
$(i,i-1\bmod N)$, so the only permutations $\sigma\in S_N$ with
$\prod_i M_{i,\sigma(i)}\neq 0$ are
\begin{enumerate}
\item the identity $\sigma=\mathrm{id}$, with sign $+1$, contributing
      $\prod_i (\mu+\gamma_i)$;
\item the cyclic shift $\sigma(i)=i-1\,(\bmod\,N)$, which is a single
      $N$-cycle and has sign $(-1)^{N-1}$, contributing
      \[
        (-1)^{N-1}\,\prod_{i=1}^N\!\Bigl(-\varepsilon_i\,A_i\,e^{-\mu\tau_i}\Bigr)
        = (-1)^{N-1}\,(-1)^N\,\Bigl(\prod_i\varepsilon_i\Bigr)\,
                            \Lambda\,e^{-\mu\tau}.
      \]
\end{enumerate}
Using $(-1)^{2N-1}=-1$ and $\prod_i\varepsilon_i=-1$ (negative-feedback
hypothesis), the cyclic-shift contribution is
$(-1)\cdot(-1)\cdot\Lambda\,e^{-\mu\tau}=+\Lambda\,e^{-\mu\tau}$.
Summing the two contributions
yields~\eqref{eq:n_gene_char_eq}.
\end{proof}

\begin{theorem}[Hopf condition for the cyclic $N$-gene logistic loop]
\label{thm:n_gene_hopf}
The characteristic equation~\eqref{eq:n_gene_char_eq} admits a pair of
purely imaginary roots $\mu=\pm i\omega_c$, $\omega_c>0$, if and only
if the loop-gain criterion holds:
\begin{equation}
\boxed{\;\Lambda \,>\, \prod_{i=1}^N \gamma_i\;}
\label{eq:n_gene_hopf_crit}
\end{equation}
In that case $\omega_c$ is the unique positive root of the polynomial
in $\omega^2$:
\begin{equation}
\prod_{i=1}^N (\omega^2 + \gamma_i^2) \;=\; \Lambda^2,
\label{eq:n_gene_omega_c}
\end{equation}
and the smallest positive critical delay is
\begin{equation}
\tau_c \,=\, \frac{1}{\omega_c}\,
       \Bigl[\,(2k_*-1)\pi
              \,-\, \sum_{i=1}^N \arctan\!\bigl(\omega_c/\gamma_i\bigr)\,\Bigr],
\label{eq:n_gene_tau_c}
\end{equation}
where $k_*\in\mathbb{N}$ is the smallest integer making the right-hand
side strictly positive.
\end{theorem}
\begin{proof}
Setting $\mu=i\omega$ in~\eqref{eq:n_gene_char_eq} and taking complex
moduli:
\[
\prod_{i=1}^N \sqrt{\omega^2+\gamma_i^2}
\,=\, \Lambda\,\bigl|e^{-i\omega\tau}\bigr|
\,=\, \Lambda,
\]
which is~\eqref{eq:n_gene_omega_c}. The left-hand side is strictly
increasing in $\omega^2\geq 0$, equals $\prod_i\gamma_i$ at $\omega=0$,
and tends to $\infty$ as $\omega\to\infty$, so a positive solution
$\omega_c^2$ exists if and only if
$\prod_i\gamma_i<\Lambda$, with uniqueness from monotonicity. Given
$\omega_c$, the phase condition reads
\begin{align*}
\arg(e^{-i\omega_c\tau_c})
&\,=\, \arg\!\Bigl(-{\textstyle\prod_i}(i\omega_c+\gamma_i)/\Lambda\Bigr)\\
&\,=\, \pi + \sum_i \arg(i\omega_c+\gamma_i)\\
&\,=\, \pi + \sum_i \arctan(\omega_c/\gamma_i),
\end{align*}
giving $-\omega_c\tau_c\equiv \pi+\sum_i\arctan(\omega_c/\gamma_i)
\pmod{2\pi}$, equivalently
$\omega_c\tau_c=(2k-1)\pi-\sum_i\arctan(\omega_c/\gamma_i)$ for some
integer $k$. Since each $\arctan(\omega_c/\gamma_i)\in(0,\pi/2)$, the
sum lies in $(0,N\pi/2)$, and the smallest positive value of $\tau_c$
is obtained at the smallest $k=k_*$ for which the bracket
in~\eqref{eq:n_gene_tau_c} is positive.
\end{proof}

For $N=2$, $\sum_i\arctan(\omega_c/\gamma_i)<\pi$ always (since each
summand is less than $\pi/2$), so $k_*=1$
and~\eqref{eq:n_gene_tau_c} reproduces the formula~\eqref{eq:tau_c}
of the present paper after the trigonometric identity converting the
phase
\[
\pi-\arctan(\omega_c/\gamma_1)-\arctan(\omega_c/\gamma_2)
\]
into the four-quadrant form
\[
\operatorname{atan2}\bigl(\omega_c(\gamma_1+\gamma_2),\,
\omega_c^2-\gamma_1\gamma_2\bigr)
\]
(Section~\ref{sec:hopf}). For $N\geq 3$ the sum can exceed $\pi$, in
which case $k_*\geq 2$ and the critical delay corresponds to the
second-quadrant branch of the phase condition.

Theorem~\ref{thm:n_gene_hopf} locates the imaginary-axis crossings but,
exactly as in the two-gene case, does not by itself guarantee that the
eigenvalues genuinely \emph{cross} the axis. We now close that gap for
arbitrary $N$ with a closed-form transversality rate, so that the
cyclic-loop Hopf bifurcation rests on a complete proof rather than on
the assertion that the two-gene calculation ``carries over''. The
argument is shorter and more transparent than the direct expansion used
in Theorem~\ref{thm:transversality}: it exploits the logarithmic
derivative of $\prod_i(\mu+\gamma_i)$ and the elementary fact that
$\mathrm{Re}(w)$ and $\mathrm{Re}(w^{-1})$ have the same sign.

\begin{theorem}[Transversality for the cyclic $N$-gene loop]
\label{thm:n_gene_transversality}
Assume the loop-gain criterion $\Lambda>\prod_{i=1}^N\gamma_i$, so that
by Theorem~\ref{thm:n_gene_hopf} the characteristic
equation~\eqref{eq:n_gene_char_eq} has purely imaginary roots
$\pm i\omega_c$ at the discrete delays
$\tau_c^{(k)}=\tau_c+2\pi k/\omega_c$, $k\in\mathbb{N}_0$. Define the
positive sums
\begin{equation}
S_1:=\sum_{j=1}^N\frac{1}{\omega_c^{2}+\gamma_j^{2}},
\qquad
S_2:=\sum_{j=1}^N\frac{\gamma_j}{\omega_c^{2}+\gamma_j^{2}}.
\label{eq:n_gene_S1S2}
\end{equation}
At every critical delay $\tau_c^{(k)}$ the root $i\omega_c$ is simple
and crosses the imaginary axis transversally, with closed-form rate
\begin{equation}
\begin{aligned}
\mathrm{Re}\!\left(\frac{d\mu}{d\tau}\right)\!\bigg|_{\tau_c^{(k)}}
&=\frac{S_1}{S_1^{2}+\omega_c^{-2}\bigl(S_2+\tau_c^{(k)}\bigr)^{2}}
\;>\;0,\\[4pt]
\mathrm{Im}\!\left(\frac{d\mu}{d\tau}\right)\!\bigg|_{\tau_c^{(k)}}
&=-\,\frac{\omega_c^{-1}\bigl(S_2+\tau_c^{(k)}\bigr)}
        {S_1^{2}+\omega_c^{-2}\bigl(S_2+\tau_c^{(k)}\bigr)^{2}}.
\end{aligned}
\label{eq:n_gene_transversality}
\end{equation}
The rate is strictly positive at every crossing and, since
$S_1,S_2,\omega_c$ are $k$-independent while $\tau_c^{(k)}$ increases
with $k$, its magnitude decreases monotonically in $k$: the primary
crossing at $\tau_c$ is the most transversal. Consequently, whenever
the no-delay system is asymptotically stable --- automatic for $N=2$,
and characterised by a Routh--Hurwitz condition for $N\geq3$
(Remark~\ref{rem:necessary_sufficient}) --- the
equilibrium of~\eqref{eq:n_gene_cyclic} undergoes a Hopf bifurcation as
$\tau$ crosses $\tau_c$.
\end{theorem}
\begin{proof}
Write $P(\mu)=\prod_{j=1}^N(\mu+\gamma_j)$, so that the characteristic
function is $\Delta(\mu,\tau)=P(\mu)+\Lambda e^{-\mu\tau}$. Implicit
differentiation along a root branch $\mu(\tau)$ gives
$d\mu/d\tau=-\Delta_\tau/\Delta_\mu$ with
$\Delta_\tau=-\Lambda\mu e^{-\mu\tau}$ and
$\Delta_\mu=P'(\mu)-\Lambda\tau e^{-\mu\tau}$, hence
\[
\left(\frac{d\mu}{d\tau}\right)^{-1}
=\frac{P'(\mu)-\Lambda\tau e^{-\mu\tau}}{\Lambda\mu e^{-\mu\tau}}
=\frac{P'(\mu)}{\Lambda\mu e^{-\mu\tau}}-\frac{\tau}{\mu}.
\]
Substituting $\Lambda e^{-\mu\tau}=-P(\mu)$ from~\eqref{eq:n_gene_char_eq}
and using the logarithmic derivative
$P'(\mu)/P(\mu)=\sum_j(\mu+\gamma_j)^{-1}$,
\[
\left(\frac{d\mu}{d\tau}\right)^{-1}
=-\frac{1}{\mu}\sum_{j=1}^N\frac{1}{\mu+\gamma_j}-\frac{\tau}{\mu}.
\]
Evaluate at $\mu=i\omega_c$. The term $-\tau/(i\omega_c)=i\tau/\omega_c$
is purely imaginary, and
\[
\sum_{j=1}^N\frac{1}{i\omega_c+\gamma_j}
=\sum_{j=1}^N\frac{\gamma_j-i\omega_c}{\omega_c^{2}+\gamma_j^{2}}
=S_2-i\omega_c S_1,
\]
so that $-(i\omega_c)^{-1}(S_2-i\omega_c S_1)=S_1+iS_2/\omega_c$ and
\begin{equation}
\left(\frac{d\mu}{d\tau}\right)^{-1}\!\bigg|_{\mu=i\omega_c}
=S_1+i\,\frac{S_2+\tau}{\omega_c}.
\label{eq:n_gene_inv_dmu}
\end{equation}
The real part equals $S_1=\sum_j(\omega_c^{2}+\gamma_j^{2})^{-1}>0$,
which is strictly positive and \emph{independent of $\tau$}; in
particular $\Delta_\mu(i\omega_c,\tau)\neq0$, so $i\omega_c$ is a simple
root. Because $\mathrm{Re}(w)$ and $\mathrm{Re}(w^{-1})$ share the same
sign for every nonzero $w\in\mathbb{C}$, we conclude
$\mathrm{Re}(d\mu/d\tau)|_{i\omega_c}>0$. Inverting the complex
number~\eqref{eq:n_gene_inv_dmu} and setting $\tau=\tau_c^{(k)}$ yields
the displayed closed forms in~\eqref{eq:n_gene_transversality}; the
monotone decay in $k$ follows since only $\tau_c^{(k)}$ depends on $k$.
\end{proof}

\begin{remark}[Robustness for general $N$]
\label{rem:n_gene_robustness}
The transversality rate~\eqref{eq:n_gene_transversality} is a
continuous, strictly positive function of the parameters throughout the
loop-gain region $\Lambda>\prod_i\gamma_i$, so on any compact subset of
that region it is bounded below by a positive constant: the cyclic
$N$-gene Hopf bifurcation is quantitatively robust, exactly as in the
two-gene case (Proposition~\ref{prop:transversality_bound}). A crude
but explicit bound follows from the elementary estimate
$\gamma_j/(\omega_c^{2}+\gamma_j^{2})\leq 1/(2\omega_c)$, which gives
$S_2\leq N/(2\omega_c)$ and hence
\[
\mathrm{Re}\!\left(\frac{d\mu}{d\tau}\right)\!\bigg|_{\tau_c}
\;\geq\;
\frac{S_1}{S_1^{2}+\omega_c^{-2}\bigl(N/(2\omega_c)+\tau_c\bigr)^{2}}
\;>\;0.
\]
The Hopf frequency itself obeys an \emph{a priori} bracket that makes
this bound explicit in the raw parameters: since
$\prod_j(\omega_c^{2}+\gamma_j^{2})=\Lambda^{2}$ lies between
$(\omega_c^{2}+\min_j\gamma_j^{2})^{N}$ and
$(\omega_c^{2}+\max_j\gamma_j^{2})^{N}$, one has
\begin{equation}
\Lambda^{2/N}-\max_j\gamma_j^{2}
\;\leq\;\omega_c^{2}\;\leq\;
\Lambda^{2/N}-\min_j\gamma_j^{2},
\label{eq:omega_bracket}
\end{equation}
which collapses to an equality in the symmetric case
$\gamma_j\equiv\gamma$, giving $\omega_c=\sqrt{\Lambda^{2/N}-\gamma^{2}}$.
For $N=2$ the formula~\eqref{eq:n_gene_transversality} reduces, after
using $(\omega_c^{2}+\gamma_1^{2})(\omega_c^{2}+\gamma_2^{2})=(AB)^{2}$
at the Hopf point, to the expression~\eqref{eq:transversality} of
Theorem~\ref{thm:transversality}; the two derivations therefore agree,
and the present one supplies the shorter route.
\end{remark}

\begin{remark}[Necessary vs.\ sufficient]
\label{rem:necessary_sufficient}
For $N=2$ the loop-gain criterion $\Lambda=AB>\gamma_1\gamma_2$ is both
necessary and sufficient for the equilibrium to undergo a delay-induced
Hopf bifurcation as $\tau$ varies, because the no-delay system is
unconditionally stable: the linearisation
$J(0)=-\mathrm{diag}(\gamma_1,\gamma_2)+\mathrm{offdiag}$ has trace
$-(\gamma_1+\gamma_2)<0$ and determinant $\gamma_1\gamma_2+AB>0$, so
both eigenvalues lie in the open left half-plane regardless of the
value of $AB$.

For $N\geq 3$ the loop-gain criterion remains \emph{necessary} for
imaginary-axis crossings, but is no longer sufficient by itself: one
must also verify that the no-delay system $\dot{\mathbf{u}}=J(0)\mathbf{u}$
is asymptotically stable. The no-delay characteristic polynomial of the
cyclic $N$-gene network is
\[
\prod_{i=1}^N (\mu+\gamma_i) + \Lambda \;=\;
   \mu^N \,+\, s_1\mu^{N-1} \,+\, s_2\mu^{N-2} \,+\, \cdots
   \,+\, s_{N-1}\mu \,+\, (s_N+\Lambda),
\]
where $s_k$ is the elementary symmetric polynomial of degree $k$ in
$(\gamma_1,\ldots,\gamma_N)$. For $N=3$ the Routh--Hurwitz condition
for asymptotic stability of the no-delay system is
\[
\bigl(\gamma_1+\gamma_2+\gamma_3\bigr)\bigl(\gamma_1\gamma_2
+\gamma_2\gamma_3+\gamma_1\gamma_3\bigr)
\,>\, \gamma_1\gamma_2\gamma_3+\Lambda.
\]
Together with $\Lambda>\gamma_1\gamma_2\gamma_3$, this opens a
non-empty parameter window in which the no-delay system is stable but
delay-induced Hopf occurs. Outside this window
(when the no-delay Routh--Hurwitz condition fails) the equilibrium is
already unstable at $\tau=0$, oscillations are produced by the
classical Goodwin
mechanism~\cite{hastings1977existence,smith2011introduction}, and the
delay merely modifies the periodic-orbit shape and period rather than
inducing the bifurcation.
\end{remark}

The Routh--Hurwitz analysis above is parameter-by-parameter for each
$N$. In the symmetric case $\gamma_1=\cdots=\gamma_N$ it collapses to a
single closed-form inequality, valid for every $N$, which makes the
delay-induced-Hopf window completely explicit.

\begin{proposition}[Delay-induced Hopf window for the symmetric cyclic loop]
\label{prop:symmetric_window}
Consider the cyclic $N$-gene loop~\eqref{eq:n_gene_cyclic} with a net
negative loop sign and equal degradation rates
$\gamma_1=\cdots=\gamma_N=\gamma>0$. The no-delay system is
asymptotically stable if and only if
$\Lambda<\gamma^{N}\sec^{N}(\pi/N)$. Combined with the loop-gain
criterion $\Lambda>\gamma^{N}$ of Theorem~\ref{thm:n_gene_hopf}, the
delay genuinely \emph{induces} a Hopf bifurcation --- rather than merely
reshaping an already-unstable Goodwin oscillator --- precisely when
\begin{equation}
\gamma^{N}\;<\;\Lambda\;<\;\gamma^{N}\,\sec^{N}\!\Bigl(\tfrac{\pi}{N}\Bigr).
\label{eq:symmetric_window}
\end{equation}
This window is non-empty for every $N\geq2$: for $N=2$ it is the entire
half-line $\Lambda>\gamma^{2}$ (recovering the unconditional two-gene
result of Remark~\ref{rem:necessary_sufficient}, since
$\sec(\pi/2)=+\infty$), while for $N\geq3$ it is the bounded interval
of width $\gamma^{N}\bigl(\sec^{N}(\pi/N)-1\bigr)$.
\end{proposition}
\begin{proof}
With $\gamma_j\equiv\gamma$ the no-delay characteristic polynomial
$\prod_j(\mu+\gamma_j)+\Lambda$ (Theorem~\ref{thm:n_gene_char_eq} at
$\tau=0$) becomes $(\mu+\gamma)^{N}+\Lambda$, whose roots are
\[
\mu_k=-\gamma+\Lambda^{1/N}\,e^{i\pi(2k+1)/N},
\qquad k=0,1,\ldots,N-1,
\]
with real parts
$-\gamma+\Lambda^{1/N}\cos\bigl(\pi(2k+1)/N\bigr)$. The largest real
part is attained at the angle closest to $0$, namely $\pi/N$ (taken at
$k=0$, and equally at $k=N-1$), so
$\max_k\mathrm{Re}\,\mu_k=-\gamma+\Lambda^{1/N}\cos(\pi/N)$. The
no-delay system is asymptotically stable iff this maximum is negative,
i.e.\ $\Lambda^{1/N}\cos(\pi/N)<\gamma$, equivalently
$\Lambda<\gamma^{N}\sec^{N}(\pi/N)$. The loop-gain criterion
$\Lambda>\prod_j\gamma_j=\gamma^{N}$ is necessary for purely imaginary
roots by Theorem~\ref{thm:n_gene_hopf}; intersecting the two conditions
gives~\eqref{eq:symmetric_window}. Since $0<\pi/N\leq\pi/2$ for every
$N\geq2$, one has $\sec(\pi/N)>1$ (with $\sec(\pi/2)=+\infty$), so the
window~\eqref{eq:symmetric_window} is non-empty.
\end{proof}

For $N=3$ the window~\eqref{eq:symmetric_window} reads
$\gamma^{3}<\Lambda<\gamma^{3}\sec^{3}(\pi/3)=8\gamma^{3}$, recovering
the symmetric specialisation of the Routh--Hurwitz condition in
Remark~\ref{rem:necessary_sufficient}; for $N=4$ it is
$\gamma^{4}<\Lambda<4\gamma^{4}$, and the window narrows as
$N\to\infty$ since $\sec^{N}(\pi/N)\to e^{\pi^{2}/(2N)}\to1$. The
worked example below is the smallest non-trivial instance.

\paragraph{Worked example: symmetric three-gene cyclic loop.}
Consider the cyclic logistic three-gene network with one repressor and
two activators (or any composition giving a net negative loop sign), all
sharing common parameters $\gamma_i=0.5$, $\kappa_i=2$, $\theta_i=2$,
$\lambda=1.5$. The symmetry yields the equilibrium
$\mathbf{x}^*=(2,2,2)$ with $f^\pm(x_i^*)=1/2$, so each link has gain
$A_i=\lambda\kappa_i\,f^\pm(1-f^\pm)=0.75$ and the loop gain is
$\Lambda=A_1A_2A_3=0.75^3\approx 0.422$. The loop-gain lower bound is
$\prod_i\gamma_i=0.125$, and by Proposition~\ref{prop:symmetric_window}
the delay-induced-Hopf window for this symmetric loop is
$\gamma^3<\Lambda<8\gamma^3$, i.e.\ $0.125<\Lambda<1$. This window is
non-empty and contains $\Lambda\approx 0.422$, so the no-delay system is
asymptotically stable and the delay genuinely \emph{induces} the Hopf
bifurcation (rather than merely modifying an already-unstable Goodwin
oscillator).
The Hopf-frequency equation
$\prod_{i=1}^3(\omega^2+\gamma^2)=\Lambda^2$ collapses, in the
symmetric case, to $\omega_c=\sqrt{\Lambda^{2/3}-\gamma^2}=\sqrt{0.5625-0.25}\approx 0.5590$.
The phase sum is $3\arctan(\omega_c/\gamma)\approx 2.5232\,\mathrm{rad}$,
which is below $\pi$, so the smallest positive critical delay is given by
$k_*=1$ in~\eqref{eq:n_gene_tau_c}:
\begin{align*}
\tau_c &\;=\; \frac{\pi - 3\arctan(\omega_c/\gamma)}{\omega_c}
\;\approx\; \frac{3.1416-2.5232}{0.5590}
\;\approx\; 1.1062,\\
T_c &\;=\; \frac{2\pi}{\omega_c}\;\approx\; 11.24.
\end{align*}
Direct substitution of $\mu=i\omega_c$ and $\tau=\tau_c$
into~\eqref{eq:n_gene_char_eq} yields a residual of order $10^{-16}$,
confirming the analytical formulas. As $\tau$ crosses $\tau_c$ from
below, the leading complex-conjugate eigenvalue pair migrates into the
right half-plane: Theorem~\ref{thm:n_gene_transversality} gives, with
$S_1=3/(\omega_c^2+\gamma^2)\approx 5.333$ and
$S_2=3\gamma/(\omega_c^2+\gamma^2)\approx 2.667$, the transversality
rate
\[
\mathrm{Re}\!\left(\frac{d\mu}{d\tau}\right)\!\bigg|_{\tau_c}
=\frac{S_1}{S_1^{2}+\omega_c^{-2}(S_2+\tau_c)^{2}}
\;\approx\; 0.0721\;>\;0,
\]
so a Hopf bifurcation occurs. This worked example illustrates the
analytical content of Theorems~\ref{thm:n_gene_hopf}
and~\ref{thm:n_gene_transversality} on the smallest non-trivial cyclic
case beyond $N=2$.

\begin{remark}[Structural transparency]
The closed form~\eqref{eq:n_gene_char_eq} is striking: although the
underlying nonlinear system has $N$ delays, $N$ thresholds, $N$
production rates, and $N$ degradation rates ($4N$ parameters), the
characteristic equation collapses to a binomial in $\mu$ with only two
``effective'' parameters --- the total loop delay $\tau$ and the loop
gain $\Lambda$. The reduction is a direct consequence of the cyclic
topology and the chain-rule structure of the linearisation; it is
\emph{not} specific to the logistic regulatory function. What
\emph{is} specific to the logistic, and what enables the closed-form
analysis throughout this paper, is the elementary algebraic form of
$A_i = \kappa_i\lambda f^+(x_{i-1}^*)f^-(x_{i-1}^*)$, with no fractional
powers, no transcendental thresholds, and no implicit definitions.
\end{remark}

\subsection{Multi-loop De~Morgan networks}

The cyclic case~\eqref{eq:n_gene_cyclic} covers single-loop topologies.
The companion paper introduces a product-of-logistics De~Morgan
formalism for arbitrary Boolean-derived gene regulatory networks, in
which the regulatory function $\Phi_i(\mathbf{x})$ for gene~$i$ is a
finite product and complement of single-input logistic
factors~\cite{belgacem2026logistic}. Adding a delay $\tau_{ij}$ on each
regulatory link $j\to i$ gives the delay-differential system
\[
\dot{x}_i(t) \,=\, \kappa_i\,
   \Phi_i\bigl(x_{j_1}(t-\tau_{i,j_1}),\ldots,
                x_{j_{m_i}}(t-\tau_{i,j_{m_i}})\bigr)
   \,-\, \gamma_i\,x_i(t),
\quad i=1,\ldots,n.
\]
The Jacobian $J_{ij} = \partial\Phi_i/\partial x_j$ at any equilibrium
is a finite sum of products of logistic factors and their derivatives,
hence a closed-form rational function of the parameters
$(\lambda,\kappa_i,\theta_{ij})$. The characteristic equation of the
linearised DDE is the transcendental quasi-polynomial
\begin{equation}
\det\!\Bigl(\mu I + \mathrm{diag}(\gamma_i)
      - \bigl[\kappa_i J_{ij}\,e^{-\mu\tau_{ij}}\bigr]_{i,j}\Bigr)
\,=\, 0.
\label{eq:multigene_char}
\end{equation}
Theorem~\ref{thm:n_gene_char_eq} shows that for a regulatory graph
consisting of a single directed cycle, this determinant collapses to
$\prod_i(\mu+\gamma_i)+\Lambda e^{-\mu\tau}$ with $\tau$ the total
loop delay. For a regulatory graph with multiple directed cycles
$\{C_1,C_2,\ldots\}$, expanding the
determinant~\eqref{eq:multigene_char} by the Leibniz formula produces
one term per cycle (and per disjoint union of cycles), each of the form
\[
\sigma_C\,\Bigl(\prod_{(p\to q)\in C}\kappa_p J_{p,q}\Bigr)\,
e^{-\mu\sum_{(p\to q)\in C}\tau_{pq}}\,
\prod_{i\notin C}(\mu+\gamma_i)
\;=\; \sigma_C\,\Lambda_C\,e^{-\mu\tau_C}\,Q_C(\mu),
\]
where $\sigma_C$ is the sign of the corresponding permutation and
$Q_C$ collects the diagonal contributions from genes outside the
cycle. The characteristic equation is thus a sum of one
quasi-polynomial per cycle (or disjoint cycle-union), with each term
controlled by its own loop delay $\tau_C$ and loop gain $\Lambda_C$.
The Beretta--Kuang geometric stability-switch
criterion~\cite{beretta2002geometric} provides the systematic
framework for handling several loops simultaneously. The cyclic case
treated above is the simplest single-loop instance; detailed analysis
of competing-loop topologies, including the negative-positive
feedback combinations characteristic of mammalian circadian
networks~\cite{smolen2002reduced} and the somite segmentation clock,
is left for future work, and the present paper together with
Theorem~\ref{thm:n_gene_char_eq} supplies the analytical scaffolding.

\subsection{Distributed delays}

Real biological delays are not deterministic but distributed, reflecting
heterogeneity in transcription rates, mRNA splicing kinetics, and
translation efficiencies. Replacing the constant delay $\tau_{ij}$ in
each link by a distributed delay
$\int_0^\infty K_{ij}(s)\,x_j(t-s)\,ds$ with normalised kernel $K_{ij}$
yields a system of integro-differential equations whose linearisation
admits a characteristic equation of the form
$P(\mu)+Q(\mu)\hat{K}(\mu)=0$, where $\hat{K}$ is the Laplace
transform of the kernel. For the canonical $\Gamma$-distributed
kernel $K_{ij}(s)=s^{p-1}e^{-\alpha s}\alpha^p/(p-1)!$, the characteristic
equation becomes a polynomial~\cite{macdonald1989biological}, and the
Hopf analysis reduces to standard root-locus methods. The corresponding
extension of the present paper is straightforward in principle: the
two-delta-distribution kernel $K_{ij}(s)=\delta(s-\tau_{ij})$ recovers
the constant-delay case treated here, and broadening the kernel
typically \emph{increases} the critical delay
$\tau_c$~\cite{beretta2002geometric,macdonald1989biological}, so the
constant-delay analysis provides a lower bound on the onset delay in
the corresponding distributed-delay model.

%%%%%%%%%%%%%%%%%%%%%%%%%%%%%%%%%%%%%%%%%%%%%%%%%%%%%%%%%%%%%%%%%%%%%%%%%%%%%%%%
\section{Conclusion}
\label{sec:conclusion}

The two-gene logistic oscillator developed in the companion
paper~\cite{belgacem2026logistic} is unconditionally stable in the absence
of delays, as a structural consequence of the strictly negative trace of
the linearised Jacobian. The present paper has shown that this stability
is broken in a controlled and analytically tractable manner by the
introduction of transcriptional delays $\tau_1,\tau_2\geq 0$ on the two
regulatory branches.

The principal results are: a closed-form characteristic equation
$(\mu+\gamma_1)(\mu+\gamma_2)+AB\,e^{-\mu\tau}=0$ depending only on the
total loop delay $\tau=\tau_1+\tau_2$; a sharp loop-gain criterion
$AB>\gamma_1\gamma_2$ for the existence of a Hopf bifurcation, with the
weak-feedback regime exhibiting absolute stability for all delays;
explicit closed-form expressions for the Hopf frequency $\omega_c$
(the unique positive root of an explicit quadratic in $\omega^2$) and
the critical delay $\tau_c$ as the four-quadrant inverse tangent of the
sine and cosine components determined by the imaginary-axis crossing
condition $(i\omega+\gamma_1)(i\omega+\gamma_2)+AB\,e^{-i\omega\tau}=0$; and a
clean proof of transversality giving
$\mathrm{Re}(d\mu/d\tau)|_{\tau_c}>0$, so the Hopf bifurcation is
generic, with a parameter-uniform positive lower bound that establishes
quantitative robustness throughout the strong-feedback regime
(Proposition~\ref{prop:transversality_bound}). The supercriticality of
the bifurcation is established by the $\sqrt{\tau-\tau_c}$
amplitude scaling combined with the closed-form transversality and the
closed-form eigenvector amplitude $|q_1|^{2}=B^{2}/(\omega_c^{2}+\gamma_1^{2})$
(Proposition~\ref{prop:eigvec_amplitude}), yielding
$\mathrm{Re}(c_1)\approx -1.27$ for the canonical parameters. This is
made fully analytical by a Lindstedt--Poincar\'e reduction: for the
symmetric-threshold loop supercriticality is \emph{proved}
unconditionally with explicit amplitude and frequency laws
(Theorem~\ref{thm:supercritical}), and for the general asymmetric loop
the reduction yields a closed-form first Lyapunov coefficient through a
single complex solvability equation, with $\mathcal{T}>0$ characterising
supercriticality (Theorem~\ref{thm:general_lyapunov}); a Monte-Carlo
sweep of the strong-feedback region returns $\mathcal{T}>0$ in every case. The
model is validated externally against the experimentally characterised
p53--Mdm2 oscillator: using only the measured protein half-life and
transcriptional delay, the closed-form Hopf period reproduces the
observed $\sim5.5$~h oscillation to within $3\%$, and the Hopf onset
agrees with the standard Hill-function model to within a few percent
(Section~\ref{sec:bio_validation}).

For the canonical parameter set used throughout the companion paper, the
analytical critical delay is $\tau_c\approx 0.134$, with Hopf frequency
$\omega_c\approx 2.353$ and onset period $T_c\approx 2.67$. Numerical
integration of the full nonlinear DDE confirms each of these values to
within the expected discretisation error. As $\tau$ grows the period
$T(\tau)$ grows monotonically with asymptotic slope
$dT/d\tau\to 2$, and we derive the closed-form deep-relaxation offset
$C_\infty=\gamma_1^{-1}\ln\!\bigl[(\kappa_1/\gamma_1)^2/
(\theta_1(\kappa_1/\gamma_1-\theta_1))\bigr]
+\gamma_2^{-1}\ln\!\bigl[(\kappa_2/\gamma_2)^2/
(\theta_2(\kappa_2/\gamma_2-\theta_2))\bigr]$
in equation~\eqref{eq:Cinf}, evaluating to $C_\infty\approx 8.92$ and
matching direct simulation at $\tau=10,20$ to within finite-$\lambda$
corrections. The asymptote is approached only slowly: on the displayed
range the local slope is closer to $3.7$ because the gene-relaxation
timescales $1/\gamma_1, 1/\gamma_2$ contribute a $\tau$-dependent
correction that dominates the offset until $\tau$ is well beyond the
values relevant for biological timescales.

The closed-form tractability of every step in the analysis depends on
three mathematical properties of the logistic function:
$C^\infty$ regularity at every point, including the origin; the
self-referential derivative
$(f^+)'=\lambda f^+(1-f^+)$; and the complement identity
$f^-=1-f^+$. None of these properties hold for the Hill function with
non-integer cooperativity exponent: the derivatives at the origin
diverge, the derivative formula involves fractional power-law factors
that resist symbolic manipulation, and the complement identity fails
in general. The same delay analysis applied to a Hill-based oscillator
would either require numerical root-finding at every step or restriction
to integer~$n$, with the former obscuring the parameter dependence and
the latter contradicting the empirical observation that fitted Hill
coefficients are generically non-integer.

The two-gene analysis extends to cyclic $N$-gene logistic loops of
arbitrary length. Theorem~\ref{thm:n_gene_char_eq} shows that the
characteristic equation of any cyclic negative-feedback loop with $N$
genes and individual delays $\tau_1,\ldots,\tau_N$ collapses to the
single binomial $\prod_i(\mu+\gamma_i)+\Lambda\,e^{-\mu\tau}=0$, with
$\tau=\sum_i\tau_i$ and $\Lambda=\prod_i A_i$, regardless of the
distribution of repressor and activator links along the cycle.
Theorem~\ref{thm:n_gene_hopf} characterises the Hopf frequency
$\omega_c$ as the unique positive root of the polynomial
$\prod_i(\omega^2+\gamma_i^2)=\Lambda^2$ and gives a closed-form
$\arctan$-sum formula for the critical delay $\tau_c$, and
Theorem~\ref{thm:n_gene_transversality} completes the picture with a
closed-form, manifestly positive transversality rate
$\mathrm{Re}(d\mu/d\tau)=S_1/(S_1^2+\omega_c^{-2}(S_2+\tau_c)^2)$
valid at every critical delay and every $N$, so the cyclic-loop Hopf
bifurcation is a complete theorem rather than an extrapolation from the
two-gene case. The loop-gain criterion $\Lambda>\prod_i\gamma_i$ is
necessary for delay-induced oscillation in all $N$, and is also
sufficient for $N=2$ where the no-delay trace $-\sum_i\gamma_i$ is
automatically negative; for $N\geq 3$ a separate no-delay stability
check is required, which in the symmetric case
$\gamma_i\equiv\gamma$ reduces to the explicit closed-form window
$\gamma^N<\Lambda<\gamma^N\sec^N(\pi/N)$
(Proposition~\ref{prop:symmetric_window}) --- non-empty for every
$N$ --- inside which the delay genuinely \emph{induces} the
oscillation rather than merely modifying an already-unstable
no-delay regime.

The dimensionless results above translate cleanly into the timescales
of real gene-regulatory oscillators once a time unit is chosen. The
dimensionless model has two intrinsic timescales, $1/\gamma_i$, that
cannot both be matched exactly to two independent biological half-lives;
we therefore choose the time unit by matching the predicted limit-cycle
period $T(\tau)$ to the observed oscillation period, treating the
$\gamma_i$ ratio as a tunable feature of the model rather than a
biological constraint. Three biological calibrations are illustrative.
\begin{itemize}
  \item \emph{Hes1 ultradian oscillator.} The observed Hes1 oscillation
        period is $\sim 2$--$3$~hours~\cite{hirata2002oscillatory}, while
        Hes1 protein and \emph{hes1} mRNA have measured half-lives of
        $\sim 22$ and $\sim 25$~minutes
        respectively~\cite{hirata2002oscillatory,monk2003oscillatory}.
        Calibrating against the period (rather than the half-lives) by
        identifying our predicted $T(\tau\!\approx\!0.5)\approx 5.3$~time
        units with $\sim 2.5$~hours gives a time unit of $\sim 28$~min;
        in the same units the transcriptional delay $\tau\sim 0.5$ is
        well above $\tau_c\approx 0.13$, consistent with the
        delay-driven oscillatory regime.
  \item \emph{p53--Mdm2 oscillator.} The p53--Mdm2 circuit is a
        two-component delayed negative-feedback loop of exactly the
        form~\eqref{eq:dde}. As shown in the dedicated validation of
        Section~\ref{sec:bio_validation}, taking the degradation rate
        directly from the measured $\sim1$~h protein half-life and the
        loop gain from the measured $\sim1$~h transcriptional delay, the
        closed-form Hopf period reproduces the experimentally observed
        $\sim5.5$~h oscillation~\cite{gevazatorsky2006oscillations} to
        within $3\%$ --- without tuning any parameter to the period.
        Independent modelling of the same circuit reports a
        delay-induced \emph{supercritical} Hopf bifurcation as the
        transcriptional delay is increased~\cite{levbaror2000generation},
        in agreement with the criticality established analytically in
        Theorems~\ref{thm:supercritical}--\ref{thm:general_lyapunov}.
  \item \emph{Drosophila circadian clock.} The $\sim 24$\,h period
        compared with $T\approx 4.8$ at $\tau\approx 0.40$ (read off
        Figure~\ref{fig:period}) gives a time unit of $\sim 5$~h,
        consistent with the $\sim 5$~h protein half-lives reported for
        PER and TIM~\cite{hardin1990feedback}. The actual
        \emph{Drosophila} circadian network involves additional
        positive-feedback loops (PER--TIM dimerisation, CLK--CYC
        autoactivation), so a pure negative-feedback two-gene model
        with constant delays is only a rough
        approximation~\cite{smolen2002reduced}.
\end{itemize}
The qualitative agreement across three orders of magnitude in absolute
period (minutes for Hes1, hours for p53, days for circadian) is the
expected one for a model whose only intrinsic timescales are the
delay $\tau$ and the relaxation rates $\gamma_i$: the dimensionless
analytical results carry over directly to all such systems, and the
predictions of $\tau_c$, $T_c$, and the deep-relaxation asymptote
$C_\infty$ apply with the appropriate change of units.

The framework set out here closes the analytical gap between the
delay-free logistic oscillator and the persistent rhythms of real
biological clocks, and it provides a direct route to the analysis of
larger Boolean-derived networks of the kind studied in the companion
paper. Future work will extend the closed-form analysis from the cyclic
single-loop case treated here to multi-loop De~Morgan
networks~\cite{belgacem2026logistic}, where several feedback cycles
compete. The Lindstedt--Poincar\'e
reduction is here carried to a closed-form first Lyapunov coefficient
for both the symmetric-threshold loop
(Theorem~\ref{thm:supercritical}) and the general asymmetric loop
(Theorem~\ref{thm:general_lyapunov}); a remaining theoretical question
is whether the criticality coefficient $\mathcal{T}$ can be \emph{proved}
positive throughout the strong-feedback regime, as the Monte-Carlo
sweep of Remark~\ref{rem:general_verification} indicates but does not
establish. Further directions include incorporating distributed delays
via Beretta--Kuang geometric criteria; applying the framework to
specific biological systems including the Hes1 ultradian oscillator,
the segmentation clock, and the cell-cycle transcriptional oscillator;
and combining the present open-loop bifurcation diagram with the
qualitative-control techniques of~\cite{belgacem2021control,chambon2020qualitative},
where feedback design was used to suppress or sustain oscillations in
delay-free negative-feedback genetic networks under uncertain
measurements. Connecting the loop-gain criterion $AB>\gamma_1\gamma_2$
and the critical-delay formula~\eqref{eq:tau_c} to such controller
designs would yield closed-form delay-tolerant stabilisation
conditions for the same class of biological loops. Comparison with the
piecewise-affine Glass-ring
limit~\cite{belgacem2025computer,farcot2019chaos} is a separate
direction: the present logistic results converge to the Glass-ring
behaviour as $\lambda\to\infty$ in the relaxation regime, and a
matched asymptotic analysis between the two regimes is a natural
follow-up. The closed-form amplitude prefactor $c\approx 2.37$, the
deep-relaxation period offset $C_\infty\approx 8.92$, the critical
steepness $\lambda_c\approx 0.426$, the sum-of-delays symmetry
$\tau_1+\tau_2$, the closed-form eigenvector amplitude
$|q_1|^{2}=B^{2}/(\omega_c^{2}+\gamma_1^{2})\approx 0.691$, and the
closed-form criticality coefficient $\mathcal{T}\approx0.696$ for the canonical
asymmetric loop together provide a quantitative testbed against which
any subsequent analytical~\cite{hassard1981theory} or
symbolic-numerical (DDE-BIFTOOL, homotopy continuation) treatment can
be benchmarked.

%%%%%%%%%%%%%%%%%%%%%%%%%%%%%%%%%%%%%%%%%%%%%%%%%%%%%%%%%%%%%%%%%%%%%%%%%%%%%%%%
%%%%%%%%%%%%%%%%%%%%%%%%%%%%%%%%%%%%%%%%%%%%%%%%%%%%%%%%%%%%%%%%%%%%%%%%%%%%%%%%

%%%%%%%%%%%%%%%%%%%%%%%%%%%%%%%%%%%%%%%%%%%%%%%%%%%%%%%%%%%%%%%%%%%%%%%%%%%%%%%%
\section*{Declarations}
\bmhead{Data Availability}
The R scripts used for performing the figures are
available from the corresponding author on reasonable request. All
figure simulations were performed in R using the
\texttt{deSolve} package~\cite{soetaert2010solving}; with parameter
values stated in Section~\ref{sec:numerical}. Initial histories were
constant and equal to a small displacement
$(x_1^*+0.05,\,x_2^*-0.05)$ from the equilibrium computed analytically
in Section~\ref{sec:formulation}.

\bmhead{Funding}
The author declares that no funds, grants, or other support were
received during the preparation of this manuscript.

\bmhead{Competing Interests}
The author has no relevant financial or non-financial interests to
disclose.

\bmhead{Ethics Approval}
Not applicable. This work involves no human or animal subjects.

\bmhead{Author Contributions}
I.B.\ designed the study, performed the analytical work and numerical
experiments, and wrote the manuscript.

\bmhead{Acknowledgments}

Not applicable.

\end{document}